\documentclass[11pt]{article}
\usepackage{amsthm,amsmath,amsfonts,mathrsfs,amssymb,systeme,bm, bbm}
\usepackage[numbers,sort&compress]{natbib}
\usepackage{enumerate}
\usepackage[ruled,vlined]{algorithm2e}

\usepackage{mlmodern}
\usepackage{graphicx}
\usepackage{array, tabularx}
\usepackage[hidelinks]{hyperref}
\usepackage{algpseudocode}
\usepackage{cancel}
\usepackage{url}

\usepackage[margin=1in]{geometry}

\newcommand{\R}{\mathbb{R}}

\newcommand{\N}{\mathbb{N}}
\newcommand{\D}{\mathcal{D}} \newcommand{\C}{\mathcal{C}}

\newcommand{\abs}[1]{\left\lvert#1\right\rvert}
\newcommand{\T}{\prime} \newcommand{\uup}[1]{{(#1)}}
\newcommand{\e}[1]{e^{\uup{#1}}}

\newcommand{\Prob}{\mathbb{P}}
\newcommand{\E}{\mathbb{E}}
\newcommand{\aseq}{\overset{a.s.}{=}}
\newcommand{\asto}{\overset{a.s.}{\to}}
\newcommand{\J}{\mathcal{J}} \newcommand{\K}{\mathcal{K}} \renewcommand{\b}[1]{u}

\newcommand{\route}[1]{\phi_{#1}}

 \newcommand{\uu}{{\uup{r}}} 
 \newcommand{\II}[1]{\mathbb{I}{( #1 )}} \newcommand{\Ind}[1]{\mathbb{I}}  

\newcommand{\uupp}[1]{{(#1)}}

\newcommand{\scale}[3]{{#1}^{\uupp{r, {#2}}}_{#3}}
\newcommand{\scaleL}[3]{{#1}^{\uupp{*, {#2}}}_{#3}}

\newcommand{\indexa}[1]{\alpha_{#1}}
\newcommand{\indexb}[1]{\beta_{#1}}

\newcommand{\defi}{{:=}} 

\renewcommand{\tilde}{\widetilde}

\def\caP{\mathcal{P}}
\def\caM{\mathcal{M}}
\def\caS{\mathcal{S}}

\def\caE{\mathcal{E}}
\providecommand{\abs}[1]{\left\lvert#1\right\rvert}

\providecommand{\norm}[1]{\lVert#1\rVert}
\providecommand{\Norm}[1]{\left\lVert#1\right\rVert}

\newtheorem{theorem}{Theorem}
\newtheorem{lemma}{Lemma}
\newtheorem{corollary}{Corollary}
\newtheorem{proposition}{Proposition}
\newtheorem{assumption}{Assumption}
\newtheorem{remark}{Remark}

\theoremstyle{definition}

\numberwithin{equation}{section}

\title{Functional Limits of Generalized Jackson Networks in Multi-scale Heavy Traffic}

\author{Zhen-Qing Chen$^1$\\{\small \href{mailto:zqchen@uw.edu}{zqchen@uw.edu}} \and J. G. Dai$^2$\\ {\small \href{mailto:jd694@cornell.edu}{jd694@cornell.edu}} \and Jin Guang$^3$\\{\small \href{mailto:jin.guang@chicagobooth.edu}{jin.guang@chicagobooth.edu}}}

\date{{$^1$Departments of Mathematics,
University of Washington, Seattle, Washington\\$^2$School of Operations Research and Information Engineering, Cornell University, Ithaca, New York\\
$^3$Booth School of Business, University of Chicago, Chicago, Illinois}} 

\begin{document}

\maketitle

\begin{abstract}
    We investigate the functional limits of generalized Jackson networks in a multi-scale heavy traffic regime where stations approach full utilization at distinct, separated rates. Our main result shows that the appropriately scaled queue length processes converge weakly to a limit process whose coordinates are mutually independent. This finding reveals the underlying dynamic mechanism that explains the asymptotic independence previously observed only in stationary distributions. The specific form of the limit process is shown to depend on the initial conditions.
    {In this paper, we consider the matching-rate and the lowest-rate {initial} conditions. {Although} {the corresponding limit processes have different laws,} {they have the same product-form exponential limiting distribution on the positive orthant as $t\to\infty$.}}
    Moreover, we introduce and analyze a blockwise multi-scale heavy traffic regime. In this regime, the network's stations are partitioned into blocks, where stations in different blocks approach the heavy traffic at different rates, while stations within the same block share a common rate. We obtain the functional limits in this regime as well, showing that the limit process exhibits blockwise independence. 
\end{abstract}

\noindent\textbf{Keywords:} reflected Brownian motion, heavy traffic, {product-form limiting distribution}, strong approximation, Skorokhod reflection map

\section{Introduction}\label{sec:Intro}

\subsection{Motivation and multi-scale heavy traffic model} \label{sec: IntroA}

The study of {queueing} networks is a cornerstone of stochastic modeling, with wide-ranging applications in production, transportation, and inventory management, as well as in the design and operational efficiency of service systems like contact centers and hospitals. A foundational result in this area is the Jackson network, which, under the assumption of Poisson arrival processes and exponential service times, exhibits an attractive product-form stationary distribution for its queue lengths (\citet{Jack1957,Jack1963}). When this assumption is relaxed to allow for general interarrival and service time distributions, the network is referred to as a generalized Jackson network (GJN). In this more practical setting, the analytical tractability of the product-form solution is lost, and the stationary distribution of queue lengths generally eludes a closed-form expression. 

Consequently, extensive research has been devoted to the approximation of GJN stationary distributions, particularly under the heavy traffic condition where the system's resources are nearly fully utilized. As the traffic intensity at each station approaches unity at a common rate, \citet{Reim1984} proved that the scaled queue length vector process converges weakly to a semimartingale reflecting Brownian motion (SRBM), which has a simpler dynamic and was first proposed by \citet{HarrReim1981}. While theoretically profound, this connection presents its own practical difficulties, as obtaining the stationary distribution of an SRBM is itself a formidable task. Only some special cases, like satisfying the skew symmetric condition \citep{HarrWill1987}, are tractable.

A recent breakthrough by \citet{DaiGlynXu2023} developed a novel asymptotic approximation by introducing a \textit{multi-scale heavy traffic} regime, where different stations approach the heavy traffic at distinct, separated rates. Specifically, they considered a family of GJNs under the condition that
\begin{equation} \label{eq: delta}
    \lim_{r\to 0} \rho_i^\uu = 1 \quad  \text{ and } \quad \lim_{r\to 0} \frac{1-\rho_j^\uu}{1-\rho_i^\uu} = 0  \quad \text{ for $1\leq i < j \leq J$,}
\end{equation}
where $\rho_j^\uu$ is the traffic intensity at station $j$ for the $r$th network. Under a specific parameterization $1-\rho_j^\uu = r^j$, they established a remarkable product-form result: the scaled steady-state queue lengths converge to a vector of independent, exponentially distributed random variables:
\begin{equation} \label{eq: intro res Dai}
    \Big( \big(1-\rho_1^\uu\big)Z_1^\uu(\infty), \ldots, \big(1-\rho_J^\uu\big)Z_J^\uu(\infty) \Big) \Longrightarrow \big( E_1,  \ldots, E_J \big) \quad \text{as } r\to 0,
\end{equation}
where $Z_j^\uu(\infty)$ is the steady-state queue length at station $j$ in the $r$th network. This elegant result provides a new and tractable approximation for the GJN stationary distribution. However, as a result concerning the system's long-run equilibrium, it does not {provide information on the dynamics of the scaled GJNs, nor does it} illuminate the temporal dynamics that lead to this state, leaving {in particular} a crucial question unanswered: what is the underlying mechanism that causes the {steady-state} queue lengths to become asymptotically independent?

{To answer this question, we {study the} functional limits {of scaled GJNs}, also known as process-level convergence.} Specifically, we investigate GJNs under the general multi-scale heavy traffic condition \eqref{eq: delta}. We show that with an appropriate temporal and spatial scaling applied to each station, the queue length process converges weakly to a limit process $Z^*$ whose coordinate processes are mutually independent:
\begin{equation} \label{eq: intro res}
    \begin{aligned}
        &\hat{Z}^\uu \defi \Big\{ \Big( \big(1-\rho_1^\uu\big)Z_1^\uu\big(t/\big(1-\rho_1^\uu\big)^2\big), \ldots, \big(1-\rho_J^\uu\big)Z_J^\uu\big(t/\big(1-\rho_J^\uu\big)^{2}\big) \Big) ;~t\geq 0   \Big\} \\
        & \qquad \qquad \Longrightarrow Z^*\defi \Big\{ \big( Z^*_1(t), \ldots, Z^*_J(t) \big);~t\geq 0 \Big\} \quad \text{as } r\to 0,
    \end{aligned}
\end{equation}
provided the scaled initial state $\hat{Z}^\uu(0)$ also converges weakly to a random vector $\xi$. The emergence of the asymptotic independence at the process level is the key finding of this work. It reveals that the asymptotic independence observed by \citet{DaiGlynXu2023} is not an isolated phenomenon of stationary distributions but is rooted in a more fundamental decoupling of the system's dynamics, which is forced by the separated timescales as shown in Section \ref{sec: proof SRBM}.

The specific form of each coordinate process $Z^*_j$ in the limit depends critically on the asymptotic behavior of the initial condition. In this paper, we concentrate on two cases{, which we call the \textit{matching-rate} and the \textit{lowest-rate} initial conditions (Assumptions \ref{assmpt: initial} and \ref{assmpt: conventional})}. First, {under the matching-rate initial condition, each station starts with a queue length of the same order as its own spatial scale; that is, the scaled initial queue length $(1-\rho_j^\uu)Z_j^\uu(0)$ converges to a {nonnegative} limit $\xi_j$. In this case,} {each} $Z^*_j$ is {an independent} one-dimensional SRBM starting from $\xi_j$; this is a consequence of a more general result we establish in Theorem~\ref{thm:GJN}. {Intuitively, {under the matching-rate condition}{{, when the system is viewed}}  on station $j$'s spatial and temporal scales, the lighter-loaded stations start from negligible levels and remain essentially empty, while the heavier-loaded stations start from effectively infinite backlogs and never idle, so each station evolves as if in isolation from time zero. For example, when $J=2$ and $1-\rho_j^\uu=r^j$, the matching-rate condition requires initial queue lengths of magnitude $1/r$ at station 1 and of magnitude $1/r^2$ at station 2: on station 1's spatial scale $1/r$, the backlog of station 2 appears infinite, while on station 2's spatial scale $1/r^2$, the queue of station 1 appears empty.} As the stationary distribution of a one-dimensional SRBM is exponential, {our} result is consistent with the {{product-form}} steady-state convergence in \eqref{eq: intro res Dai}. Second, {under the lowest-rate initial condition,} the initial queue length at each station is of {the same} magnitude $1/(1-\rho_1^\uu)${, the scale of the least-loaded station,} including the zero length case{{: A} network that starts empty is {its} canonical example. On station $j$'s scale with $j\geq 2$, a heavier-loaded station now starts from a negligible level rather than an effectively infinite backlog, so it does idle, and its idleness feeds back into {the whole system} through the reflection term. As a result,} each $Z^*_j$ becomes a coordinate process of {a $(J-j+1)$}-dimensional SRBM. {Nevertheless, {we show in Proposition \ref{prop:Zstar-exp} that the limit processes have the same product-form exponential limiting distribution as $t\to\infty$ in both cases.}} These two cases represent two extremes among the possible initial conditions. We leave the investigation of other reasonable  initial conditions to future work.

We further extend our framework by introducing and analyzing a \textit{blockwise multi-scale heavy traffic} regime, {which generalizes the fully separated multi-scale setting considered in \citet{DaiGlynXu2023}.} In this regime, the network's stations are partitioned into $K$ distinct blocks, $A_1, \ldots, A_K$. Stations in different blocks approach the heavy traffic at different rates, while stations within the same block share a common rate. Formally, for any $1\leq i< j\leq J$,
\begin{equation*} 
    \lim_{r\to 0} \rho_{i}^\uu = 1, \quad \text{ and } \quad \lim_{r\to 0} \frac{1-\rho_{j}^\uu}{1-\rho_{i}^\uu} \in \begin{cases}
        (0,\infty) & \text{if $i$ and $j$ are in the same block}, \\
        \{0 \} & \text{if $i$ and $j$ are in different blocks}.
    \end{cases}
\end{equation*}
For example, a 6-station network with a traffic intensity vector $\rho^\uu=\mathbf{1}-(r, 2r, 3r^2, 2r^2, r^2,$ $2r^3)$ falls into this regime with three blocks: $A_1 = \{1, 2\}$ (idle rate of order $r$), $A_2 = \{3, 4, 5\}$ (idle rate of order $r^2$), and $A_3 = \{6\}$ (idle rate of order $r^3$). We extend our functional limit theorem to this more general setting and show that the limit process exhibits blockwise independence, as established in Theorem \ref{thm:GJN block}. This framework is powerful as it unifies both conventional heavy traffic (where all stations form a single block) and the fully multi-scale heavy traffic (where each station is its own block) as special cases.

{We summarize the main contributions of the paper as follows. First, we establish functional limits that provide a dynamic {perspective on} the steady-state product-form independence in \citet{DaiGlynXu2023}; a central ingredient is the asymptotic independence of Brownian motions observed under separated diffusion scalings. Second, we show that the initial condition is an essential part of the process-level approximation: different asymptotic initial conditions can lead to different limit processes, {but these processes have the same product-form limiting distribution as $t\to\infty$}. Third, we identify the scale-separation mechanism in the SRBM analysis. A key step is a block Gaussian elimination of the reflection matrix, which separates the effective one-dimensional reflected dynamics from the coordinates operating on other scales. Finally, we transfer the SRBM limits to GJNs through an asymptotic functional strong approximation that allows $r$-dependent, possibly unbounded initial queue lengths and {holds for all diffusion scalings}. We also extend the analysis to blockwise multi-scale heavy traffic.}

Our proof of the main result (Theorem \ref{thm:GJN}) consists of two major parts. The first stage reduces the problem's complexity by employing a \textit{functional strong approximation}, as presented in Proposition~\ref{prop: SA}. This allows us to construct a family of SRBMs, denoted $\tilde{Z}^\uu$, on the same probability space as the GJN processes $Z^\uu$ for each $r\in (0,1)$, such that the two processes become asymptotically indistinguishable as $r\to 0$. This crucial step effectively reduces the challenge of analyzing complex GJNs to the more tractable problem of analyzing the asymptotic behavior of their corresponding SRBMs.

The second, and more technically involved, stage is to prove that this family of scaled SRBMs converges weakly as $r\to 0$ to the process $Z^*$, whose coordinates are mutually independent. The analysis for this stage hinges on the \textit{Skorokhod reflection mapping}, the foundational mathematical tool for the study of SRBMs. For any  c\`adl\`ag input path $x$ in $\R^J$ starting from the nonnegative orthant $\mathbb{R}_+^J$, 
denoted as $x\in \D_+([0,\infty),\R^J)$,
the mapping produces a unique pair of  
c\`adl\`ag
 output paths, $(z, y)$ in $\R^J_+\times \R_+^J$, that satisfies the following conditions for a given reflection matrix $R$:
\begin{enumerate}
    \item[(i)] $z(t)=x(t)+R y(t) \text{ for all } t\geq 0$;
      \item[(ii)] for each $j\in \J$, $y_j$ is nondecreasing with $y_j(0)=0$; 
    \item[(iii)] for each $j\in \J$, $y_j$ only increases at times $t$ where $z_j(t) = 0$, i.e., $\int_0^\infty z_j(t) d y_j(t)=0$.
\end{enumerate}
The path $z = \Phi(x;R)$ is called the reflected path, and $y = \Psi(x;R)$ the regulator path. When $J = 1$, $z$ will not depend on the reflection matrix (scalar) $R$, we simply denote $z = \Phi(x)$. Intuitively, condition (iii) ensures that the regulator $y$ acts only to enforce $z(t)\in \mathbb{R}_+^J$ for all $t\geq 0$. When the path $z$ hits a boundary face $F_j=\{ z \in \mathbb{R}_+^J : z_j = 0 \}$, the regulator pushes it back into the interior in a direction given by the $j$th column of $R$. This is first established by \citet{HarrReim1981} for 
continuous path $x$ in $\R^J$ and  by \citet{John1983} for  c\`adl\`ag path $x$ in $\R^J$. 
 The Skorokhod reflection mapping $(\Phi, \Psi)$ is Lipschitz continuous under the uniform norm; that is, for any $T>0$ and any $x,x'\in \D_+([0,\infty),\R^J)$,
\begin{equation} \label{eq: Lipschitz}
    \begin{gathered}
        \norm{\Phi(x;R)-\Phi(x';R)}_{T} \leq \kappa(R)\norm{x-x'}_{T}, \text{ and }
        \norm{\Psi(x;R)-\Psi(x';R)}_{T} \leq \kappa(R)\norm{x-x'}_{T},
    \end{gathered}
\end{equation}
where {{$\kappa(R)\in (0,\infty)$ is a Lipschitz constant depending only on $R$}} and 
$\norm{x}_{T}$ is the uniform norm over $[0, T]$ defined as
\begin{equation} \label{eq: uniform norm}
    \norm{x}_{T} :=\sup_{0\leq t\leq T}\norm{x(t)}:=\sup_{0\leq t\leq T}\max_{1\leq k \leq J}\abs{x_k(t)}.
\end{equation}

An SRBM $\tilde{Z}$ is the process obtained by applying the Skorokhod reflection mapping to a driving $J$-dimensional Brownian motion $\tilde{X}$. Thus, $\tilde{Z} = \Phi(\tilde{X}; R)$, and its dynamics are described by the equation $\tilde{Z}(t) = \tilde{X}(t) + R \tilde{Y}(t)$, where $\tilde{Y} = \Psi(\tilde{X}; R)$ is the regulator process, often representing the cumulative idle time at each station in queueing applications.

Having defined the SRBM, the convergence of the scaled SRBM $\tilde{Z}^\uu$ like $\hat{Z}^\uu$ in \eqref{eq: intro res} proceeds in two steps. 
 First,    we show that each individual coordinate process has a weak limit and 
 characterize its distribution. To do this, we analyze the system from the perspective of a single station's timescale, say that of station $k$. This analysis reveals a critical \textit{scale separation} phenomenon, formalized in Proposition~\ref{prop: rk initial}: for each $i<k$ and $j>k$, 
 as $r\to 0$, 
\begin{equation*}
    \big(1-\rho_k^\uu \big)\tilde Z^\uu_i\big(t/\big(1-\rho_k^\uu \big)^2\big) \Longrightarrow 0  \quad \text{and} \quad \big(1-\rho_k^\uu \big)\tilde Y^\uu_j\big(t/\big(1-\rho_k^\uu \big)^2\big) \Longrightarrow 0.
\end{equation*}
From station $k$'s viewpoint, lighter-loaded stations ($i<k$) appear to be always empty, while heavier-loaded stations ($j>k$) appear to have no idle time. This scale separation effectively isolates the dynamics of station $k$ from the rest of the network. Consequently, its scaled process converges weakly to a one-dimensional SRBM $Z^*_k$. Furthermore,  
this limit  SRBM $Z^*_k $ for station $k$ is driven by a Brownian motion which is a functional of the weak limit  of the $k$th  scaled Brownian motion $\tilde X^{(k, r)}:= \{ (1-\rho_k^\uu) \tilde X(t/(1-\rho_k^\uu)^2);~t\geq 0  \} $ as $r\to 0$.

Second,  we establish the asymptotic independence of the components of the scaled SRBM process. This is a  central insight of our paper. 
It is a consequence of the conclusion in the first step above, the continuity of the Skorokhod reflection mapping and a simple but fundamental result established in Proposition \ref{proposition: joint BM} which asserts the asymptotic independence of $ \{ \tilde X^{(k, r)}; k=1, \cdots, J  \}$ as $r\to 0$.  This  explains the decoupling mechanism for the multi-scale GJNs, and completes the characterization of the limit process $Z^*= (Z^*_1, \cdots , Z^*_J)$.

The rest of the paper is organized as follows. Section \ref{sec: literature} reviews the relevant literature, and Section \ref{sec: notation} introduces some key notations. In Section~\ref{sec: GJN}, we formally define the multi-scale GJN model and present our main functional limit results for GJNs and SRBMs under a matching-rate initial condition. The proof for the GJN functional limit is given in Section \ref{sec: proof GJN}, which builds upon the asymptotic strong approximation and the SRBM functional limit; these two foundational results are established in Sections \ref{pf: SA} and~\ref{sec: proof SRBM}, respectively. Section \ref{sec: conventional} presents the functional limit results under a lowest-rate initial condition, with the proofs deferred to Appendix \ref{sec: conventional proof}. Finally, Section~\ref{sec: block} presents our analysis on the blockwise multi-scale heavy traffic regime, with the corresponding proofs provided in Appendix \ref{sec: block heavy traffic proof}. {Table~\ref{tab: roadmap} provides a compact roadmap of the functional limit results established in this paper, summarizing the regimes, the initial conditions, the corresponding theorems with their assumptions, and the {properties} of the limit processes.}

{
\begin{table}[ht]
\centering
\caption{Roadmap of the functional limit results.}
\label{tab: roadmap}
\small
\begin{tabularx}{\textwidth}{llp{4.9cm}>{\raggedright\arraybackslash}X}
\hline
Regime & Initial condition & Results (assumptions) & Limit process \\
\hline
Multi-scale & Matching-rate & GJN: Thm.~\ref{thm:GJN} (Assum.~\ref{assmpt: moment}--\ref{assmpt: initial});\newline SRBM: Thm.~\ref{prop: multi-SRBM} (Assum.~\ref{assmpt: multiscaling}, \ref{assmpt: initial SRBM}) & mutually independent {coordinate processes: each is a} one-dimensional SRBM {on {{$\mathbb R_+$}}} \\[2pt]
Multi-scale & Lowest-rate & GJN: Thm.~\ref{thm:GJN cv} (Assum.~\ref{assmpt: moment}, \ref{assmpt: multiscale}, \ref{assmpt: conventional});\newline SRBM: Thm.~\ref{prop: multi-SRBM conventional} (Assum.~\ref{assmpt: multiscaling}, \ref{assmpt: conventional SRBM}) & mutually independent coordinate {processes}: each is the first coordinate of a multi-dimensional SRBM \\[2pt]
Blockwise & Matching-rate & GJN: Thm.~\ref{thm:GJN block} (Assum.~\ref{assmpt: moment}, \ref{assmpt: blockscale}, \ref{assmpt: initial block});\newline SRBM: Thm.~\ref{thm:SRBM block} (Assum.~\ref{assmpt: blockscaling}, \ref{assmpt: initial SRBM block}) & mutually independent blocks: each is a multi-dimensional SRBM \\[2pt]
Blockwise & Lowest-rate & GJN: Thm.~\ref{thm:GJN block cv} (Assum.~\ref{assmpt: moment}, \ref{assmpt: conventional}, \ref{assmpt: blockscale});\newline SRBM: Thm.~\ref{thm:SRBM block cv} (Assum.~\ref{assmpt: conventional SRBM}, \ref{assmpt: blockscaling}) & mutually independent blocks: each is the first block of a multi-dimensional SRBM \\
\hline
\end{tabularx}
\end{table}
}

\subsection{Literature Review}
\label{sec: literature}

\subsubsection{Related work on generalized Jackson networks}

Queueing networks offer potent and practical tools for assessing the performance of complex systems, encompassing manufacturing, computer, and communication systems; see, for example, \citet{Mor2013} and \citet{SrikYing2014}. Despite the growing body of literature dedicated to the performance analysis of queueing networks, deriving an analytical expression for the steady-state performance of such networks remains challenging. Exact performance analysis has been constrained to a few highly specialized networks. As a consequence, researchers have proposed various approximation methods as alternatives. Notably, significant efforts concerning single and multi-class queueing networks have centered around {fluid models} for stability analysis, {diffusion approximations} for performance analysis, and exploring the interconnections between these methods and their objectives. In this subsection, we present a concise overview of these advancements within the context of GJNs.

To analyze the steady-state behavior of GJNs, a first task is to establish the conditions for stability, i.e., the existence of a stationary distribution. In this regard, \citet{Sigm1990} made a seminal contribution by demonstrating that the GJN exhibits a unique stationary distribution only when the traffic intensity at each station remains below unity. \citet{Dai1995} introduced a similar result through a notable linkage between the stochastic and fluid models of queueing networks; see also \citet{Stol1995}. Although the stability conditions in GJNs exhibit a straightforward form, the exact expression of the stationary distribution is generally intractable, except for Jackson networks where interarrival and service times are exponential.

Regarding diffusion approximations, \citet{Reim1984} demonstrated that, as the traffic intensity approaches unity at the same rate, the scaled queue length processes converge weakly to an SRBM, which was first introduced by \citet{HarrReim1981}. This approximation is supported by the functional central limit theorem, also known as the diffusion approximation or the heavy traffic limit theorem. In this context, time is linearly accelerated, and space is compressed by a square-root factor. However, process-level convergence does not necessarily guarantee steady-state convergence. Under mild moment-bound conditions for interarrival times and service times, \citet{GamaZeev2006} and \citet{BudhLee2009} further established that the stationary distributions of the scaled queue lengths in GJNs converge weakly to that of the associated SRBM. This convergence is achieved through the application of the {limit interchange argument} between the scaling parameter and the time horizon, a methodology that has been widely employed to study steady-state approximations in various queueing systems, such as the multi-class single server queue with feedback in \citet{Kats2010} and the limited processor sharing queue in \citet{ZhanZwar2008}. 

An alternative approach to the limit interchange argument is the {basic adjoint relationship} (BAR) approach, developed by \citet{Miya2017} and \citet{BravDaiMiya2017}. One significant advantage of the BAR approach is its direct characterization of the stationary distribution of queueing networks, bypassing the need to handle their transient dynamics. This approach has found application in multi-class queueing networks with static-buffer-priority service policies in \citet{BravDaiMiya2024}.

The recent study by \citet{DaiGlynXu2023} introduced a novel approximation method for GJNs, named multi-scale heavy traffic. In this regime, the traffic intensity of each station approaches unity at different rates, leading to an asymptotic product-form structure in the stationary distribution of the scaled queue length. This independence result is rather surprising, since neither the pre-limit stationary distribution is of product form, nor is the stationary distribution of the conventional heavy traffic limit. This regime has been further extended to multiclass queueing networks by \citet{DaiHuo2024} and to SRBMs by \citet{GuanChenDaiGlyn2024}. They all employed the BAR approach to characterize the stationary distribution of GJNs in this multi-scale heavy traffic regime. 

A similar asymptotic independence has been observed at the process level in specialized contexts. For instance, \citet{kriukov2025}, in a concurrent and independent study, established functional convergence for Gaussian-driven, feed-forward networks, showing the same phenomenon. However, their analysis is specific to a restrictive network topology that prohibits traffic merging, relies on the properties of Gaussian inputs, and focuses on the system being in stationarity from initial time 0.
 In contrast, our framework accommodates general initial conditions for generalized Jackson networks with arbitrary routing, revealing the critical role of the initial state in shaping the limit. We thus demonstrate that this asymptotic independence is a fundamental consequence of timescale separation in complex stochastic networks, not a feature of a particular model structure or a particular initial state of the system. {At a high level, the asymptotic independence of Brownian motions under separated scalings (Proposition~\ref{proposition: joint BM}) is a statement about Brownian motions alone, and the scale-separation analysis in Section~\ref{sec: proof SRBM} applies to any family of SRBMs with an $\caM$-matrix reflection matrix, regardless of the network that generates them. The ingredients specific to generalized Jackson networks are the asymptotic strong approximation (Proposition~\ref{prop: SA}), which couples the network with the SRBM family, and the $\caM$-matrix structure of the reflection matrix.}

\subsubsection{Related work on Semimartingale reflecting Brownian motions}

SRBMs have been extensively studied as diffusion approximations—also known as functional central limit theorems—for a wide range of stochastic networks operating under conventional heavy traffic conditions, where the traffic intensities of all stations converge to unity at the same rate. The concept of SRBMs was introduced by \citet{HarrReim1981} through the Skorokhod reflection mapping, providing a foundational framework for analyzing such systems. This framework was first applied to prove diffusion approximations for GJNs by \citet{Reim1984} and \citet{HarrWill1987}. 

Equipped with the Skorokhod reflection mapping, SRBMs have been applied to multiclass queueing networks under various service disciplines. Research in this area includes studies on multiclass single-station networks \citep{Reim1988,DaiKurt1995}, multiclass feedforward networks \citep{Pete1991,HarrNguy1993}, reentrant lines with priority disciplines \citep{ChenZhan1996, DaiYehZhou1997,BramDai2001}, multiclass queueing networks with first-in-first-out disciplines \citep{ChenZhan2000}, and other head-of-line disciplines \citep{Will1998, Bram1998}, among others. See, for example, \citet{Miya2015} for a comprehensive survey on diffusion approximations for queueing networks.

The necessary and sufficient condition for the existence of SRBMs in an orthant is that the reflection matrix $R$ is a completely-$\caS$ matrix—meaning every principal submatrix of \( R \) is an \( \mathcal{S} \)-matrix. A matrix \( R \) is called an \( \mathcal{S} \)-matrix if there exists a vector \( u \geq 0 \) such that \( Ru > 0 \). Under this condition, the SRBM is in fact unique in distribution \citep{Will1995}. However, the necessary and sufficient conditions for the stability (or positive recurrence) of general SRBMs remain unknown. \citet{BramDaiHarr2010} provided a necessary condition for an SRBM to possess a stationary distribution: 
\begin{equation} \label{eq1:necessary}
  \text{$R$ is nonsingular $\quad$ and  $\quad$  $\delta\defi-R^{-1}\theta>0$},
\end{equation}
where \( \theta \) is the drift vector.
If $R$ is an $\caM$-matrix (i.e., $R=sI-B$ where $s>0$ and $B$ is a nonnegative matrix with the spectral radius less than $s$), then this condition is also sufficient for the stability of the SRBM \citep{HarrWill1987}.
Additionally, \citet{BramDaiHarr2010} demonstrated that for the case when $d=2$, stability is achieved if and only if condition \eqref{eq1:necessary} holds and $R$ is a so-called $\caP$-matrix (that is, every principal minor is positive). However, condition \eqref{eq1:necessary} is found to be insufficient for stability in three and higher dimensions, even if $R$ is a $\mathcal{P}$-matrix. For the case when $d=3$, \citet{BramDaiHarr2010} further provided a necessary and sufficient condition for stability, albeit one involving a complex expression. Nevertheless, in four and higher dimensions, there is no established necessary and sufficient condition for the stability of SRBMs.

The stationary distribution of SRBMs is crucial for assessing the long-run performance of complex stochastic networks. However, as indicated earlier, obtaining explicit forms of the stationary distribution is generally intractable, except in special cases such as those discussed by \citet{Fodd1983,HarrWill1987}. This challenge has led to the development of numerical methods by \citet{DaiHarr1992} and \citet{ShenChenDaiDai2002} to compute the stationary distribution of SRBMs. Effective simulation methods have been proposed by \citet{BlanChen2015,BlanChenSiGlyn2021} to estimate the stationary distributions of SRBMs.

\subsection{Notation}\label{sec: notation}
Let $\N$ denote the set of positive integers $\{1,2,\ldots\}$ and, for $J\in \N$, $\R^J$ denote the set of $J$-dimensional real-valued vectors. The symbol $\defi$ is used for definitions. For $a,b\in \R$, let $a\wedge b := \min\{a,b\}$ and $a\vee b:= \max\{a,b\}$. For a vector $a$ or a matrix $A$, we use the prime $a^\T$ or $A^\T$ to denote its transpose. For simplicity, we denote by $\mathbf{0}$ and $\mathbf{1}$ the zero vector and the all-ones vector, respectively, whose dimensions are determined by the context. Throughout the paper, all comparison operations for vectors and matrices are component-wise, and we define $\R_+^J \defi \{a\in \R^J\mid a\geq \mathbf{0}\}$.

Let $\C([0,\infty),\R^J)$ denote the space of continuous functions from $[0,\infty)$ into $\R^J$. We endow $\C([0,\infty),\R^J)$ with the topology of uniform convergence on compact time intervals.
A sequence $\{x^{(n)}\}_{n=1}^\infty$ in $\C([0,\infty),\R^J)$ converges to $x$ if and only if for each $T>0$, $ \norm{x^{(n)} - x}_T\to 0$ as $n\to \infty$, where
$\norm{\cdot}_{T}$ is the uniform norm defined in \eqref{eq: uniform norm}. 
We write $x_i$ for the $i$th component of $x$ for $i\in \J \defi \{1,\ldots,J\}$.

Let $\D([0,\infty),\R^J)$ denote the space of functions from $[0,\infty)$ into $\R^J$ that are right-continuous and have finite left limits. 
We endow $\D([0,\infty),\R^J)$ with the (Skorokhod) $J_1$-topology, which is induced by a metric $m_{J_1}$ on $\D([0,\infty),\R^J)$. 
For a function $x$ in $\C([0,\infty),\R^J)$, a sequence $\{x^{(n)}\}_{n=1}^\infty$ in $\D([0,\infty),\R^J)$ converges to $x$ in the $J_1$-topology if and only if $\{x^{(n)}\}_{n=1}^\infty$ converges to $x$ uniformly on any compact time interval of $[0,\infty)$ (u.o.c.). See Theorems 1.14 and 1.15 in Chapter VI of \citet{jacod2013limit} for a precise definition of the $J_1$-topology and a proof of this fact. Since the limits in this paper are all continuous functions, this u.o.c.~characterization of the $J_1$-topology is sufficient for our purposes.

For two random elements $X$ and $Y$, we use $X\aseq Y$ and $X \overset{d}{=}Y$ to denote that $X$ equals $Y$ almost surely and in distribution, respectively. For two functions $f$ and $g$, we use $f(r) = O(g(r))$ as $r\to 0$ to denote $\limsup_{r\to 0} |f(r)/g(r)| < \infty$ and $f(r) = o(g(r))$ as $r\to 0$ to denote $\lim_{r\to 0} f(r)/g(r) = 0$. We use  $\Longrightarrow$ to  denote weak convergence of random variables or random processes. {For $x^\uu\in\R_+^d$, we write $x^\uu\to\infty$ if $\min_{1\leq i\leq d}x_i^\uu\to\infty$. For an $\R_+^d$-valued random vector $X^\uu$, we write $X^\uu\Longrightarrow\infty$ if $\min_{1\leq i\leq d}X_i^\uu\to\infty$ in probability.} 
For a set $A$, $\II {A}$ denotes the indicator function of $A$.

\section{Main Results} \label{sec: GJN}

\subsection{Model setting and assumptions}\label{sss:setting}

We consider an open generalized Jackson network (GJN) with $J$ service stations. Each station consists of a single server and an infinite-capacity buffer. Each station receives jobs from an external arrival source (possibly null) and service completions from some other stations. When an arriving job at a station finds the server busy, it waits in the buffer. Jobs at each station are processed following the first-come-first-serve (FCFS) discipline. Upon completion of a job at a station, the job is either routed to another station or exits the network. All jobs visiting a station are homogeneous and share the same service time distribution and routing probabilities.

To formalize the dynamics of the network, we utilize the following model, consistent with \citet{DaiGlynXu2023}, which incorporates stochastic interarrival and service times, routing mechanisms, and initial conditions for each station. For each station $j\in \J:=\{1,\ldots,J\}$, there are two i.i.d.~sequences of random variables, $\{T_{e,j}(i);i\in \N\}$ and $\{T_{s,j}(i);i\in \N\}$, two parameters $\alpha_j\ge 0$ and $\mu_j>0$, an i.i.d.~sequence of random vectors $\{\route{j}(i); i\in\N\}$, and a nonnegative random variable $Z_j(0)$, all defined on a common probability space $(\Omega, \mathscr{F}, \mathbb{P})$.
We assume that the following $4J$ sequences are mutually independent:
\begin{equation}  \label{eq: 4J sequence}
    \{T_{e,j}(i);i\in \mathbb{N}\}_{j=1}^J,\quad \{T_{s,j}(i);i\in \mathbb{N}\}_{j=1}^J,\quad \{\route{j}(i); i\in\N\}_{j=1}^J, \quad \{Z_j(0);j\in \J\}
\end{equation}
and that the first $2J$ sequences are unitized: $\E[T_{e,j}{(1)}]=1$ and $\E[T_{s,j}{(1)}]=1$ for all $j\in \J$. 

The value of $\alpha_j$ denotes the external arrival rate to station $j$. Let $\caE = \{j : \alpha_j > 0\}$ be the set of indices for those stations that have external arrivals. For each $i\in \N$ and $j\in \caE$, $T_{e,j}(i)/\alpha_j$ represents the interarrival time between the $i$th and $(i+1)$th externally arriving jobs to station $j$. The value of $\mu_j$ is the service rate at station $j$, and $T_{s,j}(i)/\mu_j$ denotes the service time of the $i$th job at station $j$. We impose the following moment assumption on the interarrival and service times.
\begin{assumption}[Moment condition] \label{assmpt: moment}
    There exists $\varepsilon\in (0,1)$ such that 
    \begin{equation} \label{eq: moment condition}
        \E\big[\big(T_{e,j}(1)\big)^{2+\varepsilon}\big]<\infty \quad \text{for all $j \in \caE$}, \quad \text{and} \quad \E\big[\big(T_{s,j}(1)\big)^{2+\varepsilon}\big]<\infty \quad \text{for all $j \in \J$}.
    \end{equation}
\end{assumption}
The random vector $\route{j}(i)$ represents the routing decision of the $i$th job at station $j$, based on the routing probability matrix $P$. Specifically, the job is routed to station $k$ if $\route{j}(i)=\e{k}$ with probability $P_{jk}$, and exits the network if $\route{j}(i)=\mathbf{0}$ with probability $P_{j0}:= 1-\sum_{k\in \J}P_{jk}$, where $\e{k}$ is the $J$-dimensional vector where the $k$th element is $1$ and all other elements are $0$. We assume that this network is open, ensuring all jobs eventually exit the network, and equivalently, $I-P$ is invertible, where $I$ is the identity matrix. For a given $t\ge 0$, let $Z(t)$ be a $J$-dimensional vector whose $j$th element $Z_j(t)$ is the queue length at station $j$ at time $t$, including the one possibly being processed. The random variable $Z_j(0)$ accounts for the initial number of jobs at station $j$.

\paragraph{Traffic equation.} Let $\{\lambda_{j}, j\in \J\}$ be the solution to the traffic equations:
\begin{equation}\label{eq:traffic}
    \lambda_{j}=\alpha_{j}+\sum_{i\in \J} \lambda_i P_{i j} \qquad \text{for each}~j\in \J.
\end{equation}
For each $j\in\J$, $\lambda_j$ is interpreted as the nominal total arrival rate to station $j$, which includes both external arrivals and internal arrivals from other stations. 
The traffic intensity at station $j$ is defined to be $\rho_j:={\lambda_j}/{\mu_j}$. 

\paragraph{Multi-scale heavy traffic.} Following \citet{DaiGlynXu2023},  we consider a sequence of GJNs indexed by $r \in(0,1)$, and denote by $Z^\uu=\{Z^\uu(t);~t\geq0\}$ the queue length process for the $r$th system. We denote the initial queue length and the service rate at station $j$ in the $r$th GJN by $Z^\uu_j(0)$ and $\mu_j^\uu$, respectively. These are the only model parameters dependent on $r$, while all other parameters, including the external arrival rates $\{\alpha_j\}_{j=1}^J$, unitized interarrival and service times, and routing vectors specified in \eqref{eq: 4J sequence}, are independent of $r$.

Since we focus on the multi-scale heavy traffic regime, we use $J$ functions $\gamma_1,\ldots,\gamma_J$ to describe the separated scales of the system. Throughout this paper, unless otherwise mentioned, we always assume the scale functions $(\gamma_1,\ldots,\gamma_J)$ satisfy the following conditions:
 they are $J$ strictly increasing functions of $r$ on the unit interval $(0,1)$ such that
    \begin{equation*}         \lim_{r\to 0}\gamma_j(r)= 0\quad \text{and} \quad\lim_{r\to 0} {\gamma_j(r)}/{\gamma_i(r)} = 0  \quad \text{ for $1\leq i < j \leq J$},
    \end{equation*}
    and
    {
    \begin{equation}\label{eq: tech}
       \lim_{r\to 0}
       \gamma_1^{(2-\varepsilon)/(2+\varepsilon)}(r)
       \sqrt{\log\log(1/\gamma_1(r))}
       \log\log(1/\gamma_J(r))
       =0,
    \end{equation}
    where $\varepsilon$ is as in Assumption~\ref{assmpt: moment}.
    }

The technical assumption \eqref{eq: tech} is only used in the proof of Proposition~\ref{prop: SA} when applying Lemma~\ref{lemma: SA_BM}. It is very mild and is satisfied by most common families of increasing  functions. For example, it holds for any polynomial functions $\gamma_j(r)=r^{\beta_j}$ with $\beta_1<\beta_2<\ldots<\beta_J$, which includes the specific case $\gamma_j(r) = r^{j}$ used in \citet{DaiGlynXu2023}. It only excludes extremely fast-decaying functions with respect to $\gamma_1(r)$, such as those of the form $\gamma_J(r)=\exp(-\exp(1/\gamma_1(r)))$.

\begin{assumption}[Multi-scale heavy traffic] \label{assmpt: multiscale}
    The service rates
    \begin{equation*}
        \mu_j^\uu  = \lambda_j + \gamma_j(r) \quad \text{for all $j\in \J$}.
    \end{equation*}
\end{assumption}
Under Assumption \ref{assmpt: multiscale}, $\rho^{(r)}_j:= \lambda_j/\mu^{(r) }_j \to 1$ as $r\to 0$ for all $j\in \J$. This implies that each station is in heavy traffic when $r\to 0$, but different stations approach heavy traffic at distinct, separated rates. 

As mentioned in the introduction, the convergence of the initial states plays a crucial role in determining the limit process. We first consider the case where the initial queue length at each station is of the same order as {{the reciprocal of}} its corresponding idle rate $1-\rho_j^\uu$.

\begin{assumption}[Matching-rate initial condition]  \label{assmpt: initial}
    There exists an $\R_+^J$-valued   random vector $\xi$ {{whose components are mutually independent}} such that
    \begin{equation} \label{eq: initA}
        \Big( \gamma_1(r) Z^\uu_1(0),~\gamma_2(r) Z^\uu_2(0),~\ldots~,\gamma_J(r)Z^\uu_J(0) \Big)  \Longrightarrow \xi \quad \text{as } r\to 0,
    \end{equation} 
    and for all $j=2,3,\ldots,J$,
    \begin{equation} \label{eq: initB}
        \gamma_{j-1}(r) Z^\uu_j(0) \Longrightarrow \infty \quad \text{as } r\to 0.
    \end{equation} 
\end{assumption}

Condition \eqref{eq: initA} is necessary for the weak convergence of the scaled 
queue length process  $\hat{Z}^\uu $ defined in \eqref{eq: intro res}.
 Condition \eqref{eq: initB} in Assumption \ref{assmpt: initial} is automatically satisfied if each component of $\xi$ in \eqref{eq: initA} is strictly positive. Condition \eqref{eq: initB} states that, relative to station $j-1$, the initial queue length at station $j$ is asymptotically infinitely large. This ensures that queue lengths of all stations are of widely separated magnitudes and close to their appropriate scales at the initial time and then this separation persists for all subsequent times $t>0$. This condition is crucial for the limit to be a process of independent one-dimensional SRBMs.

To illustrate that the distribution of the limit process depends critically on the asymptotic behavior of the initial condition of the system,
  we will study in Section \ref{sec: conventional}  a case when condition \eqref{eq: initA} holds but \eqref{eq: initB} fails. The functional limits of multi-scale GJNs under all other initial conditions when \eqref{eq: initA} holds but \eqref{eq: initB} fails will be studied in a subsequent paper.

\subsection{Generalized Jackson networks in multi-scale heavy traffic} \label{sss:results}

In this section, we present our main functional limit result for the queue length process $Z^\uu$ of GJNs in the multi-scale heavy traffic regime. 

\begin{theorem}\label{thm:GJN}
    Suppose Assumptions \ref{assmpt: moment}-\ref{assmpt: initial} hold. Then
    \begin{equation} \label{eq: lim GJN}
        \left\{ \left( \gamma_1(r)Z_1^\uu\big(t/\gamma_1^2(r)\big),  \ldots, \gamma_J(r) Z_J^\uu\big(t/\gamma_J^2(r)\big) \right);~t\geq 0 \right\} \Longrightarrow Z^* \quad \text{as } r\to 0,
    \end{equation}
    where $Z^*=\{Z^*(t);~t\geq 0\}$ is a $J$-dimensional diffusion process whose components are mutually independent one-dimensional SRBMs.
    Specifically, each coordinate process $Z^*_j$ is a one-dimensional SRBM with initial state $\xi_j$, drift $\theta_j = -(1-w_{jj})<0$ and variance $\sigma^2_j$
    \begin{equation*}
        \begin{aligned}
            \sigma_j^2&=  \sum_{i<j} \alpha_i\left(w_{i j}^2 c_{e, i}^2+w_{i j}\left(1-w_{i j}\right)\right)+\alpha_j c_{e, j}^2+\sum_{i>j} \lambda_i\left(w_{i j}^2 c_{s, i}^2+w_{i j}\left(1-w_{i j}\right)\right) \\ 
            &\quad +\lambda_j\left(c_{s, j}^2\left(1-w_{j j}\right)^2+w_{j j}\left(1-w_{j j}\right)\right),
        \end{aligned}
    \end{equation*}
    where the matrix $(w_{ij})$ is defined below in \eqref{eq: w}.
\end{theorem}

Our result is consistent with \citet{DaiGlynXu2023} under a specific parameterization $\gamma_j(r)=r^j$, as the stationary distribution of the limit $Z^*$ in \eqref{eq: lim GJN} coincides with the steady-state limit they derived.

We now define the $J \times J$ matrix $w=\left(w_{i j}\right)$ used in Theorem \ref{thm:GJN}. For each $j \in \mathcal{J}$,
\begin{equation} \label{eq: w}
    \begin{aligned}
        & \left(w_{1 j}, \ldots, w_{j-1, j}\right)^{\prime}=\left(I-P_{j-1}\right)^{-1} P_{[1: j-1], j} \\
        & \left(w_{j j}, \ldots, w_{J j}\right)^{\prime}=P_{[j: J], j}+P_{[j: J],[1: j-1]}\left(I-P_{j-1}\right)^{-1} P_{[1: j-1], j}.
        \end{aligned}
\end{equation}
Here for sets $A, B \subseteq \{1,\ldots,J\}$, we denote by $P_{A, B}$ the submatrix obtained from the matrix $P$ by deleting the rows not in $A$ and the columns not in $B$ with the following conventions: for $\ell \leq k$,
\begin{equation} \label{eq: submatrix}
    [\ell: k]=\{\ell, \ell+1, \ldots, k\}, \quad P_{A, k}=P_{A,\{k\}}, \quad P_{k, B}=P_{\{k\}, B}, \quad P_k=P_{[1: k],[1: k]}.
\end{equation}

Although the expressions in \eqref{eq: w} appear complex, the matrix $(w_{ij})$ has a clear probabilistic interpretation in GJNs as demonstrated by \citet{DaiGlynXu2023}. Recall that the routing matrix $P$ can be embedded within a transition matrix of a discrete-time Markov chain (DTMC) on state space $\{0\} \cup \{1,2,\ldots,J\}$ with 0 being the absorbing state. For each $i,j \in \{1,2,\ldots,J\}$, $w_{i j}$ is the probability that starting from state $i$, the DTMC will eventually visit state $j$ without first visiting any state in the set $\{0, j+1, \ldots, J\}$.

The following asymptotic functional strong approximation result relates the asymptotic behavior of GJNs to that of SRBMs. Hence, it reduces the study of asymptotic limit of scaled queue length process $\{\gamma_k(r) Z^\uu(t/\gamma_k^2(r));~t\geq 0\}$ to that of SRBMs $\{\gamma_k(r) \tilde Z^\uu(t/\gamma_k^2(r));~t\geq 0\}$, which have simpler dynamics with fewer parameters and are easier to analyze.

\begin{proposition}[Asymptotic functional strong approximation] \label{prop: SA}
    Suppose that Assumption~\ref{assmpt: moment} holds, $\mu^\uu_j>\lambda_j$ with $\mu^\uu_j - \lambda_j =O(\gamma_1(r))$ for all $j\in \J$. If the underlying probability space $(\Omega, \mathscr{F}, \mathbb{P})$ is rich enough, there exists a family of SRBMs $\tilde{Z}^\uu$ with $r\in (0,1)$ on this space, such that for any fixed $T>0$, we have
    \begin{equation*}         \sup_{0\leq t\leq T} \max_{k\in \J} \Norm{\gamma_k(r) Z^\uu\big(t/\gamma_k^2(r)\big)-\gamma_k(r) \tilde{Z}^\uu\big(t/\gamma_k^2(r)\big)}  \aseq o\left(\gamma_1^{\varepsilon/(2+\varepsilon)}(r)\right) \quad \text{as $r\to 0$},
    \end{equation*}
    where $\varepsilon>0$ is the parameter in the moment condition \eqref{eq: moment condition} of Assumption~\ref{assmpt: moment}. 
    
    The SRBM $\tilde{Z}^\uu$ has initial state $\tilde Z^\uu(0)\aseq Z^\uu(0)$, reflection matrix $R=I-P^\T$, drift vector $\theta^\uu=-R(\mu^\uu-\lambda)$, and covariance matrix $\Gamma$ given by
    \begin{equation} \label{eq: SA_Gamma}
        \Gamma_{jk} = \alpha_j c_{e, j}^2 \delta_{j k} + \sum_{i\in \J} \lambda_i \left[ P_{i j}\left(\delta_{j k}-P_{i k}\right) + c_{s,i}^2 (\delta_{ij}-P_{ij})(\delta_{ik}-P_{ik}) \right],
    \end{equation}
    where $c_{e,j}^2=\operatorname*{Var}(T_{e,j}(1))$ and $c_{s,j}^2=\operatorname*{Var}(T_{s,j}(1))$ are the squared coefficient of variation of the interarrival time $T_{e,j}$ and service time $T_{s,j}$ of station $j$, and $\delta_{j k}=1$ if $j=k$ and $0$ otherwise.
\end{proposition}

\begin{remark}
    The ``rich enough" condition is standard in strong approximation literature. It can always be satisfied by construction, for example, by applying the Skorokhod representation theorem (see, e.g., Remark~2.2.1 in \citet{Csor1981}) or by expanding the probability space.
\end{remark}

\citet{Horv1992} and \citet{ChenMand1994} established the functional strong approximation of a GJN assuming a fixed initial queue length. Proposition \ref{prop: SA} extends their result to a family of GJNs and does not require the initial queue lengths to be fixed or bounded. This makes our proposition applicable under various initial conditions, including both Assumption \ref{assmpt: initial} and Assumption \ref{assmpt: conventional}, which we will later introduce in Section \ref{sec: conventional}. {The primitive strong approximation estimates used in the proof are standard. The role of Proposition~\ref{prop: SA} is to assemble these estimates so that the comparison between the GJN and the approximating SRBM remains valid when the initial queue lengths depend on $r$, possibly diverge, and have different orders across stations, and when the processes are observed under the diffusion scalings used below. The GJN and the approximating SRBM are constructed with the same initial state, so the initial term does not enter the approximation error. However, the initial condition still affects the busy-time processes, and therefore affects the random time changes in the service and routing noise terms. The remaining work is to control the primitive arrival, service, and routing errors together with these random time changes.} The proof of Proposition \ref{prop: SA} will be given in Section \ref{pf: SA}.

\subsection{SRBMs in multi-scale heavy traffic} \label{sec: SRBM}

In this section, we first define a similar multi-scaling regime for SRBMs, and then establish the functional limit result for SRBMs in this regime.

We have defined the SRBM in Section \ref{sec: IntroA} by the Skorokhod reflection mapping. The data of an SRBM are an initial state $\tilde Z(0)$, a drift vector $\theta$, a positive definite covariance matrix $\Gamma$ and an $\caM$-reflection matrix $R$. 
We consider a sequence of SRBMs with data $(\tilde Z^\uu(0), \theta^\uu, \Gamma, R)$ indexed by $r \in(0,1)$, and define the $r$th SRBM $\tilde Z^\uu=\{\tilde Z^\uu(t);~t\geq0\}$ by 
\begin{equation} \label{eq: RBM def}
    \tilde Z^\uu(t) = \tilde X^\uu(t) + R \tilde Y^\uu(t)
\end{equation}
where 
\begin{equation} \label{eq: RBM def2}
    \tilde X^\uu(t) = \tilde Z^\uu(0) + \theta^\uu t + L W(t) \quad \text{and} \quad \tilde{Y}^\uu = \Psi(\tilde X^\uu; R).
\end{equation}
Here, the lower triangular matrix $L$ is defined by the Cholesky decomposition of the covariance matrix $\Gamma=L L^\T$, $W:=\{(W_1(t),W_2(t), \ldots, W_J(t));~t\geq 0\}$ is a $J$-dimensional standard Brownian motion and $\Psi$ is the Skorokhod reflection mapping defined in Section~\ref{sec: IntroA}. Recall that $W$ is said to be a $J$-dimensional standard Brownian motion if $\{W_j\}_{j\in \J}$ are independent and each $W_j$ is a one-dimensional Brownian motion with $W_j(0)=0$ and $\E[W_j^2(t)]=t$ for all $t\geq 0$. {We assume that $\tilde Z^\uu(0)$ is independent of $W$ for every $r\in(0,1)$.}

The initial state $\tilde Z^\uu(0)$ and the drift vector $\theta^\uu$ are the only model parameters dependent on $r$, while all other parameters, including the covariance matrix, reflection matrix, and standard Brownian motion, are independent of $r$.

We assume that the family of SRBMs satisfies the multi-scaling regime as follows.

\begin{assumption}[Multi-scaling regime for SRBMs] \label{assmpt: multiscaling}
        For all $r\in (0,1)$, 
    \begin{equation}\label{eq: drift epsilon}
        \theta^\uu= -R\delta^\uu \quad \text{with } \quad \delta^\uu = \left[ \gamma_1(r), \gamma_2(r), \ldots, \gamma_J(r) \right]^\T.
    \end{equation}
\end{assumption}

We impose an initial condition on the SRBMs that is analogous to Assumption \ref{assmpt: initial} for GJNs.
\begin{assumption}[Matching-rate initial condition for SRBMs]
    \label{assmpt: initial SRBM}
    There exists an $\R_+^J$-valued   random vector $\tilde\xi$ {{whose components are mutually independent}} such that
    \begin{equation*}
        \Big( \gamma_1(r) \tilde Z_1^\uu(0), \gamma_2(r) \tilde Z_2^\uu(0), \ldots, \gamma_J(r) \tilde Z_J^\uu(0) \Big)  \Longrightarrow \tilde\xi, \quad \text{as } r\to 0,
    \end{equation*}
    and for all $j=2,3,\ldots,J$,
    \begin{equation*}         \gamma_{j-1}(r) \tilde Z^\uu_j(0) \Longrightarrow \infty \quad \text{as } r\to 0.
    \end{equation*} 
\end{assumption}

The following theorem establishes the functional limit $\tilde{Z}^*$ of the multi-scaling SRBMs  as $r\to 0$, whose coordinate processes are mutually independent one-dimensional SRBMs on $\R_+$. The result is consistent with \citet{GuanChenDaiGlyn2024} in that the stationary distribution of our limit process $\tilde Z^*$ matches the multi-scaling limit of the stationary distribution they derived.

\begin{theorem}\label{prop: multi-SRBM}
    Suppose that Assumptions \ref{assmpt: multiscaling} and \ref{assmpt: initial SRBM}  hold. Then   as $ r\to 0, $ 
    \begin{equation*}
        \left\{ \left( \gamma_1(r) \tilde Z_1^\uu\big(t/\gamma_1^2(r)\big), \ldots, \gamma_J(r) \tilde Z_J^\uu\big(t/\gamma_J^2(r)\big) \right);~t\geq 0 \right\} \Longrightarrow \tilde Z^*  ,
    \end{equation*}
    where  $\tilde Z^*=\{\tilde Z^*(t);~t\geq 0\}$ is a $J$-dimensional diffusion process. Furthermore, all coordinate processes of $\tilde Z^*$ are mutually independent, and each coordinate process $\tilde Z_j^*$ is a one-dimensional SRBM with initial state $\tilde\xi_j$, drift $\tilde\theta_j= -(1-w_{jj})<0$ and variance
    \begin{equation} \label{eq: u}
        \tilde\sigma^2_j = {(u^{(j)})^\T} \Gamma {u^{(j)}} \quad \text{ with } {u^{(j)}}=(w_{1j},\ldots,w_{j-1,j},1,0,\ldots,0)
    \end{equation}
    where $w_{ij}$ is defined for each $i\leq j$ in $\J$ by
    \begin{equation*}
        \begin{aligned}
            & \left(w_{1 j}, \ldots, w_{j-1, j}\right)=-R_{j,[1: j-1]}R_{j-1}^{-1} \quad  \text{and} \quad w_{j j}=1-R_{jj}+R_{j,[1: j-1]}R_{j-1}^{-1}R_{[1: j-1],j}.
            \end{aligned}
    \end{equation*}
    When $R=I-P^\T$, $\{w_{ij}, 1\leq i \leq j \leq J\}$ is the same as those in the matrix $(w_{ij})$ in Theorem \ref{thm:GJN}.
\end{theorem}

\section{Convergence of GJNs} \label{sec: proof GJN}

Armed with Proposition \ref{prop: SA} and Theorem \ref{prop: multi-SRBM}, whose proofs are deferred to {Sections} \ref{pf: SA} and~\ref{sec: proof SRBM}, respectively, we now prove Theorem \ref{thm:GJN}.
We first show that under the condition of Theorem~\ref{thm:GJN}, the associated multi-scaling SRBMs $\tilde Z^\uu$ in Proposition \ref{prop: SA} converge weakly to  $Z^*$. We then complete the proof of Theorem \ref{thm:GJN} for the GJN queue length process by utilizing the asymptotic functional strong approximation in Proposition \ref{prop: SA}.

\begin{proof}[Proof of Theorem \ref{thm:GJN}.]
    Observe that the SRBM family $\tilde Z^\uu$ in Proposition \ref{prop: SA} inherits the necessary properties from the GJN assumptions. Specifically, Assumption \ref{assmpt: multiscale} on the GJN implies that Assumption \ref{assmpt: multiscaling} holds for the SRBMs. Similarly, Assumption \ref{assmpt: initial} implies that Assumption \ref{assmpt: initial SRBM} also holds with $\tilde \xi \overset{d}{=} \xi$. Therefore, Theorem~\ref{prop: multi-SRBM} implies that the multi-scaling of the associated SRBMs $\tilde Z^\uu$ converge weakly to a $J$-dimensional diffusion process $\tilde{Z}^*$:
    \begin{equation*}
        \left\{ \left( \gamma_1(r) \tilde Z_1^\uu\big(t/\gamma_1^2(r)\big),\ldots, \gamma_J(r) \tilde Z_J^\uu\big(t/\gamma_J^2(r)\big) \right);~t\geq 0 \right\} \Longrightarrow \tilde Z^* \quad \text{as } r\to 0,
    \end{equation*}
    where all coordinate processes of $\tilde Z^*$ are independent, and each coordinate process $\tilde Z_j^*$ is a one-dimensional SRBM with initial state $\xi_j$, drift $-(1-w_{jj})$ and variance $\tilde\sigma^2_j={(u^{(j)})^\T} \Gamma {u^{(j)}}$ with $\Gamma$ defined in \eqref{eq: SA_Gamma}. Since the initial states and drifts of $\tilde Z^*$ and $Z^*$ match, their equality in distribution hinges on their variances being equal. We verify that $\tilde\sigma^2_j=\sigma^2_j$ for all $j\in \J$ in Appendix~\ref{sec: proof GJN sigma}, which confirms that $\tilde Z^*\overset{d}{=}Z^*$. Hence, we have 
    \begin{equation*}
        \left\{ \left( \gamma_1(r) \tilde Z_1^\uu\big(t/\gamma_1^2(r)\big), \ldots, \gamma_J(r) \tilde Z_J^\uu\big(t/\gamma_J^2(r)\big) \right);~t\geq 0 \right\} \Longrightarrow Z^* \quad \text{as } r\to 0.
    \end{equation*}

    The asymptotic functional strong approximation in Proposition \ref{prop: SA} implies that 
    \begin{equation*}         \sup_{0\leq t\leq T} \max_{k\in \J} \abs{\gamma_k(r) Z^\uu_k\big(t/\gamma_k^2(r)\big)-\gamma_k(r) \tilde{Z}^\uu_k\big(t/\gamma_k^2(r)\big)}  \to 0 \quad \text{a.s.} \quad \text{as } r\to 0,
    \end{equation*}
    Applying the convergence-together theorem (e.g., Theorem 11.4.7 in \citet{Whit2002}) to the preceding two results yields
    \begin{equation*}
        \left\{ \left( \gamma_1(r) Z_1^\uu\big(t/\gamma_1^2(r)\big), \ldots, \gamma_J(r) Z_J^\uu\big(t/\gamma_J^2(r)\big) \right);~t\geq 0 \right\} \Longrightarrow Z^* \quad \text{as } r\to 0.
    \end{equation*}
    This completes the proof of Theorem \ref{thm:GJN}. 
\end{proof}

\section{Asymptotic Strong Approximation} \label{pf: SA}
In this section, we present the proof of Proposition~\ref{prop: SA}. We will first establish the system dynamics of the GJN and formulate this process by the Skorokhod reflection mapping in Section \ref{sec: system dynamics}. The proof of Proposition~\ref{prop: SA} is presented in Section \ref{sec: proof SA}. {The proof consists of three steps. First, using the Skorokhod representation of the queue-length process and the Lipschitz continuity \eqref{eq: Lipschitz} of the reflection mapping, we reduce the strong approximation of the queue-length process $Z^\uu$ to that of its noise component $V^\uu$, simultaneously under all diffusion scalings. Second, we construct the limiting Brownian motion $W^{(v)}$ explicitly from the independent Brownian motions that strongly approximate the primitive (arrival, service, and routing) processes, and decompose the resulting approximation error into five terms. Third, we bound the five {error} terms {{using}} Lemmas~\ref{lemma: SA_diffusion}--\ref{lemma: SA_BM}, whose roles are explained before their statements in Section~\ref{sec: proof SA}.}

\subsection{System dynamics of GJN} \label{sec: system dynamics}
From the primitive sequences \eqref{eq: 4J sequence} for a GJN, we can construct the queue-length process $Z=\{Z(t);~t\geq0\}$ as follows. We first construct three key processes for each station. For $j\in \caE$, the external arrival counting process to station $j$ is denoted by 
\begin{equation*}     A_j(t)\defi \max \Big\{ n \ge 0: \sum_{i=1}^n T_{e,j}(i)/\alpha_j \le t \Big\} \quad \text{for } t\geq 0,
\end{equation*}
which counts the number of jobs that arrive at station $j$ during the time interval $[0,t]$. When $j\notin \caE$, we set $A_j(t)=0$ for $t\geq 0$.
To separate the parameter $r$ from the underlying renewal process, we define the unit-rate service counting process at station $j$ as
\begin{equation*}
     {S}_j(t)\defi \max \Big\{ n \ge 0: \sum_{i=1}^n T_{s,j}(i)\le t \Big\}\quad \text{for } t\geq 0,
\end{equation*}
and hence, ${S}_j(\mu_j^\uu t) =\max \{ n \ge 0: \sum_{i=1}^n T_{s,j}(i)/\mu_j^\uu \le t \}$ counts the number of jobs that complete service at station $j$ when the server is busy for a total of $t$ time units. 
We also define the routing sequence at station $j$ as 
\begin{equation*}     \psi_j(0) = 0, \quad \psi_j(n) = \sum_{i=1}^{n} \route{j}(i), \quad n\ge 1,
\end{equation*}
whose $k$th element $\psi_{j,k}(n)$ is the number of jobs that are routed to station $k$ from station $j$ during the first $n$ service completions at station $j$.

We further define the busy-time process $B^\uu=\{B^\uu(t);~t\geq0\}$, where $B^\uu(t)$ is a $J$-dimensional vector whose $j$th element $B_j(t)$ is the total busy time of the server at station $j$ up to time $t$. It follows from their definitions that $  S_j(\mu_j^\uu B_j^\uu(t))$ is the number of service completions at station $j$ up to time $t$, and that $\psi_{j,k}(  S_j(\mu_j^\uu B_j^\uu(t)))$ is the number of jobs that are routed to station $k$ from station $j$ up to time $t$. This leads to the following system equation for the queue length process: for all $t \geq 0$ and $j\in \J$,
\begin{equation*}     Z_j^\uu(t)=Z_j^\uu(0)+A_j(t)+\sum_{i\in \J} \psi_{i,j}\left( {S}_i\left(\mu_i^\uu B_i^\uu(t)\right)\right)- {S}_j\left(\mu_j^\uu B_j^\uu(t)\right).
\end{equation*}
Furthermore, for a work-conserving system, the busy time is defined by
\begin{equation*}     B_j^\uu(t)=\int_0^t \II {Z_j^\uu(s)>0} ds.     
\end{equation*}

Following \citet{Reim1984}, we can formulate the queue-length process $Z^\uu$ by the Skorokhod reflection mapping with reflection matrix $R=I-P^\prime$. {Specifically, we first rewrite the system equation above as
\begin{equation*}
    Z^\uu(t) = X^\uu(t) + R Y^\uu(t), \quad t\geq 0,
\end{equation*}
where the free (unreflected) netput process $X^\uu$ and the regulator process $Y^\uu$ are given by}
\begin{equation*}
    \begin{aligned}
        X^\uu(t) &= Z^\uu(0) - R(\mu^\uu-\lambda)t + V^\uu(t),\\
        Y_j^\uu(t)&=  \mu_j^\uu \left( t- B_j^\uu(t) \right)=\mu_j^\uu \int_0^t \II{Z_j^\uu(s)=0} d s, \quad j\in \J,
    \end{aligned}
\end{equation*}
and the centered noise process is given by
\begin{equation*}
    \begin{aligned}
        V_j^\uu(t) &= \left[A_j(t)-\alpha_j t\right]  +\sum_{i\in\J} \left( P_{i j} - \delta_{i j} \right)\left[ {S}_i\left(\mu_i^\uu B_i^\uu(t)\right)-\mu_i^\uu  B_i^\uu(t)\right] \\
        &\quad +\sum_{i\in\J}\left[\psi_{i,j}\left( {S}_i\left(\mu_i^\uu B_i^\uu(t)\right)\right)-P_{i j}  {S}_i\left(\mu_i^\uu B_i^\uu(t)\right)\right], \quad j\in \J.
    \end{aligned}
\end{equation*}
{In this decomposition, the boundary regulator term $R Y^\uu$ keeps $Z^\uu$ in the nonnegative orthant $\R_+^J$: each $Y_j^\uu$ is nondecreasing with $Y_j^\uu(0)=0$, and increases only at times $t$ at which $Z_j^\uu(t)=0$. The pair $(Z^\uu, Y^\uu)$ thus satisfies conditions (i)--(iii) of the Skorokhod reflection mapping with input path $X^\uu$, and therefore}
\begin{equation} \label{eq: Skorokhod GJN}
    Z^\uu = \Phi(X^\uu; R), \quad Y^\uu = \Psi(X^\uu; R).
\end{equation}
{This representation underlies the strategy of the proof: we first establish a strong approximation for the free netput process $X^\uu$, and then transfer it to $Z^\uu$ via the Lipschitz continuity \eqref{eq: Lipschitz} of the Skorokhod reflection mapping.}

\subsection{Proof of asymptotic strong approximation} \label{sec: proof SA}
Throughout the paper, for any $k\in\J$, we denote the $\gamma_k^2(r)$-(diffusion) scaled process of the process $Z^\uu$ by
\begin{equation} \label{eq: def scale}
    \scale{Z}{k}{{}}  \defi \left\{ \scale{Z}{k}{{}}(t);~t\geq 0 \right\} \defi \left\{ \gamma_k(r) Z^\uu(t/\gamma_k^2(r));~t\geq 0 \right\}
\end{equation}
and the $\gamma_k^2(r)$-fluid scaled processes of the process $Z^\uu$ by
\begin{equation*}
    \scale{\bar{Z}}{k}{{}}  \defi \left\{ \scale{\bar{Z}}{k}{{}}(t);~t\geq 0 \right\} \defi \left\{ \gamma_k^2(r) Z^\uu(t/\gamma_k^2(r));~t\geq 0 \right\}.
\end{equation*}
The $\gamma_k^2(r)$-diffusion and the $\gamma_k^2(r)$-fluid scaled processes of other processes are defined similarly.

Applying this scaling to  \eqref{eq: Skorokhod GJN} yields:
\begin{equation*}
    \scale{Z}{k}{{}} = \Phi\big(\scale{X}{k}{{}}; R\big), \quad \scale{Y}{k}{{}} = \Psi\big(\scale{X}{k}{{}}; R\big),
\end{equation*}
where
\begin{align}
    \scale{X}{k}{{}}(t) &= \gamma_k(r)Z^\uu(0) - R(\mu^\uu-\lambda)t/\gamma_k(r) + \scale{V}{k}{{}}(t), \notag\\
    \scale{Y}{k}{{j}}(t)&=  \mu_j^\uu \left( t/\gamma_k(r)- \scale{B}{k}{{j}}(t) \right), \quad j\in \J, \notag \\
    \scale{V}{k}{{j}}(t) &= \left[\scale{A}{k}{j}(t)-\alpha_j t/\gamma_k(r)\right] \notag  \\
    &\quad+\sum_{i\in\J} \left( P_{i j} - \delta_{i j} \right)\left[ \scale{S}{k}{i}\left(\mu_i^\uu \scale{\bar{B}}{k}{i}(t)\right)-\mu_i^\uu  \scale{\bar{B}}{k}{i}(t)/\gamma_k(r)\right]\notag  \\
    &\quad +\sum_{i\in\J}\left[\scale{\psi}{k}{i,j}\left(\scale{\bar{S}}{k}{i} \left(\mu_i^\uu \scale{\bar{B}}{k}{i}(t)\right)\right)-P_{i j}  \scale{\bar{S}}{k}{i}\left(\mu_i^\uu \scale{\bar{B}}{k}{i}(t)\right)/\gamma_k(r)\right], ~ j\in \J. \label{eq: SA_V727} 
\end{align}

Hence, by the Lipschitz continuity of the Skorokhod reflection mapping, it is sufficient to show that the $\scale{X}{k}{{}}$ can be strongly approximated by a Brownian motion with initial state $\gamma_k(r)Z^\uu(0)$, drift vector ${-}R(\mu^\uu-\lambda)/\gamma_k(r)$ and covariance matrix $\Gamma$ defined in \eqref{eq: SA_Gamma}. By comparing the expression of $\scale{X}{k}{{}}$, we only need to show that the noise process  $\scale{V}{k}{{}}$ can be strongly approximated by a driftless Brownian motion $W^{(v)}$ with zero initial state and covariance matrix $\Gamma$: 
\begin{equation} \label{eq: SA_VV}
    \sup_{0\leq t\leq T} \max_{k\in \J} \Norm{\scale{V}{k}{{}}(t) - \gamma_k(r) W^{(v)}(t/\gamma_k^2(r))}  \aseq o\left(\gamma_1^{\varepsilon/(2+\varepsilon)}(r)\right) \quad \text{as $r\to 0$}.
\end{equation}

The proof of \eqref{eq: SA_VV} needs the following three lemmas. The proofs of these lemmas are given in {Appendix}~\ref{sec: lem proof SA}. {Lemma~\ref{lemma: SA_diffusion} gives the strong approximation of the primitive processes (arrivals, services, and routing) by Brownian motions under all diffusion scalings. {It} directly bounds the first three error terms {of \eqref{eq: SA_V727}; see} \eqref{eq: SA1}--\eqref{eq: SA3}. Lemma~\ref{lemma: SA_fluid} provides fluid bounds of law-of-the-iterated-logarithm type, which are used to control the random (busy-time) time changes appearing in the service and routing terms: they bound the range of the random time changes and show that these are uniformly close to the corresponding deterministic time changes. Lemma~\ref{lemma: SA_BM}, a modulus-of-continuity estimate for Brownian motion under all diffusion scalings, converts this closeness into a bound on the increments of the limiting Brownian motions. {It gives the} bounds {for} the last two terms {of \eqref{eq: SA_V727}}{; see} \eqref{eq: SA4} and \eqref{eq: SA5}.}

\begin{lemma} \label{lemma: SA_diffusion}
    Suppose that Assumption \ref{assmpt: moment} holds. If the underlying probability space $(\Omega, \mathscr{F}, \mathbb{P})$ is rich enough, there exist $J + 2$ independent $J$-dimensional standard Brownian motions $W^{(e)}$, $W^{(s)}$ and $W^{(\ell)}$, $\ell=1,\ldots,J$ on  this space and a set $\Omega_1 \subseteq \Omega$ with $\Prob(\Omega_1)=1$, such that for any fixed $T>0$, $k\in \J$ and $\omega\in \Omega_1$, as $r\to 0$,
    \begin{equation}
        \sup_{0\leq t \leq T} \Norm{\left[ \scale{A}{k}{{}}(t, \omega) - \alpha t/\gamma_k(r)\right] - (\Gamma^{(e)})^{1/2}\gamma_k(r)W^{(e)}(t/\gamma_k^2(r), \omega)}  = o\left( \gamma_k^{\varepsilon/(2+\varepsilon)}(r) \right), \label{eq: ASA_A}
    \end{equation}
    \begin{equation}
        \sup_{0\leq t \leq T} \Norm{\left[ \scale{S}{k}{{}}(t, \omega) - et/\gamma_k(r) \right] - (\Gamma^{(s)})^{1/2}\gamma_k(r)W^{(s)}(t/\gamma_k^2(r), \omega)}  = o\left(\gamma_k^{\varepsilon/(2+\varepsilon)}(r)\right), \label{eq: ASA_S}
    \end{equation}
    and
    \begin{equation}\label{eq: ASA_psi}
        \sup_{0\leq t \leq T} \Norm{\left[ \scale{\psi}{k}{\ell}(t, \omega) - p_{\ell}t/\gamma_k(r) \right] - (\Gamma^{(\ell)})^{1/2}\gamma_k(r)W^{(\ell)}(t/\gamma_k^2(r), \omega)} \leq C_1(\omega) \gamma_k(r)\log (1/\gamma_k(r)), 
    \end{equation}
    for any $r\in (0,r_1(T))$, where $C_1(\omega)$ is a positive random variable, $r_1(T)$ is a positive constant in $(0,1)$ and the column vector $p_{\ell}$ is the $\ell$th column of $P^\T$. The covariance matrices $\Gamma^{(e)}, \Gamma^{(s)}$ and $\Gamma^{(\ell)}$ for $\ell\in \J$ are given by
    $$
    \begin{aligned}
    \Gamma^{(e)} & =\left(\Gamma_{j k}^{(e)}\right) & & \text { with } & \Gamma_{j k}^{(e)} & =\alpha_j c_{e, j}^2 \delta_{j k}, \\
    \Gamma^{(s)} & =\left(\Gamma_{j k}^{(s)}\right) & & \text { with } & \Gamma_{j k}^{(s)} & = c_{s,j}^2 \delta_{j k}, \\
    \Gamma^{(\ell)} & =\left(\Gamma_{j k}^{(\ell)}\right) & & \text { with } & \Gamma_{j k}^{(\ell)} & =P_{\ell j}\left(\delta_{j k}-P_{\ell k}\right).
    \end{aligned}
    $$
\end{lemma}

\begin{lemma} \label{lemma: SA_fluid}
    Suppose that Assumption \ref{assmpt: moment} holds. There exists a set $\Omega_2 \subseteq \Omega$ with $\Prob(\Omega_2)=1$, such that for any fixed $T>0$, $k\in \J$, and $\omega\in \Omega_2$, we have, for any $r\in (0,r_2(T))$,    
    \begin{equation}
        \sup_{0\leq t \leq T} \Norm{\scale{\bar{A}}{k}{{}}(t, \omega) - \alpha t} \le C_2(\omega, T)\gamma_k(r){\sqrt{\log\log (1/\gamma_k(r))}} \label{eq: LIL_A}
    \end{equation}
    \begin{equation}
        \sup_{0\leq t \leq  T} \Norm{\scale{\bar{S}}{k}{{}}(t, \omega) - et}  \le C_2(\omega, T)\gamma_k(r){\sqrt{\log\log (1/\gamma_k(r))}}, \label{eq: LIL_S}
    \end{equation}
    \begin{equation}
        \sup_{0\leq t \leq T} \Norm{\scale{\bar{\psi}}{k}{\ell}(t, \omega) - p_{\ell}t}  \le C_2(\omega, T)\gamma_k(r){\sqrt{\log\log (1/\gamma_k(r))}}. \label{eq: LIL_psi}
    \end{equation}
    where $C_2(\omega,T)$ is a positive random variable and $r_2(T)$ is a constant in $(0,1)$.
    If we further suppose that $\mu^\uu_j>\lambda_j$ and $\mu^\uu_j - \lambda_j =O(\gamma_1(r))$ for all $j\in \J$, then, for any $r\in (0,r_2'(T))$ and $\omega\in \Omega_2$,
    \begin{equation}
        \sup_{0\leq t \leq T} \Norm{\scale{\bar{B}}{k}{{}}(t, \omega) - et} \leq C_2'(\omega,T) \gamma_1(r){\sqrt{\log\log (1/\gamma_1(r))}}, \label{eq: LIL_B}
    \end{equation}
    where $C_2'(\omega,T)$ is a positive random variable.
\end{lemma}
\begin{lemma} \label{lemma: SA_BM}
    Let $W_0=\{W_0(t);~t\geq 0\}$ be a one-dimensional standard Brownian motion. Then there exists a set $\Omega_3 \subseteq \Omega$ with $\Prob(\Omega_3)=1$, such that for any fixed $T>0$, every $k\in \J$, $\omega\in \Omega_3$, and $C>0$, there exists a constant $r_3\defi r_3(T,C)\in (0,1)$ so that for any $r\in (0,r_3)$,
    \begin{equation*}
        \begin{aligned}
            &\sup_{ \substack{0\leq s, t \leq T\\ \abs{s-t} \leq C \gamma_1(r){\sqrt{\log\log (1/\gamma_1(r))}}}} \abs{\gamma_k(r)W_0(s/\gamma_k^2(r), \omega)-\gamma_k(r)W_0(t/\gamma_k^2(r), \omega)} \\
            &\qquad  \leq 4\sqrt{C}\sqrt{ \gamma_1(r)\sqrt{\log\log ({1}/{\gamma_1(r)})}\left[\log(1/ \gamma_1(r))+\log \log \left( 1/\gamma_k(r)\right)\right]}.
        \end{aligned}
    \end{equation*}
\end{lemma}

\begin{proof}[Proof of Proposition \ref{prop: SA}.]
It suffices to prove the strong approximation of the noise process $\scale{V}{k}{{}}$ in \eqref{eq: SA_VV}. For $j\in \J$, define
\begin{equation*}
    W^{(v)}_j(t) = (\Gamma^{(e)})^{1/2}_{j,[1:J]}W^{(e)}(t) + \sum_{i\in \J} (P_{ij} - \delta_{ij})(\Gamma^{(s)})^{1/2}_{i,[1:J]} W^{(s)}(\lambda_i t) + \sum_{i\in \J}(\Gamma^{(i)})^{1/2}_{j,[1:J]} W^{(i)}\left( \lambda_i t\right)
\end{equation*}
and then $W^{(v)}=\{(W^{(v)}_1(t),W^{(v)}_2(t), \ldots, W^{(v)}_J(t));~t\geq 0\}$ is a $J$-dimensional driftless Brownian motion with zero initial state and covariance matrix $\Gamma$ given by \eqref{eq: SA_Gamma}. Since we have
\begin{equation*}
    \begin{aligned}
        &\sup_{0\leq t \leq T} \max_{k\in \J} \Norm{\scale{V}{k}{{}}(t) - \gamma_k(r) W^{(v)}(t/\gamma_k^2(r))} = \sup_{0\leq t \leq T} \max_{k\in \J} \max_{j\in \J} \abs{\scale{V}{k}{{j}}(t) - \gamma_k(r) W^{(v)}_j(t/\gamma_k^2(r))},
    \end{aligned}
\end{equation*}
where 
\begin{equation*}
    \begin{aligned}
        &\scale{V}{k}{{j}}(t) - \gamma_k(r) W^{(v)}_j(t/\gamma_k^2(r)) \\
        &\quad = \scale{A}{k}{j}(t)-\alpha_j t/\gamma_k(r)-(\Gamma^{(e)})^{1/2}_{j,[1:J]}\gamma_k(r)W^{(e)}(t/\gamma_k^2(r))\\
        & \qquad + \sum_{i\in\J} \left( P_{i j} - \delta_{i j} \right)\left[ \scale{S}{k}{i}\left(\mu_i^\uu \scale{\bar{B}}{k}{i}(t)\right)-\mu_i^\uu  \scale{\bar{B}}{k}{i}(t)/\gamma_k(r)\right. \notag\\
        &\qquad \qquad \qquad \qquad \quad \left. -(\Gamma^{(s)})^{1/2}_{i,[1:J]}\gamma_k(r)W^{(s)}(\mu_i^\uu  \scale{\bar{B}}{k}{i}(t)/\gamma_k^2(r))\right]\\
        &\qquad + \sum_{i\in\J}\left[\scale{\psi}{k}{i,j}\left(\scale{\bar{S}}{k}{i} \left(\mu_i^\uu \scale{\bar{B}}{k}{i}(t)\right)\right)-P_{i j}  \scale{\bar{S}}{k}{i}\left(\mu_i^\uu \scale{\bar{B}}{k}{i}(t)\right)/\gamma_k(r)\right.\notag\\
        &\qquad \qquad \quad  \left. -(\Gamma^{(i)})^{1/2}_{j,[1:J]}\gamma_k(r)W^{(i)}(\scale{\bar{S}}{k}{i} \left(\mu_i^\uu \scale{\bar{B}}{k}{i}(t)\right)/\gamma_k^2(r))\right] \\
        &\qquad + \sum_{i\in\J} \left( P_{i j} - \delta_{i j} \right)(\Gamma^{(s)})^{1/2}_{i,[1:J]}\left[ \gamma_k(r)W^{(s)}\left( \mu_i^\uu  \scale{\bar{B}}{k}{i}(t)/\gamma_k^2(r) \right)-\gamma_k(r)W^{(s)}(\lambda_it/\gamma_k^2(r))\right] \\
        &\qquad + \sum_{i\in\J} (\Gamma^{(i)})^{1/2}_{j,[1:J]}\left[ \gamma_k(r)W^{(i)}\left(\scale{\bar{S}}{k}{i} \left(\mu_i^\uu \scale{\bar{B}}{k}{i}(t)\right)/\gamma_k^2(r)\right)-\gamma_k(r)W^{(i)}(\lambda_it/\gamma_k^2(r))\right],
    \end{aligned}
\end{equation*}
the proof now reduces to showing that each of the five terms in the decomposition above converges to zero at the desired rate. Specifically, we need to establish the following five bounds for all $i,j,k\in \J$,
\begin{equation} \label{eq: SA1}
    \sup_{0\leq t \leq T} \abs{\scale{A}{k}{j}(t)-\alpha_j t/\gamma_k(r)-(\Gamma^{(e)})^{1/2}_{j,[1:J]}\gamma_k(r)W^{(e)}(t/\gamma_k^2(r))} \aseq o\left(\gamma_1^{\varepsilon/(2+\varepsilon)}(r)\right),
\end{equation}
\begin{equation} \label{eq: SA2}
    \begin{aligned}
        &\sup_{0\leq t \leq T} \left|\scale{S}{k}{i}\left(\mu_i^\uu \scale{\bar{B}}{k}{i}(t)\right)-\mu_i^\uu  \scale{\bar{B}}{k}{i}(t)/\gamma_k(r) \right.\\
        &\qquad \quad   \left. -(\Gamma^{(s)})^{1/2}_{i,[1:J]}\gamma_k(r)W^{(s)}(\mu_i^\uu  \scale{\bar{B}}{k}{i}(t)/\gamma_k^2(r))\right| \aseq o\left(\gamma_1^{\varepsilon/(2+\varepsilon)}(r)\right),
    \end{aligned}
\end{equation}
\begin{equation} \label{eq: SA3}
    \begin{aligned}
        &\sup_{0\leq t \leq T} \left|\scale{\psi}{k}{i,j}\left(\scale{\bar{S}}{k}{i} \left(\mu_i^\uu \scale{\bar{B}}{k}{i}(t)\right)\right)-P_{i j}  \scale{\bar{S}}{k}{i}\left(\mu_i^\uu \scale{\bar{B}}{k}{i}(t)\right)/\gamma_k(r)\right.\\
        &\qquad \quad   \left. -(\Gamma^{(i)})^{1/2}_{j,[1:J]}\gamma_k(r)W^{(i)}(\scale{\bar{S}}{k}{i} \left(\mu_i^\uu \scale{\bar{B}}{k}{i}(t)\right)/\gamma_k^2(r))\right| \aseq o\left(\gamma_1^{\varepsilon/(2+\varepsilon)}(r)\right),
    \end{aligned}
\end{equation}
\begin{equation} \label{eq: SA4}
    \sup_{0\leq t \leq T} \abs{\gamma_k(r)W^{(s)}\left( \mu_i^\uu  \scale{\bar{B}}{k}{i}(t)/\gamma_k^2(r) \right)-\gamma_k(r)W^{(s)}(\lambda_it/\gamma_k^2(r))} \aseq o\left(\gamma_1^{\varepsilon/(2+\varepsilon)}(r)\right),
\end{equation}
\begin{equation} \label{eq: SA5}
    \sup_{0\leq t \leq T} \abs{\gamma_k(r)W^{(i)}\left(\scale{\bar{S}}{k}{i} \left(\mu_i^\uu \scale{\bar{B}}{k}{i}(t)\right)/\gamma_k^2(r)\right)-\gamma_k(r)W^{(i)}(\lambda_it/\gamma_k^2(r))} \aseq o\left(\gamma_1^{\varepsilon/(2+\varepsilon)}(r)\right),
\end{equation}

    Let $\Omega_1$, $\Omega_2$ and $\Omega_3$ be the subsets of $\Omega$ from Lemmas \ref{lemma: SA_diffusion}, \ref{lemma: SA_fluid} and \ref{lemma: SA_BM} that have full probability measures. We first note by \eqref{eq: ASA_A} in  Lemma~\ref{lemma: SA_diffusion} that \eqref{eq: SA1} holds for any $\omega\in \Omega_1$. 

    To prove \eqref{eq: SA2}, we note that the condition $\mu^\uu_j - \lambda_j =O(\gamma_1(r))$ for all $j\in \J$ implies that there exists a positive constant $c_0$ such that $\mu^\uu_j - \lambda_j \leq c_0 \gamma_1(r)$ for all $j\in \J$.
    Since the busy time process is bounded by $0\leq B_i^\uu(t, \omega) \leq t$ for any $t\geq 0$, $i\in \J$, $r\in(0,1)$, and $\omega \in \Omega$, we obtain that 
    \begin{equation} \label{eq: B_bound}
        0\leq \mu^\uu_i \scale{\bar{B}}{k}{i}(t, \omega) \leq (\lambda_i + c_0\gamma_1(r))\gamma_k^2(r) t/\gamma_k^2(r) \leq (\lambda_{\max} + c_0)t,
    \end{equation}
    where we denote $\lambda_{\max}=\max_{j\in \J} \lambda_j$. Hence, we conclude by applying \eqref{eq: ASA_S} of Lemma~\ref{lemma: SA_diffusion} on the time horizon $[0, (\lambda_{\max} + c_0)T]$ that \eqref{eq: SA2} holds for any $\omega\in \Omega_1$ as $r\to 0$.

    To prove \eqref{eq: SA3}, we obtain by applying \eqref{eq: LIL_S} of Lemma~\ref{lemma: SA_fluid} on the time horizon $[0, (\lambda_{\max} + c_0)T]$ that
    \begin{equation} \label{eq: V_SB}
        \begin{aligned}
            &\sup_{0\leq t \leq T} \abs{\scale{\bar{S}}{k}{i} (\mu_i^\uu \scale{\bar{B}}{k}{i}(t, \omega), \omega) - \mu_i^\uu \scale{\bar{B}}{k}{i}(t, \omega)} \\
            & \quad \leq C_2(\omega, (\lambda_{\max} + c_0)T)\gamma_k(r) \sqrt{ \log\log (1/\gamma_k(r))},
        \end{aligned}
    \end{equation}
    for any $r\in (0,r_2((\lambda_{\max} + c_0)T))$ and $\omega\in \Omega_2$. Hence, we have 
    \begin{equation}\label{eq: V_SB_max}
        \begin{aligned} 
            \scale{\bar{S}}{k}{i} (\mu_i^\uu \scale{\bar{B}}{k}{i}(t, \omega), \omega) &\leq \mu_i^\uu \scale{\bar{B}}{k}{i}(t, \omega) + C_2(\omega, (\lambda_{\max} + c_0)T) \gamma_k(r) \sqrt{ \log\log(1/\gamma_k(r))}  \\
            &\leq (\lambda_{\max} + c_0)T +  C_2(\omega, (\lambda_{\max} + c_0)T)=: C_2''(\omega, T), 
        \end{aligned}
    \end{equation}
    for any $r\in (0,r_2((\lambda_{\max} + c_0)T))$, $t\in [0,T]$ and $\omega\in \Omega_2$.
    Consequently, applying \eqref{eq: ASA_psi} of Lemma~\ref{lemma: SA_diffusion} on the time horizon $[0,C_2''(\omega, T)]$ shows that \eqref{eq: SA3} holds for any $\omega\in \Omega_1 \cap \Omega_2$.

    To prove \eqref{eq: SA4}, we note that the condition $\mu^\uu_j - \lambda_j =O(\gamma_1(r))$ for all $j\in \J$ and \eqref{eq: LIL_B} of Lemma~\ref{lemma: SA_fluid} indicates that
    \begin{equation}\label{eq: V_B}
        \begin{aligned} 
            ~\sup_{0\leq t \leq T} \abs{\mu_i^\uu\scale{\bar{B}}{k}{{i}}(t, \omega) - \lambda_i t} &\leq  \sup_{0\leq t \leq T} \abs{\mu_i^\uu - \lambda_i}t + \sup_{0\leq t \leq T} \mu_i^\uu\abs{\scale{\bar{B}}{k}{{i}}(t, \omega) -  t}  \\
            & \leq Tc_0 \gamma_1(r) + (\lambda_i + c_0\gamma_1(r))C_2'(\omega, T) \gamma_1(r)\sqrt{ \log\log (1/\gamma_1(r))}  \\
            & \leq (Tc_0 + (\lambda_{\max}+c_0)C_2'(\omega, T))\gamma_1(r)\sqrt{ \log\log (1/\gamma_1(r))},
        \end{aligned}
    \end{equation}
for any $r\in (0,r_2'(T))$ and $\omega \in \Omega_2$. Hence, Lemma \ref{lemma: SA_BM} implies that
\begin{equation*}
    \begin{aligned}
        &\sup_{0\leq t \leq T} \abs{\gamma_k(r)W^{(s)}(\mu_i^\uu  \scale{\bar{B}}{k}{i}(t, \omega)/\gamma_k^2(r), \omega)-\gamma_k(r)W^{(s)}(\lambda_it/\gamma_k^2(r), \omega)} \\
        & ~ \leq 4\sqrt{(Tc_0 + (\lambda_{\max}+c_0)C_2'(\omega, T))}
        \sqrt{{\gamma_1(r)}\sqrt{ \log\log (1/\gamma_1(r))}
        \left[\log(1/ \gamma_1(r))+\log \log \left( 1/\gamma_k(r)\right)\right]},
    \end{aligned}
\end{equation*}
for any $r\in (0, r_2'((\lambda_{\max} + c_0)T) \wedge r_3(T,c_0 T + (\lambda_{\max}+c_0)C_2'(\omega, T)))$ and $\omega\in \Omega_2\cap \Omega_3$ by replacing $s$, $t$ and $T$ with $\mu_i^\uu  \scale{\bar{B}}{k}{i}(t, \omega)$, $\lambda_i t$ and $(\lambda_{\max}+c_0)T$.

For any fixed $T>0$ and $\omega\in \Omega_2\cap \Omega_3$, there exists a random variable $C_3(\omega, T)$ so that for any $r\in [r_2'((\lambda_{\max} + c_0)T) \wedge r_3(T, c_0 T + (\lambda_{\max}+c_0)C_2'(\omega, T) ), \gamma_1^{-1}(1/e) \wedge \gamma_k^{-1}(1/e)]$,
    \begin{equation*}
        \begin{aligned}
            &\sup_{0\leq t \leq T} \abs{\gamma_k(r)W^{(s)}(\mu_i^\uu  \scale{\bar{B}}{k}{i}(t, \omega)/\gamma_k^2(r), \omega)-\gamma_k(r)W^{(s)}(\lambda_it/\gamma_k^2(r), \omega)} \leq C_3(\omega, T).
        \end{aligned}
    \end{equation*}
    Hence, for the fixed $T>0$, there {exists} a random variable $C_4(\omega, T)$ such that for any $r\in (0, \gamma_1^{-1}(1/e) \wedge \gamma_k^{-1}(1/e))$,
    \begin{equation*}
        \begin{aligned}
            &\sup_{0\leq t \leq T} \abs{\gamma_k(r)W^{(s)}(\mu_i^\uu  \scale{\bar{B}}{k}{i}(t, \omega)/\gamma_k^2(r), \omega)-\gamma_k(r)W^{(s)}(\lambda_it/\gamma_k^2(r), \omega)} \\
            & \quad \leq C_4(\omega, T)
            \sqrt{{\gamma_1(r)}\sqrt{\log\log (1/\gamma_1(r))}
            \left[\log(1/ \gamma_1(r))+\log \log \left( 1/\gamma_k(r)\right)\right]}.
        \end{aligned}
    \end{equation*}
Therefore, under assumption \eqref{eq: tech}, we conclude that \eqref{eq: SA4} holds for any  $\omega\in \Omega_2\cap \Omega_3$.

To prove \eqref{eq: SA5}, by the triangular inequality, \eqref{eq: V_SB} and \eqref{eq: V_B} imply that
\begin{equation}
    \begin{aligned}
        &\sup_{0\leq t \leq T} \abs{\scale{\bar{S}}{k}{i} (\mu_i^\uu \scale{\bar{B}}{k}{i}(t, \omega), \omega) - \lambda_i t}\\
        &\quad \leq \left( C_2(\omega, (\lambda_{\max} + c_0)T)+(Tc_0 + (\lambda_{\max}+c_0)C_2'(\omega, T)) \right) \gamma_1(r)\sqrt{\log\log (1/\gamma_1(r))}
    \end{aligned}
\end{equation}
for any $r\in (0, r_2((\lambda_{\max} + c_0)T) \wedge r_2'(T))$ and any $\omega\in \Omega_2$. Hence, by Lemma \ref{lemma: SA_BM} and assumption \eqref{eq: tech},
we can similarly obtain that \eqref{eq: SA5} holds for any $\omega\in \Omega_2\cap \Omega_3$.

Combining the bounds \eqref{eq: SA1}-\eqref{eq: SA5}, we conclude that
\begin{equation*}
    \begin{aligned}
        &\sup_{0\leq t \leq T} \max_{k\in \J} \Norm{\scale{V}{k}{{}}(t, \omega) - \gamma_k(r) W^{(v)}(t/\gamma_k^2(r), \omega)} = o(\gamma_1^{\varepsilon/(2+\varepsilon)}(r)),
    \end{aligned}
\end{equation*}
for any $\omega\in \Omega_1 \cap \Omega_2\cap \Omega_3$.
This completes the proof of Proposition \ref{prop: SA}. 
\end{proof}

\section{Convergence of SRBMs} \label{sec: proof SRBM}

The proof of Theorem \ref{prop: multi-SRBM} is based on two key preliminary results: an asymptotic independence result of Brownian motions, stated in Proposition \ref{proposition: joint BM}, and the process-level convergence of the SRBMs at each individual scaling, presented in Proposition \ref{prop: rk initial}. We prove Proposition \ref{prop: rk initial} in Section \ref{sec: proof multi-SRBM each}, and then use it to complete the proof of Theorem~\ref{prop: multi-SRBM} in Section \ref{sec: proof multi-SRBM T}.

\begin{proposition}[Asymptotically independent Brownian motions] \label{proposition: joint BM}
    Let $W_0:=\{W_0(t);~t\geq 0\}$ be a  one-dimensional standard Brownian motion.
    Then  as  $r\to 0,$
    \begin{equation*}
        \begin{aligned}
            &W_0^\uu \defi \left\{ \left(
            \gamma_1(r)W_0(t/\gamma_1^2(r)),~\ldots,~ \gamma_J(r)W_0(t/\gamma_J^2(r))
        \right) ;~t\geq 0 \right\} \\
        &\qquad \Longrightarrow W^*_0 := \left\{ \left(W_{0,1}^*(t),~\ldots,~ W_{0,J}^*(t)\right);~t\geq 0 \right\} ,
        \end{aligned}
    \end{equation*} 
    where $W^*_0$ is a $J$-dimensional standard Brownian motion.
\end{proposition}

{This proposition is the conceptual heart of our analysis: it is the mechanism that produces the asymptotic independence in the functional limits. Although the $J$ coordinates of $W_0^\uu$ are all driven by the same Brownian motion $W_0$, observing $W_0$ under $J$ separated diffusion scalings yields, in the limit, $J$ mutually independent Brownian motions. The subsequent SRBM analysis then identifies the scale-specific reflected processes driven by these limiting Brownian motions. The key step of the proof is the cross-covariance computation, which shows that the correlation between the coordinates observed at two separated scales $i<j$ is of order $\gamma_j(r)/\gamma_i(r)$ {that tends to $0$ as $r\to 0$}.}

{
\begin{proof}[Proof of Proposition \ref{proposition: joint BM}.]
    First, we note that the laws of the processes $\{W_0^\uu; 0<r<1\}$ are tight in $C([0,\infty),\mathbb{R}^J)$ because each coordinate process is a  one-dimensional standard  Brownian motion and hence is tight.
Let $\tilde W_0=\{ (\tilde  W_{0, 1}(t), \cdots, \tilde W_{0, J}(t));~t\geq 0\}$ be any sub-sequential weak limit of $W_0^\uu$ as $r\to 0$, say, along the sequence $r_k\to 0$.  As the weak convergence in $C([0,\infty),\mathbb{R}^J)$ implies the convergence of finite dimensional distributions and
the weak limit of multivariate Gaussian random vectors is Gaussian (see, e.g., Proposition 3.36 on p.148 of \citet{Baldi}),
  $\tilde W_0$ is a $J$-dimensional Gaussian process. Evidently, each $\tilde W_{0, j}$ is a one-dimensional standard Brownian motion.
   For each $1\leq i<j\leq J$ and $s, t>0$,
  $(\tilde W_{0, i}(s), \tilde W_{0, j}(t))$ is jointly Gaussian with covariance
  \begin{eqnarray*}
  {\rm Cov} (\tilde W_{0, i}(s), \tilde W_{0, j}(t))
  &=&  \lim_{k\to \infty} \E \left[  \gamma_i (r_k) \gamma_j (r_k)  W_0(s/\gamma_i^2(r_k))   W_0(t/\gamma_j^2(r_k)) \right]  \\
  &=& \lim_{k\to \infty}    \gamma_i (r_k) \gamma_j (r_k)  \left( \frac{s}{\gamma_i^2(r_k)} \wedge \frac{t}{\gamma_j^2(r_k)} \right)=0.
  \end{eqnarray*}
  So  $\tilde W_{0, i}(s)$ and $ \tilde W_{0, j}(t)$ are independent for every  $1\leq i<j\leq J$ and every  $s, t>0$.
  It follows that $\{ \tilde W_{0, j}; 1\leq j\leq J\}$ are mutually independent, and so $\tilde W_0$ is a  $J$-dimensional  standard
  Brownian motion.
  As this is true for every sub-sequential weak limit of $W_0^\uu$ as $r\to 0$, we conclude that  $W_0^\uu$ converges weakly as $r\to 0$
  to  a $J$-dimensional standard Brownian motion.
\end{proof}
}
In our case, we consider the underlying $J$-dimensional standard Brownian motion $W$ in \eqref{eq: RBM def2}, which is a direct consequence of Proposition \ref{proposition: joint BM}. Following the notation in \eqref{eq: def scale}, we similarly define the $\gamma_k^2(r)$-scaled process of $W$ by 
\begin{equation*}
    \scale{ W}{k}{{}}  \defi \left\{  \scale{ W}{k}{{}}(t);~t\geq 0 \right\} \defi \left\{ \gamma_k(r) W(t/\gamma_k^2(r));~t\geq 0 \right\}.
\end{equation*}

\begin{corollary} \label{cor: BM}
    The $J$-dimensional scaled Brownian motions $\scale{ W}{k}{{}}$ for $k\in \J$ converge weakly and jointly  as $ r\to 0$:
    \begin{equation*}
        \begin{aligned}
            &\begin{bmatrix}
                \scale{ W}{1}{{}} \ldots \scale{ W}{J}{{}}
            \end{bmatrix} \Longrightarrow \begin{bmatrix}
                \scaleL{ W}{1}{{}}  \ldots \scaleL{ W}{J}{{}} 
            \end{bmatrix}  ,
        \end{aligned}
    \end{equation*} 
    where the limit process is a $J\times J$-dimensional standard Brownian motion.
\end{corollary}

The following key proposition establishes the process-level convergence for the SRBMs when viewed at each individual scaling. {{It} is the main scale-separation result in the SRBM analysis. Classical heavy-traffic limits use a single diffusion scaling, while here the SRBM must be analyzed under several diffusion scalings at once. The vanishing of the lighter-loaded queue lengths and the heavier-loaded regulator processes is suggested by the drift and initial-condition scalings, but it must be proved inside a coupled reflected system, where each regulator process enters the dynamics of the {whole system}. The block Gaussian elimination below provides the triangular structure used to isolate the effective reflected dynamics at each scale.}

\begin{proposition}[$\gamma_k^2(r)$-scaling convergence] \label{prop: rk initial}
    Suppose that Assumptions \ref{assmpt: multiscaling} and \ref{assmpt: initial SRBM} hold. Then 
      as $r\to 0$,        
    \begin{equation*}
        \left\{ \gamma_k(r) \tilde Z^\uu_k \big(t / \gamma_k^2(r)\big) ;~t\geq 0\right\} \Longrightarrow \tilde Z^*_k  \quad 
        \hbox{for each $1\leq k\leq J$, }
    \end{equation*}
    where  $\tilde Z_k^*$ is a one-dimensional SRBM with initial state $\tilde\xi_k$, drift $\tilde\theta_k= -(1-w_{kk})$ and variance $\tilde\sigma^2_k$ defined in \eqref{eq: u},  while  
        \begin{equation*}
        \left\{ \gamma_k(r) \tilde Z^\uu_i \big(t / \gamma_k^2(r)\big) ;~t\geq 0 \right\} \Longrightarrow 0  
        \quad \hbox{for every $1\leq i<k \leq J$,}
            \end{equation*}
    \begin{equation*}
         \left\{ \gamma_k(r) \tilde Y^\uu_j \big(t / \gamma_k^2(r)\big) ;~t\geq 0 \right\} \Longrightarrow 0   
          \quad \hbox{for every $1\leq k <j \leq J$,}
    \end{equation*}
  
\end{proposition}

\subsection{Proof of Proposition \ref{prop: rk initial}} \label{sec: proof multi-SRBM each}

To illustrate the idea of the proof, we first present the proof of Proposition~\ref{prop: rk initial} with a two-dimensional example in Section \ref{sec: toy}. The full proof of Proposition~\ref{prop: rk initial} is given in Section~\ref{sec: proof multi-SRBM}.

The proof of Proposition \ref{prop: rk initial} depends on the following two technical lemmas. {Lemma~\ref{lemma: Skorokhod negative infty} is the one-dimensional case of Lemma~\ref{lemma: block Skorokhod negative infty}, whose proof is given in Appendix~\ref{sec: block heavy traffic proof}.} The proof of Lemma \ref{lemma: Skorokhod positive infty} is presented in Appendix \ref{sec: proof lemma k}. {Lemma~\ref{lemma: Skorokhod negative infty} shows that a one-dimensional reflected path tends to $0$ when its drift diverges to $-\infty$, and is used to show that the queue-length processes of the lighter-loaded stations vanish under the $\gamma_k^2(r)$-scaling, that is, $\scale{\tilde Z}{k}{i}\Longrightarrow 0$ {as $r\to 0$} for $i<k$. Lemma~\ref{lemma: Skorokhod positive infty} shows that the regulator process tends to $0$ when the initial level of the input path diverges to $\infty$, and is used to show that the regulator processes of the heavier-loaded stations vanish, that is, $\scale{\tilde Y}{k}{j}\Longrightarrow 0$ {as $r\to 0$} for $j>k$.}  

\begin{lemma} \label{lemma: Skorokhod negative infty}
    Consider two families of functions $u^\uu,v^\uu\in \C([0,\infty),\R)$ and scalars $m^\uu\in \R$ for $r\in (0,1)$. 
    Suppose the following conditions hold:
    \begin{enumerate}
        \item[\rm{(i)}] {$u^\uu(0)\in\R_+$ for all $r\in(0,1)$, $u^\uu(0)\to0$,} and $u^\uu \rightarrow u$ u.o.c.~as $r \rightarrow 0$;
        \item[\rm{(ii)}] $v^\uu(0)=0$ and $v^\uu$ is  nondecreasing on $[0,\infty)$ for each $r\in (0,1)$;
        \item[\rm{(iii)}] $\lim_{r\to0}m^\uu =\infty$.
    \end{enumerate}
    Define
    $$
    z^\uu=\Phi\left(\left\{ u^\uu(t)-v^\uu(t)-m^\uu t,~t\geq 0\right\}\right).
    $$
    Then the function $z^\uu$ converges to $0$ u.o.c.~as $r\to 0$.
\end{lemma}

\begin{lemma} \label{lemma: Skorokhod positive infty}
    Consider three families of functions $u^\uu,v^\uu\in \C([0,\infty),\R)$ and scalars $a^\uu\in \R_+$ for $r\in (0,1)$. Suppose the following conditions hold:
    \begin{enumerate}
        \item[\rm{(i)}] $u^\uu(0)=0$ for all $r\in(0,1)$ and $u^\uu \rightarrow u$ u.o.c.~as $r \rightarrow 0$;
        \item[\rm{(ii)}] $v^\uu(0)=0$, $v^\uu$ is  nondecreasing on $[0,\infty)$ for each $r\in (0,1)$, and there exists a $\tilde{v}^\uu\in \C([0,\infty),\R)$ such that $v^\uu \leq \tilde{v}^\uu$ and $\tilde{v}^\uu \rightarrow \tilde{v}$ u.o.c.~as $r \rightarrow 0$;
        \item[\rm{(iii)}] $\lim_{r\to0}a^\uu = \infty$.
    \end{enumerate}
    Define
    $$
    y^\uu =\Psi\left(\left\{ u^\uu(t)-v^\uu(t)+a^\uu  ;~t\geq 0\right\}; c\right),
    $$
    where the reflection scalar $c$ is a positive constant.
    Then $y^\uu$ converges to $0$ u.o.c.~as $r\to 0$.
\end{lemma}

Following the notation in \eqref{eq: def scale} and applying the $\gamma_k^2(r)$-scaling to the definition of SRBM in \eqref{eq: RBM def} and \eqref{eq: RBM def2} yields
\begin{equation*}
    \scale{\tilde Z}{k}{{}}(t) = \scale{\tilde X}{k}{{}}(t) + R  \scale{\tilde Y}{k}{{}}(t),
\end{equation*}
where
\begin{equation*}
    \scale{\tilde X}{k}{{}}(t) = \scale{\tilde Z}{k}{{}}(0) + {\theta^\uu} t / \gamma_k(r) + {L}\scale{W}{k}{{}}(t) \quad \text{and} \quad \scale{\tilde Y}{k}{{}} = \Psi\big(\scale{\tilde X}{k}{{}}; R\big).
\end{equation*}
Here, the drift vector is $\theta^\uu=-R\delta^\uu$ by Assumption \ref{assmpt: multiscaling}. 
We similarly write Proposition~\ref{prop: rk initial} as 
\begin{equation*}
        \scale{\tilde Z}{k}{i} \Longrightarrow {0}, \quad \scale{\tilde Y}{k}{j} \Longrightarrow {0}, \quad \scale{\tilde Z}{k}{k}  \Longrightarrow \tilde Z^*_k  \quad \text{for $i<k$ and $j>k$ as }r\to 0,    
\end{equation*}

We employ the Skorokhod representation theorem and reduce the problem to establishing the almost sure convergence. Specifically, since the initial state and the Brownian motion are the only randomness of the SRBM, by the weak convergence of the initial states in Assumption \ref{assmpt: initial SRBM} and Brownian motions in Corollary~\ref{cor: BM}, it is sufficient to prove that there {exist} random variables $\scale{\hat Z}{k}{{}}(0)\overset{d}{=}\scale{\tilde Z}{k}{{}}(0)$ and $\hat{\xi}\overset{d}{=} \tilde{\xi}$ with $\scale{\hat Z}{j}{{j}}(0)\overset{a.s.}{\to} \hat{\xi}_j$ for $j\in \J$ as $r\to 0$ and scaled Brownian motions $\scale{\hat W}{k}{{}}\overset{d}{=}\scale{W}{k}{{}}$ and $\scaleL{\hat W}{k}{{}}\overset{d}{=}\scaleL{W}{k}{{}}$ with $\scale{\hat W}{k}{{}}\overset{a.s.}{\to} \scaleL{\hat W}{k}{{}}$ u.o.c.~as $r\to 0$ such that
\begin{equation} \label{eq: suff prop3}
    \scale{\hat Z}{k}{{i}} \overset{a.s.}{\to} {0}, \quad \scale{\hat Y}{k}{{j}} \overset{a.s.}{\to} {0}, \quad \scale{\hat Z}{k}{k} \overset{a.s.}{\to} \hat Z^*_k \quad \text{for $i<k$ and $j>k$ u.o.c.~as $r\to 0$},
\end{equation}
where 
\begin{equation*}
    \scale{\hat Z}{k}{} = \scale{\hat X}{k}{{}} +  R \scale{\hat Y}{k}{}, \quad \scale{\hat Y}{k}{} = \Psi\left(\scale{\hat X}{k}{{}}; R\right), \quad \hat Z^*_k \overset{d}{=} \tilde Z^*_k,
\end{equation*}
and
\begin{equation*} 
    \scale{\hat X}{k}{{}}(t) = \scale{\hat Z}{k}{{}}(0) - R \delta^\uu t / \gamma_k(r) + {L}\scale{\hat W}{k}{{}}(t).
\end{equation*}

\subsubsection{Illustration of the proof of Proposition \ref{prop: rk initial} by a two-dimensional example} \label{sec: toy}

We now use a two-dimensional SRBM to illustrate the main ideas behind the proof of Proposition \ref{prop: rk initial}. The full proof of Proposition \ref{prop: rk initial} is given in Section \ref{sec: proof multi-SRBM}.
Without loss of generality (see Proposition 8 in \citet{DaiHarr1992}), we let the covariance matrix be $\Gamma=L L^\T$
(via the Cholesky decomposition) and the reflection matrix be
\begin{equation*}
    L = \begin{bmatrix}
        1 & 0\\
        L_{21}  & 1
    \end{bmatrix} \quad \text{and} \quad 
    R=\begin{bmatrix}
        1 & -q_{12} \\
        -q_{21} & 1 
    \end{bmatrix}.
\end{equation*}
We assume that $R$ is an $\caM$-matrix, and hence $q_{12}q_{21}<1$ and $q_{12}, q_{21}\geq 0$. The multi-scaling regime in \eqref{eq: drift epsilon} implies that the drift vector $\theta^\uu= - R \delta^\uu = [-\gamma_1(r)+q_{12}\gamma_2(r), q_{21}\gamma_1(r)-\gamma_2(r)]^\T$.

In this case, Proposition \ref{prop: rk initial} asserts that as $r\to 0$,
\begin{align*}
    &\left\{\gamma_1(r)\tilde Z^{(r)}_1(t/\gamma_1^2(r));~t\geq 0\right\}\Longrightarrow \Phi\left(\left\{\tilde \xi_1 - t + \scaleL{W}{1}{{1}}(t);~t\geq 0\right\}\right),\\
    &\left\{ \gamma_2(r)\tilde Z^{(r)}_2(t/\gamma_2^2(r));~t\geq0 \right\}  \\
    &\qquad\qquad\qquad\Longrightarrow \Phi\left(\left\{ \tilde \xi_2-\left( 1-q_{12}q_{21} \right)t +(q_{21}+L_{21})\scaleL{W}{2}{1}(t)+\scaleL{W}{2}{{2}}(t);~t\geq 0 \right\}\right).
\end{align*}

By the discussion of {the} last section, it is sufficient to prove that as $r\to 0$,
\begin{align}
    \left\{\scale{\hat Z}{1}{{}}(t);~t\geq 0\right\} &\overset{a.s.}{\to} \Phi\left(\left\{\hat \xi_1 - t +   \scaleL{\hat W}{1}{{1}}(t);~t\geq 0\right\}\right) \label{eq: rkZ1}, \\
    \left\{\scale{\hat Z}{2}{{}}(t);~t\geq 0\right\} &\overset{a.s.}{\to} \Phi\left(\left\{\hat \xi_2-\left( 1-q_{12}q_{21} \right)t +(q_{21}+L_{21}) \scaleL{ \hat W}{2}{{1}}(t)+ \scaleL{\hat W}{2}{{2}}(t);~t\geq 0\right\}\right). \label{eq: rkZ2}
\end{align}
We consider the $\gamma_1^2(r)$-scaled version as follows: for all $t\geq 0$,
\begin{equation}
    \begin{aligned}
        \begin{bmatrix}
            \scale{\hat Z}{1}{{1}}(t) \\ \scale{\hat Z}{1}{{2}}(t)
        \end{bmatrix}&= \begin{bmatrix}
            \scale{\hat Z}{1}{{1}}(0) \\ \scale{\hat Z}{1}{{2}}(0)
        \end{bmatrix} + \begin{bmatrix}
            -1+q_{12}{\gamma_2(r)}/{\gamma_1(r)}\\
            q_{21}-{\gamma_2(r)}/{\gamma_1(r)}
        \end{bmatrix} t  + L  \begin{bmatrix}
            \scale{\hat W}{1}{{1}}(t)\\ \scale{\hat W}{1}{{2}}(t)
        \end{bmatrix} + R\begin{bmatrix}
            \scale{\hat Y}{1}{{1}}(t) \\ \scale{\hat Y}{1}{{2}}(t)
        \end{bmatrix}.
    \end{aligned}
    \label{eq: RBM1 r 2d} 
\end{equation}
Note that
\begin{equation*}
    \begin{aligned}
        & \scale{\hat Z}{1}{{}}(t/\gamma_1^2(r))\in \R_+^2, \quad t\geq 0,\notag\\
        &\text{$\{\scale{\hat Y}{1}{{}}(t);~t\geq 0\}$ is continuous and nondecreasing in $t$ with $\scale{\hat Y}{1}{{}}(0)=\mathbf{0}$,}\notag \\
        & \int_0^\infty \scale{\hat Z}{1}{{}}(s)d\left(\scale{\hat Y}{1}{{}}(s)\right) = 0. \notag
    \end{aligned}
\end{equation*}
The first component of \eqref{eq: RBM1 r 2d} is given by
\begin{equation*}
    \begin{aligned}
        \scale{\hat Z}{1}{{1}}(t) &= \scale{\hat Z}{1}{{1}}(0) - \left( 1  -q_{12}{\gamma_2(r)}/{\gamma_1(r)}\right)  t + \scale{\hat W}{1}{{1}}(t) -q_{12} \scale{\hat Y}{1}{{2}}(t) + \scale{\hat Y}{1}{{1}}(t).
    \end{aligned}
\end{equation*}
By the uniqueness of the solution to the one-dimensional Skorokhod problem, we can write
\begin{equation*}
    \begin{aligned}
        &\left\{ \scale{\hat Z}{1}{{1}}(t);~t\geq 0 \right\}  = \Phi\left(\left\{  \scale{\hat Z}{1}{{1}}(0) - \left( 1  -q_{12}{\gamma_2(r)}/{\gamma_1(r)}  \right) t + \scale{\hat W}{1}{{1}}(t) -q_{12} \scale{\hat Y}{1}{{2}}(t) ;~t\geq 0\right\}\right).
    \end{aligned}
\end{equation*}
Assuming for now that the process $\{\scale{\hat Y}{1}{{2}}(t);~t\geq0\}$ converges weakly to $0$ u.o.c.~as $r\to 0$ (which we will prove shortly), due to  Assumptions \ref{assmpt: multiscaling} and \ref{assmpt: initial SRBM} and  Corollary~\ref{cor: BM}, we have
\begin{equation*}
    \begin{aligned}
        &\left\{ \scale{\hat Z}{1}{{1}}(0) - \left( 1  -q_{12}{\gamma_2(r)}/{\gamma_1(r)}  \right)  t + \scale{\hat W}{1}{{1}}(t) -q_{12} \scale{\hat Y}{1}{{2}}(t)  ;~t\geq 0 \right\}  \\
        &\qquad \qquad \qquad \overset{a.s.}{\to} \left\{ \hat \xi_1-t +\scaleL{\hat W}{1}{{1}}(t) ;~t\geq 0 \right\}
    \end{aligned}
\end{equation*}
u.o.c.~as $r\to0$.
Thus, \eqref{eq: rkZ1} follows from the Lipschitz continuity of the Skorokhod reflection mapping.

\vspace{0.5em}

To complete the proof, it remains to show that $\{\scale{\hat Y}{1}{{2}}(t);~t \geq 0 \}\overset{a.s.}{\to} 0$ u.o.c.~as $r\to 0$. We first show that the family $\{\scale{\hat Y}{1}{{2}}(t);~t\geq 0\}_{r\in(0,1)}$ is pathwise bounded by a continuous nondecreasing process. 

To show this, we use a comparison argument. Consider the driving process $\{\scale{\check X}{1}{{}}(t);~t\geq 0\}:=\{\scale{\hat Z}{1}{{}}(0);~t\geq 0\}$, which has a regulator process $\{\scale{\check Y}{1}{{}}(t);~t\geq 0\}:=\Psi(\{\scale{\check X}{1}{{}}(t);~t\geq 0\}; R)=\mathbf{0}$. By the Lipschitz continuity of the Skorokhod reflection mapping $\Psi$ in \eqref{eq: Lipschitz}, we have for any $T>0$,
\begin{equation*}
    \begin{aligned}
    &\sup_{t\in [0,T]}\Norm{\scale{\hat Y}{1}{{2}}(t) - \scale{\check Y}{1}{{2}}(t) }  \leq  \kappa(R)\sup_{t\in [0,T]}\Norm{\begin{bmatrix}
        -1+q_{12}{\gamma_2(r)}/{\gamma_1(r)}\\
        q_{21}-{\gamma_2(r)}/{\gamma_1(r)}
    \end{bmatrix} t + L \scale{\hat W}{1}{{2}}(t) }  \quad  \text{ a.s.}
\end{aligned}
\end{equation*}
Thus, for any fixed $T>0$,
\begin{equation*}
    \begin{aligned}
    \sup_{t\in [0,T]}\max_{j=1,2} \scale{\hat Y}{1}{{j}}(t) & \leq \kappa(R)\left( \max{ \left\{\abs{-1+q_{12}\frac{\gamma_2(r)}{\gamma_1(r)}}, \abs{q_{21}-\frac{\gamma_2(r)}{\gamma_1(r)}} \right\} } T + \sup_{t\in [0,T]}\Norm{  L \scale{\hat W}{1}{{}}(t) } \right) \\
    & \leq \kappa(R)\max\{1,q_{12},q_{21}\} T + \kappa(R)\sup_{t\in [0,T]}\Norm{ L \scale{\hat W}{1}{{}}(t) } ,
\end{aligned}
\end{equation*}
{ 
where the right-hand side converges u.o.c.~almost surely.  Hence, the
preceding bound provides the majorant required in condition~(ii) of
Lemma~\ref{lemma: Skorokhod positive infty}. 
}
To apply Lemma \ref{lemma: Skorokhod positive infty} to prove $\{\scale{\hat Y}{1}{{2}}(t);~t\geq 0\}\overset{a.s.}{\to} 0$ u.o.c.~as $r\to 0$, we consider the second component of \eqref{eq: RBM1 r 2d}:
\begin{equation*}
    \begin{aligned}
        \scale{\hat Z}{1}{{2}}(t)
        &= \scale{\hat Z}{1}{{2}}(0) + \left(q_{21}-{\gamma_2(r)}/{\gamma_1(r)}\right) t + {L_{21}} \scale{\hat W}{1}{{1}}(t)  +  \scale{\hat W}{1}{{2}}(t)- q_{21}\scale{\hat Y}{1}{{1}}(t) + \scale{\hat Y}{1}{{2}}(t). 
    \end{aligned}
\end{equation*}
By the uniqueness of the solution to the one-dimensional Skorokhod problem, we have
\begin{equation*}
    \begin{aligned}
    & \scale{\hat Y}{1}{{2}}= \Psi\left(\left\{  \scale{\hat Z}{1}{{2}}(0) + \left(q_{21}-\frac{\gamma_2(r)}{\gamma_1(r)}\right)t + {L_{21}} \scale{\hat W}{1}{{1}}(t)+\scale{\hat W}{1}{{2}}(t)- q_{21}\scale{\hat Y}{1}{{1}}(t);~t\geq 0 \right\}; 1\right).
    \end{aligned}
\end{equation*}
By \eqref{eq: initB} of Assumption \ref{assmpt: initial SRBM}, we have $\scale{\hat Z}{1}{{2}}(0)\overset{a.s.}{\to} \infty$ and conclude that $\{\scale{\hat Y}{1}{{2}}(t);~t\geq 0\}\overset{a.s.}{\to} 0$ u.o.c.~as $r\to 0$ by Lemma \ref{lemma: Skorokhod positive infty}. This completes the proof of \eqref{eq: rkZ1}.

\vspace{0.5em}

To prove the weak convergence \eqref{eq: rkZ2}, we consider the $\gamma_2^2(r)$-scaled version as follows
\begin{equation}\label{eq: RBM1 r2 2d}
    \begin{aligned}
        \begin{bmatrix}
            \scale{\hat Z}{2}{{1}}(t) \\ \scale{\hat Z}{2}{{2}}(t)
        \end{bmatrix}&= \begin{bmatrix}
            \scale{\hat Z}{2}{{1}}(0) \\ \scale{\hat Z}{2}{{2}}(0)
        \end{bmatrix} + \begin{bmatrix}
            -{\gamma_1(r)}/{\gamma_2(r)}+q_{12}\\
            q_{21}{\gamma_1(r)}/{\gamma_2(r)}-1
        \end{bmatrix} t  +  \begin{bmatrix}
            \scale{\hat W}{2}{{1}}(t)\\{L_{21}}\scale{\hat W}{2}{{1}}(t) + \scale{\hat W}{2}{{2}}(t)
        \end{bmatrix}\\
        & \quad + \begin{bmatrix}
            \scale{\hat Y}{2}{{1}}(t)  -q_{12}\scale{\hat Y}{2}{{2}}(t)\\ -q_{21} \scale{\hat Y}{2}{{1}}(t)  +\scale{\hat Y}{2}{{2}}(t)
        \end{bmatrix} , \quad t\geq 0.
    \end{aligned}
\end{equation}
Note that
\begin{equation*}
    \begin{aligned}
        & \scale{\hat Z}{2}{{}}(t)\in \R_+^2, \quad t\geq 0, \notag\\
        &\text{$\{\scale{\hat Y}{2}{{}}(t);~t\geq 0\}$ is continuous and nondecreasing in $t$ with $\scale{\hat Y}{2}{{}}(0)=\mathbf{0}$}\notag\\
        & \int_0^\infty \scale{\hat Z}{2}{{j}}(s)d\left(\scale{\hat Y}{2}{{j}}(s)\right) = 0, \quad j=1,2. \notag
    \end{aligned}
\end{equation*}

To analyze the second component, we construct a linear combination of the two state processes that eliminates the diverging drift term $\gamma_1(r)/\gamma_2(r)$. {In the general {case}, this drift-cancelling combination is performed by the block Gaussian elimination in \eqref{eq: Gauss elim}.} Multiplying the first component of \eqref{eq: RBM1 r2 2d} by $q_{21}$ and adding it to the second yields:
\begin{equation*}
    \begin{aligned}
            q_{21}\scale{\hat Z}{2}{{1}}(t) +\scale{\hat Z}{2}{{2}}(t)&= q_{21} \scale{\hat Z}{2}{{1}}(0) + \scale{\hat Z}{2}{{2}}(0) - \left( 1-q_{12}q_{21} \right)t+ (q_{21}+L_{21})\scale{\hat W}{2}{{1}}(t) \\
            &\qquad +\scale{\hat W}{2}{{2}}(t) +  \left( 1-q_{12}q_{21} \right) \scale{\hat Y}{2}{{2}}(t)
        , \quad t\geq 0.
    \end{aligned}
\end{equation*}
Hence, we have
\begin{equation*}
    \begin{aligned}
            &\left\{ \scale{\hat Z}{2}{{2}}(t);~t\geq 0 \right\}= \Phi\left(\left\{ -q_{21}\scale{\hat Z}{2}{{1}}(t)  + q_{21}\scale{\hat Z}{2}{{1}}(0) +  \scale{\hat Z}{2}{{2}}(0) - \left( 1-q_{12}q_{21} \right) t \right.\right.\\
            &\qquad \qquad\qquad \qquad  \qquad +\left.\left. (q_{21}+L_{21})\scale{\hat W}{2}{{1}}(t)+\scale{\hat W}{2}{{2}}(t);~t\geq 0\right\} \right).
    \end{aligned}   
\end{equation*}
If the process $\{\scale{\hat Z}{2}{{1}}(t);~t\geq 0\}$ converges weakly to $0$ u.o.c.~as $r\to 0$ (which will be established below), by Assumption \ref{assmpt: initial SRBM} and Corollary \ref{cor: BM}, we have
\begin{equation*}
    \begin{aligned}
        &\left\{-q_{21}\scale{\hat Z}{2}{{1}}(t)  + q_{21}\scale{\hat Z}{2}{{1}}(0) +  \scale{\hat Z}{2}{{2}}(0) - \left( 1-q_{12}q_{21} \right) t  + (q_{21}+L_{21})\scale{\hat W}{2}{{1}}(t) +\scale{\hat W}{2}{{2}}(t);t\geq 0\right\}  \\
        &\qquad \Longrightarrow \left\{\hat \xi_2 - \left( 1-q_{12}q_{21} \right) t  +  (q_{21}+L_{21})\scaleL{\hat W}{2}{{1}}(t)+\scaleL{\hat W}{2}{{2}}(t);~t\geq 0\right\}.
    \end{aligned}
\end{equation*}
Thus, \eqref{eq: rkZ2} follows from the Lipschitz continuity of the Skorokhod reflection mapping.

To complete the proof, it remains to show that $\{\scale{\hat Z}{2}{{1}}(t);~t\geq 0\}\asto 0$ u.o.c.~as $r\to 0$. 
To apply Lemma \ref{lemma: Skorokhod negative infty} to prove this, we consider the first component of \eqref{eq: RBM1 r2 2d}:
\begin{equation*}
    \begin{aligned}
        \scale{\hat Z}{2}{{1}}(t) &= \scale{\hat Z}{2}{{1}}(0) +(-{\gamma_1(r)}/{\gamma_2(r)}+q_{12})t + \scale{\hat W}{2}{{1}}(t) + \scale{\hat Y}{2}{{1}}(t)  -q_{12}\scale{\hat Y}{2}{{2}}(t), \quad t\geq 0.
    \end{aligned}
\end{equation*}
By the uniqueness of the Skorokhod reflection mapping for one-dimensional SRBM, we have
\begin{equation*}
    \begin{aligned}
        &\left\{ \scale{\hat Z}{2}{{1}}(t);~t\geq 0 \right\} = \Phi\left(\left\{  \scale{\hat Z}{2}{{1}}(0) +(-{\gamma_1(r)}/{\gamma_2(r)}+q_{12})t + \scale{\hat W}{2}{{1}}(t)  -q_{12}\scale{\hat Y}{2}{{2}}(t);~t\geq 0 \right\}\right).
    \end{aligned}
\end{equation*}
By Assumption \ref{assmpt: initial SRBM}, we have $\scale{\hat Z}{2}{{1}}(0)\overset{a.s.}{\to} 0$ and conclude that $\{\scale{\hat Z}{2}{{1}}(t);~t\geq 0\}\asto 0$ u.o.c.~as $r\to 0$ by Lemma \ref{lemma: Skorokhod negative infty}. This completes the proof of Proposition \ref{prop: rk initial} for the two-dimensional case.

\subsubsection{Proof of Proposition \ref{prop: rk initial}} \label{sec: proof multi-SRBM}
The proof is based on a block partitioning of the reflection matrix $R$. For any $k=2,\ldots,J$, we write:
\begin{equation*}
    \begin{aligned}
        R = \begin{bmatrix}
            A^{(k)} & B^{(k)} \\
            C^{(k)} & D^{(k)}
        \end{bmatrix} 
        &=
        \begin{bmatrix}
            R_{1,1} & \cdots & R_{1,k-1} & \vline & R_{1,k} & \cdots & R_{1,J} \\
            \vdots & \ddots & \vdots & \vline & \vdots & \ddots & \vdots \\
            R_{k-1,1} & \cdots & R_{k-1,k-1} & \vline & R_{k-1,k} & \cdots & R_{k-1,J} \\
            \hline
            R_{k,1} & \cdots & R_{k,k-1} & \vline & R_{k,k} & \cdots & R_{k,J} \\
            \vdots & \ddots & \vdots & \vline & \vdots & \ddots & \vdots \\
            R_{J,1} & \cdots & R_{J,k-1} & \vline & R_{J,k} & \cdots & R_{J,J} 
        \end{bmatrix}\\
        &= 
        \begin{bmatrix}
            R_{[1:k-1],[1:k-1]} & R_{[1:k-1],[k:J]} \\
            R_{[k:J],[1:k-1]} & R_{[k:J],[k:J]}
        \end{bmatrix},
    \end{aligned}
\end{equation*}
where the submatrix notation is defined in \eqref{eq: submatrix}.
Since $R$ is an $\caM$-matrix, $A^{(k)}$ is invertible. This allows for a block Gaussian elimination of $R$:
\begin{equation} \label{eq: Gauss elim}
    \underbrace{\begin{bmatrix}
        I_{k-1} & 0 \\
        -C^{(k)}(A^{(k)})^{-1} & I_{J-k+1}
    \end{bmatrix}}_{E^{(k)}}
    \underbrace{\begin{bmatrix}
        A^{(k)} & B^{(k)} \\
        C^{(k)} & D^{(k)}
    \end{bmatrix}}_{R}
    =
    \underbrace{\begin{bmatrix}
        A^{(k)} & B^{(k)} \\
        0 & D^{(k)}-C^{(k)}(A^{(k)})^{-1}B^{(k)}
    \end{bmatrix}}_{G^{(k)}},
\end{equation}
where $E^{(k)}$ is the elementary row operation matrix and $G^{(k)}$ is the eliminated matrix  for $k=2,\ldots,J$. {This operation is the high-dimensional analogue of the linear combination used in the two-dimensional illustration above: it cancels the diverging drift terms generated by the coordinates of the lighter-loaded stations and isolates the effective dynamics of the block $[k:J]$.}
For $k=1$, we set $E^{(1)} := I$ and $G^{(1)} := R$.
Since $R$ is an $\caM$-matrix, $E^{(k)}$ has nonnegative elements and the Schur complement $G^{(k)}_{[k:J],[k:J]}=D^{(k)}-C^{(k)}(A^{(k)})^{-1}B^{(k)}$ is also an $\caM$-matrix by Lemma~3.5 of \citet{McDoSivaSushTsatWendWend2020}. 
By the definition of {$u^{(j)}$} and $w_{jj}$ in Theorem \ref{prop: multi-SRBM}, we have the following relation:
\begin{equation} \label{eq: uw}
    {u^{(j)}} = (E^{(j)}_{j,[1:J]})^\T \quad \text{and} \quad 1-w_{jj} = G^{(j)}_{jj}.
\end{equation}

We are now ready to prove the first statement of \eqref{eq: suff prop3}. The first part, $\scale{\hat Z}{k}{i} \asto 0$ for $i<k$ u.o.c.~as $r\to 0$, is vacuously true for $k=1$. For $k\geq 2$, we use the mathematical induction to prove $\scale{\hat Z}{k}{\ell} \asto 0$ u.o.c.~as $r\to 0$ from $\ell=1$ to $k-1$. The $\gamma_k^2(r)$-scaled version is
\begin{equation}
    \scale{\hat Z}{k}{{}}(t) = \scale{\hat Z}{k}{{}}(0) - R\delta^\uu t/\gamma_k(r) + L \scale{\hat W}{k}{{}}(t) + R\scale{\hat Y}{k}{{}}(t), \quad t\geq 0, \label{eq: RBM1 r}
\end{equation}
where
\begin{align}
    &\scale{\hat Z}{k}{{}}(t)\in \R_+^J  \quad \text{for all }t\geq 0, \label{eq: RBM2 r} \\
    &\text{$\scale{\hat Y}{k}{{i}}$ is continuous and nondecreasing with $\scale{\hat Y}{k}{{i}}(0)=0$  $\quad \hbox{for every } i\in \J$} \label{eq: RBM3 r}\\
    & \int_0^\infty \scale{\hat Z}{k}{i}(s)d\scale{\hat Y}{k}{i}(s) = 0    \quad \hbox{for every } i\in \J \label{eq: RBM4 r}.
\end{align}

For the base step $\ell=1$, we examine the first component of \eqref{eq: RBM1 r}:
\begin{equation*}
    \scale{\hat Z}{k}{1}(t) = \scale{\hat Z}{k}{1}(0) - R_{1,[1:J]}\delta^\uu t/\gamma_k(r) + L_{1,[1:J]} \scale{\hat W}{k}{{}}(t) + R_{1,[1:J]} \scale{\hat Y}{k}{{}}(t), \quad t\geq 0,
\end{equation*}
which is equivalent to, for $t\geq 0$,
\begin{equation}
    \begin{aligned} 
        \scale{\hat Z}{k}{1}(t) &= \Big[ \scale{\hat Z}{k}{1}(0) -\Big( R_{11} \frac{\gamma_1(r)}{\gamma_k(r)} + \sum_{j=2}^J R_{1j}\frac{\gamma_j(r)}{\gamma_k(r)} \Big)t + L_{1,[1:J]} \scale{\hat W}{k}{{}}(t) + \sum_{j=2}^J R_{1j}\scale{\hat Y}{k}{j}(t) \Big] \\
        & \quad + R_{11}\scale{\hat Y}{k}{1}(t) .
    \end{aligned} \label{eq: RBM1 r 1}
\end{equation}
Therefore, \eqref{eq: RBM2 r}-\eqref{eq: RBM1 r 1} uniquely define $\scale{\hat Z}{k}{1}$ as the Skorokhod reflection mapping $\Phi$ of the one-dimensional driving process $\scale{\hat U}{k}{1}$, where $\scale{\hat U}{k}{1}$ is the process inside the square bracket of \eqref{eq: RBM1 r 1}. 

Assumption \ref{assmpt: initial SRBM} {implies} $\scale{\hat Z}{k}{1}(0)\asto 0$ u.o.c.~as $r\to 0$ for $k\geq 2$. Since $R$ is an $\mathcal{M}$-matrix, $R_{11} > 0$ and $R_{1 j} \leq 0$ for $j>1$, we have $ R_{11} \gamma_1(r)/\gamma_k(r) + \sum_{j=2}^J R_{1j}\gamma_j(r)/\gamma_k(r) \to \infty$ as $r\to 0$. By Corollary~\ref{cor: BM}, the driving Brownian motion converges weakly. Since the drift term in $\scale{\hat U}{k}{1}$ diverges to $-\infty$, we can apply Lemma~\ref{lemma: Skorokhod negative infty} to conclude that $\scale{\hat Z}{k}{1} \asto 0$ u.o.c.~as $r\to 0$.

For the induction step $2\leq \ell< k$, we assume that $\scale{\hat Z}{k}{i} \to 0$ u.o.c.~as $r\to0$ for $i=1,\ldots,\ell-1$. Left-multiplying $E^{(\ell)}$ on both sides of \eqref{eq: RBM1 r}, we have, for $t\geq 0$,
\begin{equation*}
    E^{(\ell)}\scale{\hat Z}{k}{{}}(t) = E^{(\ell)}\scale{\hat Z}{k}{{}}(0) - G^{(\ell)}\delta^\uu t/\gamma_k(r) + E^{(\ell)}L \scale{\hat W}{k}{{}}(t) + G^{(\ell)}\scale{\hat Y}{k}{{}}(t).
\end{equation*}
Since $E^{(\ell)}_{\ell, \ell} = 1$, $E^{(\ell)}_{\ell, j} = 0$ for $j>\ell$, and $G^{(\ell)}_{\ell, j} = 0$ for $j<\ell$ in \eqref{eq: Gauss elim}, the $\ell$th component of the above equation becomes
\begin{equation*}
    \begin{aligned}
        E^{(\ell)}_{\ell, [1:\ell-1]}\scale{\hat Z}{k}{{[1:\ell-1]}}(t) + \scale{\hat Z}{k}{{\ell}}(t) &= E^{(\ell)}_{\ell, [1:\ell]}\scale{\hat Z}{k}{{[1:\ell]}}(0) - G^{(\ell)}_{\ell, [\ell:J]}\delta^\uu_{[\ell:J]} t/\gamma_k(r) \\
        &\quad + E^{(\ell)}_{\ell, [1:J]}L \scale{\hat W}{k}{{}}(t) + G^{(\ell)}_{\ell, [\ell:J]}\scale{\hat Y}{k}{{[\ell:J]}}(t),
    \end{aligned}
\end{equation*}
which is equivalent to
\begin{equation} \label{eq: RBM1 r ell}
    \begin{aligned}
        \scale{\hat Z}{k}{\ell}(t) &= \Big[ -E^{(\ell)}_{\ell, [1:\ell-1]}\scale{\hat Z}{k}{{[1:\ell-1]}}(t) + E^{(\ell)}_{\ell, [1:\ell]}\scale{\hat Z}{k}{{[1:\ell]}}(0) - \Big( G_{\ell\ell}^{(\ell)} \frac{\gamma_{\ell}(r)}{\gamma_k(r)} + \sum_{j=\ell+1}^{J}G_{\ell j}^{(\ell)} \frac{\gamma_j(r)}{\gamma_k(r)} \Big)t  \\
        &\qquad + E_{\ell,[1:J]}^{(\ell)}L \scale{\hat W}{k}{{}}(t)  + \sum_{j=\ell+1}^{J}G_{\ell j}^{(\ell)}\scale{\hat Y}{k}{j}(t) \Big] + G_{\ell\ell}^{(\ell)}\scale{\hat Y}{k}{\ell}(t).
    \end{aligned}
\end{equation}
Therefore, \eqref{eq: RBM2 r}-\eqref{eq: RBM4 r} and \eqref{eq: RBM1 r ell} uniquely define $\scale{\hat Z}{k}{\ell}$ as the Skorokhod reflection mapping $\Phi$ of the one-dimensional driving process $\scale{\hat U}{k}{\ell}$, where $\scale{\hat U}{k}{\ell}$ is the process inside the square bracket of \eqref{eq: RBM1 r ell}. 

By the induction hypothesis, we have $\scale{\hat Z}{k}{i} \asto 0$ u.o.c.~as $r\to 0$ for $i=1,\ldots,\ell-1$. Assumption~\ref{assmpt: initial SRBM} implies $\scale{\hat Z}{k}{{i}}(0)\asto 0$ as $r\to 0$ for $i<k$ and $k\geq 2$. Since $G^{(\ell)}_{[\ell:J],[\ell:J]}$ is an $\mathcal{M}$-matrix, $G^{(\ell)}_{\ell\ell} > 0$ and $G^{(\ell)}_{\ell j} \leq 0$ for $j>\ell$, we have $G_{\ell\ell}^{(\ell)} \gamma_{\ell}(r)/\gamma_k(r) + \sum_{j=\ell+1}^{J}G_{\ell j}^{(\ell)} \gamma_j(r)/\gamma_k(r) \to \infty$ as $r\to 0$. Hence, we can apply Lemma~\ref{lemma: Skorokhod negative infty} to conclude $\scale{\hat Z}{k}{\ell} \asto 0$ u.o.c.~as $r\to 0$.

In summary, by mathematical induction, we have proved that the process $\scale{\hat Z}{k}{i} \asto 0$ u.o.c.~as $r\to 0$ for $i=1,\ldots,k-1$, which completes the proof of the first part of \eqref{eq: suff prop3}.

\vspace{0.5em}

To prove the remaining parts of the proposition, we left-multiply \eqref{eq: RBM1 r} by $E^{(k)}$ to obtain: for $t\geq 0$,
\begin{equation*}
    E^{(k)}\scale{\hat Z}{k}{{}}(t) = E^{(k)}\scale{\hat Z}{k}{{}}(0) - G^{(k)}\delta^\uu t/\gamma_k(r) + E^{(k)}L \scale{\hat W}{k}{{}}(t) + G^{(k)}\scale{\hat Y}{k}{{}}(t).
\end{equation*}
Since $E^{(k)}_{[k:J], [k:J]} = I_{J-k+1}$ and $G^{(k)}_{[k:J], [1:k-1]}=0$ in \eqref{eq: Gauss elim}, the $k$th to the $J$th components of the above equation become
\begin{equation} \label{eq: RBM1 rk k}
    \begin{aligned} 
        \scale{\hat Z}{k}{[k:J]}(t) &= \Big[ -E^{(k)}_{[k:J],[1:k-1]}\scale{\hat Z}{k}{[1:k-1]}(t) + E^{(k)}\scale{\hat Z}{k}{{}}(0)- G^{(k)}_{[k:J],[k:J]}\delta_{[k:J]}^\uu t/\gamma_k(r)  \\
        &\qquad   + E^{(k)}_{[k:J], [1:J]}L \scale{\hat W}{k}{{}}(t) \Big]  + G^{(k)}_{[k:J],[k:J]}\scale{\hat Y}{k}{{[k:J]}}(t).                                  
    \end{aligned}
\end{equation}
Therefore, \eqref{eq: RBM2 r}-\eqref{eq: RBM4 r} and \eqref{eq: RBM1 rk k} uniquely define $\scale{\hat Y}{k}{[k:J]}$ as the Skorokhod reflection mapping $\Psi$ with reflection matrix $G^{(k)}_{[k:J],[k:J]}$ of the $J+1-k$-dimensional driving process $\scale{\hat S}{k}{{}}$, where $\scale{\hat S}{k}{{}}$ is the process inside the  square bracket of \eqref{eq: RBM1 rk k}. 

To prove the second statement of \eqref{eq: suff prop3}. We will first show that the family  $\{\hat Y^{(r,k)}_{[k:J]}(t);~t\geq 0\}_{r\in(0,1)}$ is pathwise bounded by a continuous nondecreasing process. For this, we use a comparison argument. Consider a simple driving process $\scale{\check{S}}{k}{{}}:=E^{(k)}\scale{\hat Z}{k}{{}}(0)$. Its associated regulator process is $\scale{\check{Y}}{k}{{}}=\Psi(\scale{\check{S}}{k}{{}}; G^{(k)}_{[k:J],[k:J]})=\mathbf{0}$. By the Lipschitz continuity of the Skorokhod reflection mapping in \eqref{eq: Lipschitz}, we have, for any fixed $T>0$,
\begin{equation*}
    \begin{aligned}
        &\sup_{t\in [0,T]}\sup_{k\leq \ell \leq J}{\scale{\hat Y}{k}{\ell}(t) }  = \sup_{t\in [0,T]}\Norm{\scale{\hat Y}{k}{[k:J]}(t) - \scale{\check{Y}}{k}{{}}(t) }  \\
        &\quad \leq \kappa(G^{(k)}_{[k:J],[k:J]})\sup_{t\in [0,T]}\Norm{ \scale{\hat S}{k}{{}}(t) - \scale{\check{S}}{k}{{}}(t)}  \\
        &\quad \leq \kappa(G^{(k)}_{[k:J],[k:J]})\sup_{t\in [0,T]}\left( \Norm{ \left\{ E^{(k)}_{[k:J],[1:k-1]}\scale{\hat Z}{k}{[1:k-1]}(t) \right\} } + \Norm{ \left\{  E^{(k)}_{[k:J], [1:J]}L \scale{\hat W}{k}{{}} (t) \right\} }  \right) \\
        &\qquad +  \kappa(G^{(k)}_{[k:J],[k:J]})\Norm{ G^{(k)}_{[k:J],[k:J]}\delta_{[k:J]}^\uu / \gamma_k(r) }t , \quad \text{a.s.}
\end{aligned}
\end{equation*}
{
where the right-hand side converges u.o.c.~almost surely.  Hence, the
preceding bound provides the majorant required in condition~(ii) of
Lemma~\ref{lemma: Skorokhod positive infty}.
}
We will use Lemma \ref{lemma: Skorokhod positive infty} to prove $\{\hat Y^{(r,k)}_{\ell}(t);~t\geq 0\}\asto 0$ u.o.c.~as $r\to 0$ for $\ell >k$. Consider the $\ell$th component of \eqref{eq: RBM1 rk k}.
Since $E^{(k)}_{\ell, \ell} = 1$ for $\ell\geq k$, $E^{(k)}_{\ell, j} = 0$ for $\ell,j\geq k$ and $\ell\neq j$, and $G^{(k)}_{\ell, j} = 0$ for $j<k<\ell$ in \eqref{eq: Gauss elim}, the $\ell$th component of \eqref{eq: RBM1 rk k} for $\ell>k$ becomes
\begin{equation} \label{eq: RBM1 rk ell}
    \begin{aligned} 
        \scale{\hat Z}{k}{\ell}(t) &= \Big[ -E^{(k)}_{\ell,[1:k-1]}\scale{\hat Z}{k}{[1:k-1]}(t) + E^{(k)}_{\ell, [1:k-1]}\scale{\hat Z}{k}{{[1:k-1]}}(0) + \scale{\hat Z}{k}{{\ell}}(0) - G^{(k)}_{\ell,[k:J]}\delta_{[k:J]}^\uu t/\gamma_k(r) \\ 
        &\quad + E^{(k)}_{\ell, [1:J]}L \scale{\hat W}{k}{{}}(t) + \sum_{j=k,j\neq \ell}^JG^{(k)}_{\ell,j}\scale{\hat Y}{k}{{j}}(t) \Big]  + G^{(k)}_{\ell,\ell}\scale{\hat Y}{k}{{\ell}}(t) .
    \end{aligned}
\end{equation}
Therefore, \eqref{eq: RBM2 r}-\eqref{eq: RBM4 r} and \eqref{eq: RBM1 rk ell} uniquely define $\scale{\hat Y}{k}{\ell}$ as the Skorokhod reflection mapping $\Psi$ with reflection scalar $G^{(k)}_{\ell,\ell}$ of the one-dimensional driving process $\scale{\hat U}{k}{\ell}$, where $\scale{\hat U}{k}{\ell}$ is the process inside the square bracket of \eqref{eq: RBM1 rk ell}. 

By Assumption~\ref{assmpt: initial SRBM}, we have $\scale{\hat Z}{k}{{[1:k-1]}}(0)\asto 0$ and $\scale{\hat Z}{k}{{\ell}}(0)\asto \infty$ as $r\to 0$ for $\ell>k$. Hence, we can conclude that $\scale{\hat Y}{k}{\ell} \asto 0$ u.o.c.~as $r\to 0$ for $\ell>k$ by Lemma \ref{lemma: Skorokhod positive infty}, which completes the proof of the second part of \eqref{eq: suff prop3}.

The $k$th component of \eqref{eq: RBM1 rk k} is
\begin{equation}\label{eq: RBM1 rk ell k}
    \begin{aligned} 
        \scale{\hat Z}{k}{k}(t) &= \left[ -E^{(k)}_{k,[1:k-1]}\scale{\hat Z}{k}{[1:k-1]}(t) + E^{(k)}_{k, [1:k-1]}\scale{\hat Z}{k}{{[1:k-1]}}(0) + \scale{\hat Z}{k}{{k}}(0) - G^{(k)}_{k,[k:J]}\delta_{[k:J]}^\uu t/\gamma_k(r) \right.  \\ 
        &\quad + E^{(k)}_{k, [1:J]}L \scale{\hat W}{k}{{}}(t) + \left.G^{(k)}_{k,[k+1:J]}\scale{\hat Y}{k}{{[k+1:J]}}(t) \right]  + G^{(k)}_{k,k}\scale{\hat Y}{k}{{k}}(t) 
    \end{aligned}
\end{equation}
Therefore, \eqref{eq: RBM2 r}-\eqref{eq: RBM4 r} and \eqref{eq: RBM1 rk ell k} imply 
\begin{equation*}     \scale{\hat Z}{k}{k} = \Phi (\scale{\hat U}{k}{k}),
\end{equation*}
where $\scale{\hat U}{k}{k}=\{\scale{\hat U}{k}{k}(t);~t\geq 0\}$ is a one-dimensional process, which is defined inside the square bracket of \eqref{eq: RBM1 rk ell k}. 

By Assumption~\ref{assmpt: initial SRBM}, we have $\scale{\hat Z}{k}{{[1:k-1]}}(0)\asto 0$ and $\scale{\hat Z}{k}{{k}}(0)\asto \xi_k$ as $r\to 0$.  As $\scale{\hat Z}{k}{[1:k-1]} \asto \mathbf{0}$ and  $\scale{\hat Y}{k}{[k+1:J]} \asto \mathbf{0}$ u.o.c.~as $r\to 0$, and Corollary \ref{cor: BM}, we have 
\begin{equation} \label{eq: limit Xk}
    \scale{\hat U}{k}{k} \asto  \scaleL{\hat U}{k}{k} \defi \left\{ \hat{\xi}_k - G^{(k)}_{k,k} t + E^{(k)}_{k,[1:J]}L \scaleL{\hat W}{k}{{}}(t);~t\geq 0 \right\}, \quad \text{u.o.c. as } r\to 0,
\end{equation}
which is a one-dimensional Brownian motion with the initial state $\hat{\xi}_k$, the drift $-G^{(k)}_{k,k}$ and the variance $E^{(k)}_{k,[1:J]}L (E^{(k)}_{k,[1:J]}L)^\T = (E^{(k)}\Gamma(E^{(k)})^\T)_{k,k}$. 
Using the relations in \eqref{eq: uw}, we identify the drift as $-(1-w_{kk})$ and the variance as $(E^{(k)}\Gamma(E^{(k)})^\T)_{k,k} = {(u^{(k)})^\T} \Gamma {u^{(k)}}$. This completes the proof of \eqref{eq: suff prop3} by the Lipschitz continuity of the Skorokhod reflection mapping.

\subsection{Proof of Theorem \ref{prop: multi-SRBM}} \label{sec: proof multi-SRBM T} \label{sec: proof multi-SRBM A}
We are now ready to prove Theorem \ref{prop: multi-SRBM}.

    It follows from {the} Skorokhod representation theorem in the proof of Proposition \ref{prop: rk initial} that
    \begin{equation*}
        \begin{aligned}
            &\left\{ \left( \gamma_1(r)\tilde Z_1^\uu(t/\gamma_1^2(r)),  \ldots, \gamma_J(r)\tilde Z_J^\uu(t/\gamma_J^2(r)) \right);~t\geq 0 \right\} = \left( \scale{\tilde Z}{1}{1}, \ldots, \scale{\tilde Z}{J}{J} \right) \\
            &\quad\overset{d}{=} \left( \scale{\hat Z}{1}{1}, \ldots, \scale{\hat Z}{J}{J} \right) =\left( \Phi(\scale{\hat U}{1}{1}), \ldots, \Phi(\scale{\hat U}{J}{J}) \right) \asto \left( \Phi(\scaleL{\hat U}{1}{1}), \ldots, \Phi(\scaleL{\hat U}{J}{J}) \right)
        \end{aligned}
    \end{equation*}
    where for each $k\in \J$, both $\scale{\hat U}{k}{k}$ and $\scaleL{\hat U}{k}{k}$ are one-dimensional processes, which is defined in the square bracket of \eqref{eq: RBM1 rk ell k} and in \eqref{eq: limit Xk}, respectively.

    Crucially, {{the mutual independence of the components of $\tilde\xi$ in Assumption~\ref{assmpt: initial SRBM}, together with}} Corollary~\ref{cor: BM}{{,}} implies that the components of the limit process, $\{\scaleL{\hat U}{k}{k}\}_{k \in \mathcal{J}}$, are mutually independent.
    By the Skorokhod representation theorem,
    we have
    \begin{equation*}
        \big(
            \scale{\tilde Z}{1}{1}, \scale{\tilde Z}{2}{2}, \ldots, \scale{\tilde Z}{J}{J}
        \big)  \Longrightarrow \big(
            \tilde Z_1^*, \tilde Z_2^*, \ldots, \tilde Z_J^*
        \big) \quad \text{ as $r\to 0$},
    \end{equation*}
    and all the coordinate processes of $\tilde Z^*$ are mutually independent. This completes the proof.

\section{Lowest-rate Initial Condition} \label{sec: conventional}

In this section, we demonstrate the crucial role of the initial condition in determining the limit process by considering an alternative to Assumption \ref{assmpt: initial}. When the initial queue lengths follow Assumption~\ref{assmpt: conventional}, we will have a different limit process. We present the functional limit results for GJNs and SRBMs under Assumption~\ref{assmpt: conventional} in Sections~\ref{sec: GJN cv} and~\ref{sec: SRBM cv}, respectively. Their proofs are given in Appendix \ref{sec: conventional proof}.
This is a particular case in the complement of condition \eqref{eq: initB} in Assumption~\ref{assmpt: initial}. More general cases about the impact of initial conditions on the functional limit for multi-scaling GJNs and SRBMs will be investigated in a subsequent paper.

\subsection{GJNs under lowest-rate initial condition} \label{sec: GJN cv}

We consider the case where all initial queue lengths are of the magnitude $1/(1-\rho_1^\uu)$.

\begin{assumption}[Lowest-rate initial condition] \label{assmpt: conventional}
    There exists an $\R_+^J$-valued random vector $\xi$ such that 
    \begin{equation*}
        \left( \gamma_1(r)Z^\uu_1(0),~\gamma_1(r)Z^\uu_2(0),~\ldots,~\gamma_1(r)Z^\uu_J(0) \right)  \Longrightarrow \xi \quad \text{as } r\to 0.
    \end{equation*} 
\end{assumption}

Note that under this assumption, {the limit in \eqref{eq: initA} of Assumption \ref{assmpt: initial} is $(\xi_1,0,\ldots,0)$, whereas} condition \eqref{eq: initB} fails.

In the following theorem, we demonstrate that under Assumption \ref{assmpt: conventional}, the sequence of GJNs still converges weakly but,
in general,  to a different limit process as $r\to 0$ when $J\geq 2$.

\begin{theorem}\label{thm:GJN cv}
    Suppose Assumptions \ref{assmpt: moment}, \ref{assmpt: multiscale} and \ref{assmpt: conventional} hold. Then 
    \begin{equation} \label{eq: lim GJN cv}
        \left\{ \left( \gamma_1(r)Z_1^\uu(t/\gamma_1^2(r)), \gamma_2(r)Z_2^\uu(t/\gamma_2^2(r)), \ldots, \gamma_J(r)Z_J^\uu(t/\gamma_J^2(r)) \right);~t\geq 0 \right\} \Longrightarrow Z^* \ \text{as } r\to 0,
    \end{equation}
    where $Z^*=\{Z^*(t);~t\geq 0\}$ is a $J$-dimensional process whose components are mutually independent. Specifically, each coordinate {process} $Z^*_j$ is the first coordinate {process} of a $(J-j+1)$-dimensional SRBM {with initial state $\xi^{(j)}$, drift vector $-G^{(j)}_{[j:J],j}$, covariance matrix $(E^{(j)} \Gamma (E^{(j)})')_{[j:J],[j:J]}$, and reflection matrix $G^{(j)}_{[j:J], [j:J]}$.}
    {Here,} the $(J-j+1)$-dimensional vector $\xi^{(j)}$ is defined as $\xi$ for $j=1$ and $\mathbf{0}$ for $j>1$, $E^{(j)}$ and $G^{(j)}$ are defined in \eqref{eq: Gauss elim}, and $\Gamma$ is given in \eqref{eq: SA_Gamma}.
\end{theorem}

Note that, unlike the case in Theorem~\ref{thm:GJN}, for $J\geq 2$ the
limit process $Z^*$ here is not a Markov process in general.
{
Nevertheless,
its distribution converges as $t\to\infty$ to a product-form exponential
distribution. }

\begin{proposition}[{Product-form limiting distribution}]
    \label{prop:Zstar-exp}
    Let $Z^*=\{Z^*(t);t\geq0\}$ be the $J$-dimensional limit process in
    \eqref{eq: lim GJN cv} given by Theorem~\ref{thm:GJN cv}. { Then
    \begin{equation}\label{eq:Zstar-limiting-distribution}
        Z^*(t)\Longrightarrow {(\zeta_1,\ldots,\zeta_J)}
        \qquad\text{as } t\to\infty,
    \end{equation}
    where {$\zeta_1,\ldots,\zeta_J$} are independent and {$\zeta_j$} is of exponential distribution
    with mean
    }
    \begin{equation}\label{eq:Zstar-limit-means}
      m_j
      = \frac{E^{(j)}_{j,[1:J]} \Gamma (E^{(j)}_{j,[1:J]})'}{2\,G^{(j)}_{jj}} = \frac{{(u^{(j)})^\T} \Gamma {u^{(j)}}}{2(1-w_{jj})}
      \qquad j=1,\dots,J,
    \end{equation}
    where the last equality follows from \eqref{eq: uw} {with $u^{(j)}$ defined in \eqref{eq: u}}.
\end{proposition}

Comparing {Proposition}~\ref{prop:Zstar-exp} with Theorems~\ref{thm:GJN} and \ref{prop: multi-SRBM}, we
see that both the matching-rate initial condition in
Assumption~\ref{assmpt: initial} and the lowest-rate initial condition in
Assumption~\ref{assmpt: conventional} {lead to the same product-form exponential limiting
distribution as $t\to\infty$.} In
particular, the mean vector $(m_1,\dots,m_J)$ in
\eqref{eq:Zstar-limit-means} coincides with that of the product-form
exponential stationary distribution identified under the matching-rate
initial condition, even though the corresponding diffusion limits are
quite different. In the Markov setting of Theorem~\ref{thm:GJN}, this
stationary distribution coincides with the   {limiting distribution} and
is independent of the initial state, and our two heavy-traffic limits are
consistent with this remarkable picture: in this multi-scale heavy traffic regime, the initial condition
influences the  {dynamics} of the diffusion limit but not its {limiting distribution}.

\subsection{SRBMs under lowest-rate initial condition} \label{sec: SRBM cv}

Since the asymptotic strong approximation in Proposition \ref{prop: SA} is still valid under the lowest-rate initial condition, it is sufficient to prove the functional limit result of SRBMs analogous to Theorem \ref{prop: multi-SRBM}.

We now consider the SRBM family from Section \ref{sec: SRBM} but under the following lowest-rate initial condition.
\begin{assumption}[Lowest-rate initial condition for SRBMs]
    \label{assmpt: conventional SRBM}
    There exists an $\R_+^J$-valued random vector $\tilde \xi$ such that 
    \begin{equation*}
        \left( \gamma_1(r)\tilde Z_1^\uu(0), \gamma_1(r)\tilde Z_2^\uu(0), \ldots, \gamma_1(r)\tilde Z_J^\uu(0) \right)  \Longrightarrow \tilde\xi \quad \text{as } r\to 0.
    \end{equation*}
\end{assumption}

In the following theorem, we establish that under Assumption \ref{assmpt: conventional SRBM}, the sequence of SRBMs still converges weakly but to a different limit process as $r\to 0$ when $J\geq 2$.

\begin{theorem}\label{prop: multi-SRBM conventional}
    Suppose Assumptions \ref{assmpt: multiscaling} and \ref{assmpt: conventional SRBM}  hold. Then
    \begin{equation*}
        \left\{ \left( \gamma_1(r)\tilde Z_1^\uu(t/\gamma_1^2(r)), \gamma_2(r)\tilde Z_2^\uu(t/\gamma_2^2(r)), \ldots, \gamma_J(r)\tilde Z_J^\uu(t/\gamma_J^2(r)) \right);~t\geq 0 \right\} \Longrightarrow \tilde Z^* \ \text{as } r\to 0,
    \end{equation*}
    where $\tilde Z^*=\{\tilde Z^*(t);~t\geq 0\}$ is a $J$-dimensional process whose components are mutually independent. Specifically, each coordinate {process} $\tilde Z_j^*$ is the first coordinate {process} of a $(J-j+1)$-dimensional SRBM {with initial state $\tilde\xi^{(j)}$, drift vector $-G^{(j)}_{[j:J],j}$, covariance matrix $(E^{(j)} \Gamma (E^{(j)})')_{[j:J],[j:J]}$, and reflection matrix $G^{(j)}_{[j:J], [j:J]}$.}
    {Here,} the $(J-j+1)$-dimensional vector $\tilde\xi^{(j)}$ is defined as $\tilde\xi$ for $j=1$ and $\mathbf{0}$ for $j>1$, and $E^{(j)}$ and $G^{(j)}$ are defined in \eqref{eq: Gauss elim}.
\end{theorem}

\section{Blockwise Multi-scale Heavy Traffic} \label{sec: block}

In this section, we consider the family of GJNs in blockwise multi-scale heavy traffic, where the network's stations are partitioned into $K$ blocks for some integer $K$ with $1 \le K \le J$. In this regime, the stations of different blocks approach heavy traffic at different rates, but the stations within the same block share a common rate. {The analysis is the block-level analogue of the scale-separation argument developed above: different blocks decouple asymptotically, while the stations within each block retain a multidimensional SRBM interaction. The same Gaussian-elimination idea separates the reflection effects across block boundaries, while preserving the reflected dynamics within each block.}

Specifically, we consider a partition of the index set $\{1,\ldots,J\}$ as follows:
\begin{equation*}
    \left( 1,2,\ldots, J \right) = 
    \Big( \underbrace{1=\alpha_1,\ldots, \beta_1}_{\text{block 1}}, \underbrace{\alpha_2, \ldots, \beta_2}_{\text{block 2}}, \ldots, \underbrace{\alpha_K, \ldots, \beta_K=J }_{\text{block K}}\Big).
\end{equation*}
We denote the $k$th block by $A_k = \{\indexa{k},\ldots,\indexb{k}\}$ for $k\in \K \defi \{1,\ldots,K\}$, and then the $k$th block contains $|A_k| = \indexb{k}-\indexa{k}+1$ elements. 
{We adopt the boundary conventions $\beta_0:=0$ and $\alpha_{K+1}:=J+1$, and interpret $[1:\beta_0]$ and $[\alpha_{K+1}:J]$ as empty index sets.}

\begin{assumption}[Blockwise multi-scale heavy traffic] \label{assmpt: blockscale}
    For all $r \in (0,1)$,
    \begin{equation*}
        \big( \mu^\uu - \lambda \big)_{A_k} = \gamma_k(r) b^{(k)} \quad \text{ for all } k \in \K,
    \end{equation*}
    where $b^{(k)}$ is a positive vector in $\R^{|A_k|}$.
\end{assumption}
For instance, $ \mu^\uu - \lambda =(r, 2r, 3r^2, 2r^2, r^2, 2r^3)$ means that $A_1=\{1,2\}$, $A_2=\{3,4,5\}$, with $A_3=\{6\}$, $b^{(1)} = [1,2]^\T$, $b^{(2)} = [3,2,1]^\T$, $b^{(3)}=2$ and $\gamma_1(r)=r$, $\gamma_2(r)=r^2$, $\gamma_3(r)=r^3$. Consequently, both multi-scale heavy traffic and conventional heavy traffic are special cases of {blockwise multi-scale} heavy traffic.

Similarly, the convergence of the initial states plays a crucial role in determining the limit process of $Z^{(r)}$ as $r \to 0$. The matching-rate initial condition for this regime is defined in Assumption \ref{assmpt: initial block} below. We also consider the lowest-rate initial condition from Assumption \ref{assmpt: conventional}. Their proofs are given in Appendix~\ref{sec: block heavy traffic proof}.

\begin{assumption}[Matching-rate initial condition] \label{assmpt: initial block}
    There exists an $\R_+^J$-valued   random vector $ \xi$ {{whose block subvectors $\xi_{A_1},\ldots,\xi_{A_K}$ are mutually independent}} such that
    \begin{equation*}
        \left( \gamma_1(r)Z^\uu_{A_1}(0),~\gamma_2(r)Z^\uu_{A_2}(0),~\ldots,~\gamma_K(r)Z^\uu_{A_K}(0) \right)  \Longrightarrow \xi \quad \text{as } r\to 0,
    \end{equation*} 
    and for all $k=2,\ldots,K$,
    \begin{equation*}
        \gamma_{k-1}(r) Z_{A_k}^\uu (0) \Longrightarrow \infty \quad \text{as } r\to 0.
    \end{equation*}
\end{assumption}

We demonstrate that the family of GJNs converges weakly to a limit process under the matching-rate initial condition and the lowest-rate initial condition as follows.

\begin{theorem}\label{thm:GJN block}
    Suppose Assumptions \ref{assmpt: moment}, \ref{assmpt: blockscale} and \ref{assmpt: initial block} hold. Then
    \begin{equation*} 
        \left\{ \left( \gamma_1(r)Z_{A_1}^\uu(t/\gamma_1^2(r)), \gamma_2(r)Z_{A_2}^\uu(t/\gamma_2^2(r)), \ldots, \gamma_K(r)Z_{A_K}^\uu(t/\gamma_K^2(r)) \right);~t\geq 0 \right\} \Longrightarrow Z^* \ \text{as } r\to 0,
    \end{equation*}
    where $Z^*=\{Z^*(t);~t\geq 0\}$ is a $J$-dimensional diffusion process.
    Furthermore, all corresponding blocks of $Z^*$ are mutually independent, and each block $Z^*_{A_k}$ for $k\in\K$ is a $|A_k|$-dimensional SRBM {with initial state $\xi_{A_k}$, drift vector $-G^{(\indexa{k})}_{A_k,A_k}b^{(k)}$, covariance matrix $(E^{(\indexa{k})} \Gamma (E^{(\indexa{k})})')_{A_k,A_k}$, and reflection matrix $G^{(\indexa{k})}_{A_k,A_k}$.}
    {Here,} $\indexa{k}$ is the lowest index in $A_k$, matrices $G$ and $E$ are defined in \eqref{eq: Gauss elim} and $\Gamma$ is given by \eqref{eq: SA_Gamma}.
\end{theorem}

In the following theorem, we consider the lowest-rate initial condition where all initial queue lengths are of the magnitude $1/\gamma_1(r)$.

\begin{theorem}\label{thm:GJN block cv}
    Suppose Assumptions \ref{assmpt: moment}, \ref{assmpt: conventional} and \ref{assmpt: blockscale} all hold. Then
    \begin{equation*} 
        \left\{ \left( \gamma_1(r)Z_{A_1}^\uu(t/\gamma_1^2(r)), \gamma_2(r)Z_{A_2}^\uu(t/\gamma_2^2(r)), \ldots, \gamma_K(r)Z_{A_K}^\uu(t/\gamma_K^2(r)) \right);~t\geq 0 \right\} \Longrightarrow Z^* \ \text{as } r\to 0,
    \end{equation*}
    where $Z^*=\{Z^*(t);~t\geq 0\}$ is a $J$-dimensional process.
    Furthermore, all corresponding blocks of $Z^*$ are mutually independent, and each block $Z^*_{A_k}$ for $k\in \K$ is the first $|A_k|$ coordinate {processes} of a $(J-\alpha_k+1)$-dimensional SRBM {with initial state $\xi^{(k)}$, drift vector $-G^{(\indexa{k})}_{[\indexa{k},J],A_k}b^{(k)}$, covariance matrix $(E^{(\indexa{k})} \Gamma (E^{(\indexa{k})})')_{[\indexa{k},J],[\indexa{k},J]}$, and reflection matrix $G^{(\indexa{k})}_{[\indexa{k},J],[\indexa{k},J]}$.}
    {Here,} $\indexa{k}$ is the lowest index in $A_k$, {for each $k\in\K$, $\xi^{(k)}$ is the $(J-\indexa{k}+1)$-dimensional vector defined as $\xi$ for $k=1$ and $\mathbf{0}$ for $k>1$}, matrices $G$ and $E$ are defined in \eqref{eq: Gauss elim} and $\Gamma$ is given by \eqref{eq: SA_Gamma}.
\end{theorem}

Analogous limit theorems can be established for blockwise multi-scaling SRBMs under both matching-rate and lowest-rate initial conditions.

\section{Conclusion}

We investigate the functional limits (process-level convergence) of generalized Jackson networks in a multi-scale heavy traffic regime where stations approach full utilization at distinct, separated rates. Our main result shows that the appropriately scaled queue length processes converge weakly to a limit process $Z^*$ whose coordinate processes are mutually independent, revealing the underlying dynamic mechanism behind the asymptotic independence previously observed only in stationary distributions. The specific dynamics of the limit process depend critically on the initial conditions, but the {limiting distributions} are all the same {as $t \to \infty$}, which is a product-form exponential distribution on $\R_+^J$. We further extend the analysis to a blockwise multi-scale heavy traffic regime, obtaining functional limits that exhibit blockwise independence. These results demonstrate that the decoupling of the dynamics is forced by the separated timescales.

{The functional limits also provide approximations for transient performance. Under the matching-rate initial condition, the queue length process at each station, observed on its own timescale, can be approximated by a one-dimensional reflected Brownian motion, whose transient distribution admits explicit expressions. This yields tractable approximations for time-dependent performance measures of GJNs, which are otherwise analytically intractable. Under the lowest-rate initial condition, the limit process yields analogous approximations {in terms of the first coordinate process of some} multi-dimensional SRBMs.}

{Two} open problems are left for future work. First, we do not establish steady-state convergence. Since process-level convergence does not generally guarantee convergence of stationary distributions, it is natural to ask whether an interchange-of-limits argument can be developed for this multi-scale heavy traffic regime. Our functional limit results assume a $(2+\varepsilon)$th moment condition on the interarrival and service times, whereas the existing steady-state analyses of \citet{DaiGlynXu2023} and \citet{GuanChenDaiGlyn2024} impose a substantially higher requirement of $(J+1+\varepsilon)$th moments. Whether one can lower this requirement toward $(2+\varepsilon)$th moments while still obtaining steady-state convergence remains open.

{Second}, {extending the functional limits beyond GJNs is an interesting direction. Multiclass queueing networks are a natural candidate, for which \citet{DaiHuo2024} established the steady-state product-form limit. The main difficulties are to establish an analogous asymptotic strong approximation and to extend the scale-separation analysis, as the reflection matrices of multiclass networks belong to the broader class of $\caP$-matrices while parts of our analysis rely on the $\caM$-matrix structure.}

\bibliographystyle{informs2014} 
\bibliography{sample}  

@book {Baldi,
 AUTHOR = {Baldi, Paolo},
     TITLE = {Probability---an introduction through theory and exercises},
    SERIES = {Universitext},
 PUBLISHER = {Springer, Cham},
      YEAR = {2023},
     PAGES = {ix+389},
}

@book{jacod2013limit,
  title={Limit theorems for stochastic processes},
  author={Jacod, Jean and Shiryaev, Albert},
  volume={288},
  year={2013},
  publisher={Springer Science \& Business Media}
}

@book{karatzas2014brownian,
  title={Brownian motion and stochastic calculus},
  author={Karatzas, Ioannis and Shreve, Steven},
  year={2014},
  publisher={springer}
}

@article{KellaRamasubramanian2012,
  title={Asymptotic Irrelevance of Initial Conditions for Skorohod Reflection Mapping on the Nonnegative Orthant},
  author={Kella, Offer and Ramasubramanian, S.},
  journal={Mathematics of Operations Research},
  volume={37},
  number={2},
  pages={301--312},
  year={2012}
}

@article{KellaWhitt1996,
  author  = {Offer Kella and Ward Whitt},
  title   = {Stability and Structural Properties of Stochastic Storage Networks},
  journal = {Journal of Applied Probability},
  volume  = {33},
  number  = {4},
  pages   = {1169--1180},
  year    = {1996}
}

@book{SrikYing2014,
	addendum = {\href{https://mathscinet.ams.org/mathscinet-getitem?mr=MR3202391}{MR3202391}, \href{https://doi.org/10.1017/CBO9781139565844}{Crossref}},
	address = {Cambridge, UK},
	author = {Srikant, R. and Ying, Lei},
	publisher = {Cambridge University Press},
	title = {Communication Networks: An Optimization, Control and Stochastic Networks Perspective},
	year = {2014}}

@article{Sigm1990,
	addendum = {\href{https://mathscinet.ams.org/mathscinet-getitem?mr=MR1062580}{MR1062580}, \href{https://doi.org/10.1109/18.61110}{Crossref}},
	author = {Karl Sigman},
	journal = {Stochastic Processes and their Applications},
	pages = {11--25},
	title = {The stability of open queueing networks},
	volume = 35,
	year = 1990}

@article{Miya2017,
	author = {Miyazawa, Masakiyo},
	journal = {Advances in Applied Probability},
	mrclass = {60K25 (60J27 60K37)},
	mrnumber = {3631221},
	number = {1},
	pages = {182--220},
	title = {A unified approach for large queue asymptotics in a heterogeneous multiserver queue},
	volume = {49},
	year = {2017}}

@article{Stol1995,
	addendum = {\href{https://mathscinet.ams.org/mathscinet-getitem?mr=MR1403094}{MR1403094}, \href{http://math-mprf.org/journal/articles/id746/}{Crossref}},
	author = {Stolyar, A. L.},
	journal = {Markov Processes and Related Fields},
	mrclass = {60K25 (90B15)},
	mrnumber = {1403094},
	number = {4},
	pages = {491--512},
	title = {On the stability of multiclass queueing networks: a relaxed sufficient condition via limiting fluid processes},
	volume = {1},
	year = {1995}}

@article{DaiHarr1992,
author = {J. G. Dai and J. M. Harrison},
title = {Reflected {B}rownian Motion in an Orthant: numerical Methods for Steady-State Analysis},
volume = {2},
journal = {Annals of Applied Probability},
number = {1},
publisher = {Institute of Mathematical Statistics},
pages = {65--86},
year = {1992},
}

@book{Whit2002,
  title={Stochastic-process limits: an introduction to stochastic-process limits and their application to queues},
  author={Whitt, Ward},
  year={2002},
  publisher={Springer Science \& Business Media}
}

@article{Dai1995,
  title={On positive {H}arris recurrence of multiclass queueing networks: a unified approach via fluid limit models},
  author={Dai, J. G.},
  journal={Annals of Applied Probability},
  volume={5},
  number={1},
  pages={49--77},
  year={1995},
  publisher={Institute of Mathematical Statistics}
}

@article{ShenChenDaiDai2002,
	author = {Xinyang Shen and Hong Chen and J. G. Dai and Wanyang Dai},
	journal = {Queueing Systems},
	pages = {33--62},
	title = {The finite element method for computing The stationary distribution of an {SRBM} in a hypercube with applications to finite buffer queueing networks},
	volume = 42,
	year = 2002}

@article{BudhLee2009,
	author = {Budhiraja, Amarjit and Lee, Chihoon},
	journal = {Mathematics of Operations Research},
	mrclass = {60G35 (60K25 90B22)},
	mrnumber = {2542988},
	mrreviewer = {Gennady Shaikhet},
	number = {1},
	pages = {45--56},
	title = {Stationary distribution convergence for generalized {J}ackson networks in heavy traffic},
	volume = {34},
	year = {2009}}

@article{Kats2010,
	author = {Katsuda, Toshiyuki},
	journal = {Queueing Systems},
	mrclass = {60K25 (90B22 93E15)},
	mrnumber = {2652044},
	mrreviewer = {Gennady Shaikhet},
	number = {3},
	pages = {237--273},
	title = {State-space collapse in stationarity and its application to a multiclass single-server queue in heavy traffic},
	volume = {65},
	year = {2010}}

@phdthesis{Fodd1983,
	author = {M. Foddy},
	school = {Department of Mathematics, Stanford University},
	title = {Analysis of {B}rownian motion with drift, confined to a quadrant by oblique reflection},
	year = {1983}}

@article{ZhanZwar2008,
	author = {Zhang, Jiheng and Zwart, Bert},
	journal = {Queueing Systems},
	mrclass = {60K25 (68M20 90B22)},
	mrnumber = {2461617 (2010a:60335)},
	mrreviewer = {Hsing Paul Luh},
	number = {3-4},
	pages = {227--246},
	title = {Steady state approximations of limited processor sharing queues in heavy traffic},
	volume = {60},
	year = {2008}}

@article{GuanChenDaiGlyn2024,
	author = {Guang, Jin and Chen, Xinyun and Dai, J. G. and Glynn, Peter},
	journal = {arXiv preprint arXiv:2503.19710},
	title = {Asymptotic Product-form Steady-state Distribution for Semimartingale Reflecting {B}rownian Motion in Multi-scaling Regime},
	year = 2024}

@article{GamaZeev2006,
	author = {Gamarnik, David and Zeevi, Assaf},
	journal = {Annals of Applied Probability},
	mrclass = {60K25 (60J25 60J65)},
	mrnumber = {MR2209336},
	number = {1},
	pages = {56--90},
	title = {Validity of heavy traffic steady-state approximation in generalized {J}ackson networks},
	volume = {16},
	year = {2006}}

@article{BlanChenSiGlyn2021,
	author = {Blanchet, Jose and Chen, Xinyun and Si, Nian and Glynn, Peter W.},
	journal = {Stochastic Systems},
	number = {2},
	pages = {174--192},
	publisher = {INFORMS},
	title = {Efficient steady-state simulation of high-dimensional stochastic networks},
	volume = {11},
	year = {2021}}

@article{BlanChen2015,
	author = {Blanchet, Jose and Chen, Xinyun},
	journal = {Annals of Applied Probability},
	number = {6},
	pages = {3209--3250},
	title = {Steady-state simulation of reflected {B}rownian motion and related stochastic networks.},
	volume = {25},
	year = {2015}}

@article{HarrNguy1993,
	addendum = {\href{https://mathscinet.ams.org/mathscinet-getitem?mr=MR1218842}{MR1218842}, \href{https://doi.org/10.1007/bf01158927}{Crossref}},
	author = {Harrison, J. Michael and Nguyen, Vi\^{e}n},
	journal = {Queueing Systems. Theory and Applications},
	mrclass = {60K25 (90B22)},
	mrnumber = {1218842},
	number = {1-3},
	pages = {5--40},
	title = {Brownian models of multiclass queueing networks: current status and open problems},
	volume = {13},
	year = {1993}}

@article{Reim1984,
	author = {Martin I. Reiman},
	journal = {Mathematics of Operations Research},
	optmonth = {August},
	optnumber = {3},
	pages = {441-458},
	title = {Open queueing networks in heavy traffic},
	volume = {9},
	year = {1984}}

@article{DaiGlynXu2023,
	author = {Dai, J. G. and Glynn, P. and Xu, Y.},
	journal = {arXiv preprint arXiv:2304.01499},
	title = {Multi-scale heavy traffic steady-state convergence in generalized {Jackson} networks},
	year = 2023
}

@article{HarrWill1987,
	author = {Harrison, J. M. and Williams, R. J.},
	fjournal = {Stochastics},
	journal = {Stochastics},
	mrclass = {60K25 (60F17 60J65 90B22)},
	mrnumber = {912049},
	mrreviewer = {Martin I. Reiman},
	number = {2},
	pages = {77--115},
	title = {Brownian models of open queueing networks with homogeneous customer populations},
	volume = {22},
	year = {1987}
}

@article{HarrReim1981,
	author = {Harrison, J. M. and Reiman, M. I.},
	journal = {Annals of Probability},
	mrclass = {60J65},
	mrnumber = {606992},
	mrreviewer = {D. Kannan},
	number = {2},
	pages = {302--308},
	title = {Reflected {B}rownian motion on an orthant},
	volume = {9},
	year = {1981}}

@article{McDoSivaSushTsatWendWend2020,
title = {M-matrix and inverse {M}-matrix extensions},
author = {J. J. McDonald and R. Nandi and K. C. Sivakumar and P. Sushmitha and M. J. Tsatsomeros and E. Wendler and M. Wendler},
pages = {186--203},
volume = {8},
number = {1},
journal = {Special Matrices},
year = {2020},
}

@book{ChenYao2001,
	address = {New York},
	author = {H. Chen and D. Yao},
	publisher = {Springer-Verlag},
	title = {Fundamentals of queueing networks: performance, asymptotics, and optimization},
	year = 2001}

@book{Bill2013,
  title={Convergence of probability measures},
  author={Billingsley, P.},
  year={2013},
  publisher={John Wiley \& Sons}
}

@article{Miya2015,
  title={Diffusion approximation for stationary analysis of queues and their networks: a review},
  author={Miyazawa, Masakiyo},
  journal={Journal of the Operations Research Society of Japan},
  volume={58},
  number={1},
  pages={104--148},
  year={2015},
  publisher={The Operations Research Society of Japan}
}

@article{DaiYehZhou1997,
  title={The {QNET} method for re-entrant queueing networks with priority disciplines},
  author={Dai, J. G. and Yeh, Din-Horng and Zhou, Chen},
  journal={Operations Research},
  volume={45},
  number={4},
  pages={610--623},
  year={1997},
  publisher={INFORMS}
}

@article{Reim1988,
  title={A multiclass feedback queue in heavy traffic},
  author={Reiman, Martin I},
  journal={Advances in Applied Probability},
  volume={20},
  number={1},
  pages={179--207},
  year={1988},
  publisher={Cambridge University Press}
}

@article{ChenZhan2000,
  title={Diffusion approximations for some multiclass queueing networks with {FIFO} service disciplines},
  author={Chen, Hong and Zhang, Hanqin},
  journal={Mathematics of Operations Research},
  volume={25},
  number={4},
  pages={679--707},
  year={2000},
  publisher={INFORMS}
}

@article{ChenZhan1996,
  title={Diffusion approximations for re-entrant lines with a first-buffer-first-served priority discipline},
  author={Chen, Hong and Zhang, Hanqin},
  journal={Queueing Systems},
  volume={23},
  pages={177--195},
  year={1996},
  publisher={Springer}
}

@article{DaiKurt1995,
  title={A multiclass station with Markovian feedback in heavy traffic},
  author={Dai, J. G. and Kurtz, Thomas G},
  journal={Mathematics of Operations Research},
  volume={20},
  number={3},
  pages={721--742},
  year={1995},
  publisher={INFORMS}
}

@article{Pete1991,
  title={A heavy traffic limit theorem for networks of queues with multiple customer types},
  author={Peterson, William P},
  journal={Mathematics of Operations Research},
  volume={16},
  number={1},
  pages={90--118},
  year={1991},
  publisher={INFORMS}
}

@article{BramDai2001,
  title={Heavy traffic limits for some queueing networks},
  author={Bramson, Maury and Dai, Jim G},
  journal={Annals of Applied Probability},
  pages={49--90},
  year={2001},
  publisher={JSTOR}
}

@article{Will1998,
  title={Diffusion approximations for open multiclass queueing networks: sufficient conditions involving state space collapse},
  author={Williams, Ruth J},
  journal={Queueing systems},
  volume={30},
  pages={27--88},
  year={1998},
  publisher={Springer}
}

@article{Bram1998,
  title={State space collapse with application to heavy traffic limits for multiclass queueing networks},
  author={Bramson, Maury},
  journal={Queueing Systems},
  volume={30},
  pages={89--140},
  year={1998},
  publisher={Springer}
}

@article{BravDaiMiya2017,
  title={Heavy traffic approximation for the stationary distribution of a generalized {J}ackson network: The {BAR} approach},
  author={Braverman, Anton and Dai, Jim G and Miyazawa, Masakiyo},
  journal={Stochastic Systems},
  volume={7},
  number={1},
  pages={143--196},
  year={2017},
  publisher={INFORMS}
}

@article{BramDaiHarr2010,
	author = {Maury Bramson and J. G. Dai and J. Michael Harrison},
	journal = {Annals of Applied Probability},
	number = {2},
	pages = {753-783},
	title = {Positive recurrence of reflecting {Brownian} motion in three dimensions},
	volume = {20},
	year = {2010}}

@inproceedings{Will1995,
	address = {New York},
	author = {Ruth J. Williams},
	booktitle = {Stochastic Networks},
	editor = {Frank P. Kelly and Ruth J. Williams},
	pages = {125--137},
	publisher = {Springer},
	series = {The {IMA} Volumes in Mathematics and its Applications},
	title = {Semimartingale reflecting {B}rownian motions in the orthant},
	volume = 71,
	year = 1995}

@article{BravDaiMiya2024,
  title={The {BAR} approach for multiclass queueing networks with {SBP} service policies},
  author={Braverman, Anton and Dai, J. G. and Miyazawa, Masakiyo},
  journal={Stochastic Systems},
  volume={15},
  number={1},
  pages={1--49},
  year={2025},
  publisher={INFORMS}
}

@article{DaiHuo2024,
  title={Asymptotic Product-form Steady-state for Multiclass Queueing Networks with {SBP} Service Policies in Multi-scale Heavy Traffic},
  author={Dai, J. G. and Huo, Dongyan},
  journal={arXiv preprint arXiv:2403.04090},
  year={2024}
}

@article{Jack1957,
	title={Networks of waiting lines},
	author={Jackson, James R.},
	journal={Operations Research},
	volume={5},
	number={4},
	pages={518--521},
	year={1957},
	publisher={INFORMS}
}

@article{Jack1963,
	author = {James R. Jackson},
	journal = {Management Science},
	number = {1},
	pages = {131--142},
	publisher = {INFORMS},
	title = {Jobshop-like Queueing Systems},
	volume = {10},
	year = {1963}
}

@book{Mor2013,
	author = {Harchol-Balter, Mor},
	publisher = {Cambridge University Press},
	title = {Performance modeling and design of computer systems: queueing theory in action},
	year = {2013}}

@book{John1983,
  title={Diffusion approximations for optimal filtering of jump processes and for queueing networks},
  author={Johnson, Daniel Peter},
  year={1983},
  publisher={The University of Wisconsin-Madison}
}

@incollection{ChenMand1994,
  title={Hierarchical modeling of stochastic networks, Part {II}: Strong approximations},
  author={Chen, Hong and Mandelbaum, Avi},
  booktitle={Stochastic modeling and analysis of manufacturing systems},
  pages={107--131},
  year={1994},
  publisher={Springer}
}

@article{Horv1992,
  title={Strong approximations of open queueing networks},
  author={Horv{\'a}th, Lajos},
  journal={Mathematics of Operations Research},
  volume={17},
  number={2},
  pages={487--508},
  year={1992},
  publisher={INFORMS}
}

@article{Horv1984,
  title={Strong approximation of renewal processes},
  author={Horv{\'a}th, Lajos},
  journal={Stochastic Processes and their Applications},
  volume={18},
  number={1},
  pages={127--138},
  year={1984},
  publisher={Elsevier}
}

@book{Csor1981,
  title={Strong approximations in probability and statistics},
  author={Cs{\"o}rgo, Miklos and R{\'e}v{\'e}sz, P{\'a}l},
  year={1981},
  publisher={Academic Press}
}

@article{kriukov2025,
	author={Nikolai Kriukov and Krzysztof Debicki and Michel Mandjes},
	journal = {arXiv preprint arXiv:2503.16633},
	title={Functional limit theorems for {G}aussian-fed queueing network in light and heavy traffic}, 
	year = 2025}

\newpage

\appendix

\section{Computation of the variance of $Z^*$ and $\tilde Z^*$ in the proof of Theorem \ref{thm:GJN}} \label{sec: proof GJN sigma}

In this section, we show that the variances of $Z^*$ and $\tilde Z^*$ are indeed the same in Section~\ref{sec: proof GJN}.
We compute $\tilde{\sigma}^2_j$ as follows:
\begin{equation*}
    \begin{aligned}
        \tilde \sigma^2_j &= {(u^{(j)})^\T} \Gamma {u^{(j)}} = \sum_{k\in \J} \sum_{\ell\in \J} {u^{(j)}_k} \Gamma_{k\ell} {u^{(j)}_\ell} \\
        & = \sum_{k\in \J} \sum_{\ell\in \J}  {u^{(j)}_k}  {u^{(j)}_\ell} \Big( \alpha_k c_{e,k}^2 \delta_{k\ell} + \sum_{i\in\J} \lambda_i \left[ P_{i k}\left(\delta_{k\ell}-P_{i \ell}\right) + c_{s,i}^2 (\delta_{ik}-P_{ik})(\delta_{i\ell}-P_{i\ell}) \right] \Big)  \\
        & =  \sum_{k\in\J} \alpha_k{(u^{(j)}_k)^2} c_{e, k}^2  + \sum_{k\in\J}    {(u^{(j)}_k)^2}\sum_{i\in \J}\lambda_iP_{i k}-  \sum_{i\in \J} \lambda_i\sum_{k\in \J}P_{i k}{u^{(j)}_k} \sum_{\ell\in \J}P_{i \ell}{u^{(j)}_\ell} \\
        &\quad +\sum_{i\in \J}  \lambda_i c_{s,i}^2 \sum_{k\in \J} {u^{(j)}_k} (\delta_{ik}-P_{ik})\sum_{\ell\in \J} {u^{(j)}_\ell}(\delta_{i\ell}-P_{i\ell}).
    \end{aligned}
\end{equation*}
Since the reflection matrix in Proposition \ref{prop: SA} is $R=I-P^\T$,
by the definitions of $w$ in \eqref{eq: w} and ${u^{(j)}}$ in \eqref{eq: u},
\begin{equation} \label{eq: property}
    \sum_{k\in \J}P_{ik}{u^{(j)}_k} = w_{ij}, \quad \text{and} \quad -{u^{(j)}_i} + w_{ij} 
    = \begin{cases}
        0, & \text{if } i<j, \\
        -1 + w_{jj}, & \text{if } i=j, \\
        w_{ij}, & \text{if } i>j.
    \end{cases}
\end{equation}
Together with traffic equation in \eqref{eq:traffic}, we have
\begin{equation*}
    \begin{aligned}
    \tilde{\sigma}^2_j &= \sum_{k\in\J} \alpha_k{(u^{(j)}_k)^2} c_{e, k}^2  + \sum_{k\in\J}    {(u^{(j)}_k)^2}\left( \lambda_k - \alpha_k \right)-  \sum_{i\in \J} \lambda_i w_{ij}^2 +\sum_{i\in \J}  \lambda_i c_{s,i}^2 ({u^{(j)}_i} - w_{ij})^2  \\
    &= \sum_{i\in\J} \alpha_i{(u^{(j)}_i)^2} c_{e, i}^2  - \sum_{i\in\J}  \alpha_i  {(u^{(j)}_i)^2} + \sum_{i\in\J}  \lambda_i  \left( {(u^{(j)}_i)^2}-   w_{ij}^2 \right) +\sum_{i\in \J}  \lambda_i c_{s,i}^2 ({u^{(j)}_i} - w_{ij})^2  \\
    & = \sum_{i<j} \alpha_iw_{ij}^2 c_{e, i}^2  + \alpha_j c_{e, j}^2  - \sum_{i<j}  \alpha_i w_{ij}^2 - \alpha_j + \lambda_j  \left( 1-   w_{jj}^2 \right)- \sum_{i>j}  \lambda_i  w_{ij}^2  \\
    &\quad  +\lambda_j c_{s,j}^2 (1 - w_{jj})^2  +\sum_{i>j}  \lambda_i c_{s,i}^2 w_{ij}^2 .
    \end{aligned}
\end{equation*}
We can see the terms regarding $c_{e,j}$ and $c_{s,j}$ are the same as the expressions $\sigma^2_j$ in Theorem~\ref{thm:GJN}. Hence, we can compute
\begin{equation*}
    \begin{aligned}
    \sigma^2_j-\tilde{\sigma}^2_j   & =  \sum_{i<j} \alpha_iw_{i j}\left(1-w_{i j}\right)+\sum_{i>j} \lambda_iw_{i j}\left(1-w_{i j}\right) +\lambda_j w_{j j}\left(1-w_{j j}\right)\\
    & \quad  -\Big( - \sum_{i<j}  \alpha_i w_{ij}^2 - \alpha_j + \lambda_j  \left( 1-   w_{jj}^2 \right)- \sum_{i>j}  \lambda_i  w_{ij}^2 \Big)  \\
    &  =  \sum_{i<j} \alpha_iw_{i j} + \alpha_j+\sum_{i>j} \lambda_iw_{i j} -\lambda_j\left(1-w_{j j}\right) =  \sum_{i\in \J} \alpha_i {u^{(j)}_i}+\sum_{i\in \J} \lambda_i \left( - {u^{(j)}_i} + w_{ij} \right)=0,
    \end{aligned}
\end{equation*}
where the last equality follows from the traffic equation in \eqref{eq:traffic} and the property in \eqref{eq: property}:
\begin{equation*}
    \sum_{i\in \J} \alpha_i {u^{(j)}_i}  =  \sum_{i\in \J} \Big( \lambda_i - \sum_{k\in \J}\lambda_k P_{ki} \Big) {u^{(j)}_i}= \sum_{i\in \J} \lambda_i {u^{(j)}_i} - \sum_{k\in \J}\lambda_k \sum_{i\in \J} P_{ki} {u^{(j)}_i} = \sum_{i\in \J} \lambda_i {u^{(j)}_i} - \sum_{k\in \J}\lambda_k w_{kj}.
\end{equation*}

This proves that two variances are the same. This completes the proof of Theorem~\ref{thm:GJN}.

\section{Proof of technical lemmas in Section \ref{sec: proof SA}} \label{sec: lem proof SA}

The proof of Lemma \ref{lemma: SA_diffusion} depends on the following lemma, which shows that the primitive processes $A,  {S}, \psi$ have the strong approximation by Brownian motions.
\begin{lemma}[FSA, Theorem 2.1 of \citet{Horv1984} or Theorem 5.14 of \citet{ChenYao2001}] \label{lemma: SA}
    Suppose that Assumption \ref{assmpt: moment} holds. If the underlying probability space is rich enough, there exist $J + 2$ independent $J$-dimensional standard Brownian motions $W^{(e)}$, $W^{(s)}$ and $W^{(\ell)}$, $\ell=1,\ldots,J$ on it and a set $\Omega_1 \subseteq \Omega$ with $\Prob(\Omega_1)=1$, such that for any $\omega\in \Omega_1$,  as $T\to \infty$,
    \begin{equation}
        \sup_{0\leq t \leq T} \Norm{A(t, \omega) - \alpha t - (\Gamma^{(e)})^{1/2}W^{(e)}(t, \omega)}  = o(T^{1/(2+\varepsilon)}), \label{eq: SAp_A}
    \end{equation}
    \begin{equation}
        \sup_{0\leq t \leq T} \Norm{ {S}(t, \omega) - et - (\Gamma^{(s)})^{1/2}W^{(s)}(t, \omega)}  = o(T^{1/(2+\varepsilon)}), \label{eq: SAp_S}
    \end{equation}
    and for any $T\geq 1$,
    \begin{equation}
        \sup_{0\leq t \leq T} \Norm{\psi_{\ell}(t, \omega) - p_{\ell}t - (\Gamma^{(\ell)})^{1/2}W^{(\ell)}(t, \omega)}   \leq \frac{C_1(\omega)}{4} \log T, \label{eq: SAp_psi}
    \end{equation}
    where $C_1(\omega)$ is a positive random variable.
    The covariance matrices of $J+2$ Brownian motions and the vector $p_\ell$ are given in Lemma~\ref{lemma: SA_diffusion}.
\end{lemma}

\begin{proof}[Proof of Lemma~\ref{lemma: SA_diffusion}.]
    By multiplying $\gamma_k(r)$ on both sides and respectively replacing $t$ and $T$ with $t/\gamma_k^2(r)$ and $T/\gamma_k^2(r)$ in \eqref{eq: SAp_A}, we obtain that for any fixed $T>0$ and $\omega\in \Omega_1$, as $r\to 0$,
    \begin{equation*}
        \begin{aligned}
            &\sup_{0\leq t/\gamma_k^2(r) \leq T/\gamma_k^2(r)} \Norm{\gamma_k(r)A(t/\gamma_k^2(r), \omega) - \alpha t/\gamma_k(r) - (\Gamma^{(e)})^{1/2}\gamma_k(r)W^{(e)}(t/\gamma_k^2(r), \omega)}\\
            &\quad   = o\left(\gamma_k(r)(T/\gamma_k^2(r))^{1/(2+\varepsilon)}\right) = o\left(T^{1/(2+\varepsilon)}\gamma_k^{\varepsilon/(2+\varepsilon)}(r)\right) = o\left(\gamma_k^{\varepsilon/(2+\varepsilon)}(r)\right),
        \end{aligned}
    \end{equation*}
    which completes the proof of \eqref{eq: ASA_A}.
    The proof of \eqref{eq: ASA_S} is similarly derived from \eqref{eq: SAp_S} in Lemma~\ref{lemma: SA}, so we omit it. 
    
    Similarly, by multiplying $\gamma_k(r)$ on both sides and respectively replacing $t$ and $T$ with $t/\gamma_k^2(r)$ and $T/\gamma_k^2(r)$ in \eqref{eq: SAp_psi}, we conclude that for any fixed $T>0$ and $\omega\in \Omega_1$, 
    \begin{equation*}
        \begin{aligned}
            &\sup_{0\leq t/\gamma_k^2(r) \leq T/\gamma_k^2(r)} \Norm{\gamma_k(r)\psi_{\ell}(t/\gamma_k^2(r), \omega) - p_{\ell}t/\gamma_k(r) - (\Gamma^{(\ell)})^{1/2}\gamma_k(r)W^{(\ell)}(t/\gamma_k^2(r), \omega)}\\
            &\quad   \leq \frac{1}{4}C_1(\omega)\gamma_k(r)\log (T/\gamma_k^2(r))  \leq C_1(\omega)\gamma_k(r)\log (1/\gamma_k(r)),
        \end{aligned}
    \end{equation*}
    for any $r\in (0, \gamma_k^{-1}(T^{-1/2} \wedge 1))$.
\end{proof}

The proof of Lemma \ref{lemma: SA_fluid} depends on the following lemma, which shows that the primitive processes $A,  {S}, \psi$ have the FLIL bound and then the busy time process $B$ also has the FLIL bound.
\begin{lemma}[FLIL, Theorem 5.13 of \citet{ChenYao2001}] \label{lemma: FLIL}
    Suppose that Assumption~\ref{assmpt: moment} holds. There exists a set $\Omega_2 \subseteq \Omega$ with $\Prob(\Omega_2)=1$, such that for any $T \geq e$, 
    \begin{align}
        \sup_{0\leq t \leq T} \Norm{A(t, \omega) - \alpha t}  &\le C_0(\omega)\sqrt{T\log\log T}, \label{eq: LILp_A}\\
        \sup_{0\leq t \leq T} \Norm{ {S}(t, \omega) - et} &\le C_0(\omega)\sqrt{T\log\log T}, \label{eq: LILp_S}\\
        \sup_{0\leq t \leq T} \Norm{\psi_{\ell}(t, \omega) - p_{\ell}t}  &\le C_0(\omega)\sqrt{T\log\log T}. \label{eq: LILp_psi}
    \end{align}
\end{lemma}

\begin{proof}[Proof of Lemma~\ref{lemma: SA_fluid}.]
    By multiplying $\gamma_k^2(r)$ on both sides and respectively replacing $t$ and $T$ with $t/\gamma_k^2(r)$ and $T/\gamma_k^2(r)$ in \eqref{eq: LILp_A}, we obtain that for any fixed $T>0$ and $\omega\in \Omega_2$,
    \begin{equation*}
        \begin{aligned}
            \sup_{0\leq t/\gamma_k^2(r) \leq T/\gamma_k^2(r)} \Norm{\gamma_k^2(r) A(t/\gamma_k^2(r), \omega) - \alpha t}  &\le C_0(\omega)\gamma_k^2(r)\sqrt{T/\gamma_k^2(r)\log\log(T/\gamma_k^2(r))}\\
            &\leq 2\sqrt{T}C_0(\omega) \gamma_k(r)\sqrt{ \log\log (1/\gamma_k(r))},
        \end{aligned}
    \end{equation*}
    for any $r\in (0, r_2(T))$ with $r_2(T)\defi\gamma_k^{-1}(\exp(-(1+\sqrt{\log (1 \vee Te)})/2) \wedge \sqrt{(1\wedge T)/e})$.  Hence, \eqref{eq: LIL_A} holds by setting $C_2(\omega, T)=2\sqrt{T}C_0(\omega)$ for any $r\in (0,r_2(T))$.
    The proofs of \eqref{eq: LIL_S} and \eqref{eq: LIL_psi} are similarly derived by \eqref{eq: LILp_S} and \eqref{eq: LILp_psi} in Lemma~\ref{lemma: FLIL}, so we omit their proof details.

    To prove \eqref{eq: LIL_B}, we need to establish the Skorokhod reflection mapping for fluid scaled GJN as follows:
    \begin{equation*}
        \scale{\bar{Z}}{k}{{}} = \Phi(\scale{\bar{X}}{k}{{}}; R), \quad \scale{\bar{Y}}{k}{{}} = \Psi(\scale{\bar{X}}{k}{{}}; R),
    \end{equation*}
    and for any $j\in \J$,
    \begin{align}
        \scale{\bar{X}}{k}{{}}(t) &= \gamma_k^2(r) Z^\uu(0) - R(\mu^\uu-\lambda)t + \scale{\bar{V}}{k}{{}}(t), \notag\\
        \scale{\bar{Y}}{k}{{j}}(t)&=  \mu_j^\uu \left( t- \scale{\bar{B}}{k}{{j}}(t) \right), \label{eq: LY}\\
        \scale{\bar{V}}{k}{{j}}(t) &= \left[\scale{\bar{A}}{k}{j}(t)-\alpha_j t\right] \notag \\
        &\quad +\sum_{i\in\J} \left( P_{i j} - \delta_{i j} \right)\left[ \scale{\bar{S}}{k}{i}\left(\mu_i^\uu \scale{\bar{B}}{k}{i}(t)\right)-\mu_i^\uu  \scale{\bar{B}}{k}{i}(t)\right] \notag\\
        &\quad +\sum_{i\in\J}\left[\scale{\bar{\psi}}{k}{i,j}\left(\scale{\bar{S}}{k}{i} \left(\mu_i^\uu \scale{\bar{B}}{k}{i}(t)\right)\right)-P_{i j}  \scale{\bar{S}}{k}{i}\left(\mu_i^\uu \scale{\bar{B}}{k}{i}(t)\right)\right]. \notag
    \end{align}
    
    It follows from \eqref{eq: LIL_A} that
    \begin{equation*}
        \sup_{0\leq t \leq T }\abs{\scale{\bar{A}}{k}{j}(t)-\alpha_j t}\leq C_2(\omega, T)\gamma_k(r)\sqrt{\log\log(1/\gamma_k(r))}
    \end{equation*}
    for any $r\in (0,r_2(T))$ and $\omega\in \Omega_2$.
    By \eqref{eq: B_bound} and replacing $T$ with $(\lambda_{\max} + c_0)T$ in \eqref{eq: LIL_S}, we have
    \begin{equation*}
        \begin{aligned}
            &\sup_{0\leq t \leq T } \abs{\sum_{i\in\J} \left( P_{i j} - \delta_{i j} \right)\left[ \scale{\bar{S}}{k}{i}\left(\mu_i^\uu \scale{\bar{B}}{k}{i}(t)\right)-\mu_i^\uu  \scale{\bar{B}}{k}{i}(t)\right]}\\
            & \quad \leq \sum_{i\in\J} \abs{P_{i j} - \delta_{i j}}C_2(\omega, (\lambda_{\max} + c_0)T)\gamma_k(r)\sqrt{\log\log(1/\gamma_k(r))}\\
            & \quad \leq JC_2(\omega, (\lambda_{\max} + c_0)T)\gamma_k(r)\sqrt{\log\log(1/\gamma_k(r))}
        \end{aligned}
    \end{equation*}
     for any $r\in (0,r_2((\lambda_{\max} + c_0)T))$ and $\omega\in \Omega_2$. 
    By \eqref{eq: V_SB_max} and replacing $T$ with $C_2''(\omega, T)$ in \eqref{eq: LIL_psi}, we have
    \begin{equation*}
        \begin{aligned}
            &\sup_{0\leq t \leq T} \abs{\sum_{i\in\J}\left[\scale{\bar{\psi}}{k}{i,j}\left(\scale{\bar{S}}{k}{i} \left(\mu_i^\uu \scale{\bar{B}}{k}{i}(t)\right)\right)-P_{i j}  \scale{\bar{S}}{k}{i}\left(\mu_i^\uu \scale{\bar{B}}{k}{i}(t)\right)\right]} \\
            & \quad \leq \sum_{i\in \J}C_2(\omega, C_2''(\omega, T))\gamma_k(r)\sqrt{\log\log(1/\gamma_k(r))}\\
            &\quad =JC_2(\omega, C_2''(\omega, T))\gamma_k(r)\sqrt{\log\log(1/\gamma_k(r))}
        \end{aligned}
    \end{equation*}
    for any $r\in (0,r_2(C_2''(\omega, T)))$ and $\omega\in \Omega_2$. For any fixed $T>0$ and $\omega\in \Omega_2$, there exists a random variable $C_p(\omega, T)$ so that for any $r\in [r_2(C_2''(\omega, T)), \gamma_k^{-1}(1/e))$,
    \begin{equation*}
        \begin{aligned}
            &\sup_{0\leq t \leq T} \abs{\sum_{i\in\J}\left[\scale{\bar{\psi}}{k}{i,j}\left(\scale{\bar{S}}{k}{i} \left(\mu_i^\uu \scale{\bar{B}}{k}{i}(t)\right)\right)-P_{i j}  \scale{\bar{S}}{k}{i}\left(\mu_i^\uu \scale{\bar{B}}{k}{i}(t)\right)\right]} \leq C_p(\omega, T)\\
            & \quad \leq C_p(\omega, T)(\gamma_k(r)/\gamma_k(r_2(C_2''(\omega, T))))\sqrt{\log\log((1/e)^{2}/\gamma_k(r))}.
        \end{aligned}
    \end{equation*}
    Hence, for the fixed $T>0$, there {exists} a random variable $C_b(\omega, T)$ such that for any $r\in (0,\gamma_k^{-1}(1/e))$,
    \begin{equation*}
        \begin{aligned}
            &\sup_{0\leq t \leq T} \abs{\sum_{i\in\J}\left[\scale{\bar{\psi}}{k}{i,j}\left(\scale{\bar{S}}{k}{i} \left(\mu_i^\uu \scale{\bar{B}}{k}{i}(t)\right)\right)-P_{i j}  \scale{\bar{S}}{k}{i}\left(\mu_i^\uu \scale{\bar{B}}{k}{i}(t)\right)\right]} \\
            & \quad \leq C_b(\omega, T)\gamma_k(r)\sqrt{\log\log(1/\gamma_k(r))}.
        \end{aligned}
    \end{equation*}
    
    Therefore, for any $r\in (0, r_v(T))$ with $r_v(T):=r_2(T) \wedge r_2((\lambda_{\max} + c_0)T) \wedge \gamma_k^{-1}(1/e)$ and $\omega\in \Omega_2$,
    $$\sup_{0\leq t \leq T}\abs{\scale{\bar{V}}{k}{{j}}(t, \omega)}\leq {C_v(\omega, T)}\gamma_k(r)\sqrt{\log\log(1/\gamma_k(r))},$$
    where $C_v(\omega, T):=\left( C_2(\omega, T) + JC_2(\omega, (\lambda_{\max} + c_0)T) +  C_b(\omega, T)\right)$.
    Hence, we have
    \begin{equation*}
        \begin{aligned}
            &\sup_{0\leq t \leq T} \Norm{\scale{\bar{X}}{k}{}(t, \omega) - \gamma_k^2(r) Z^\uu(0, \omega)} \leq \sup_{0\leq t \leq T}\Norm{R(\mu^\uu-\lambda)t} + \sup_{0\leq t \leq T}\Norm{\scale{\bar{V}}{k}{{j}}(t, \omega)}\\
            &\qquad \leq \max_{i\in \J}\abs{\sum_{j\in \J}R_{ij}c_0 \gamma_1(r)T} + C_v(\omega, T)\gamma_k(r)\sqrt{\log\log(1/\gamma_k(r))}
        \end{aligned}
    \end{equation*}
    for any $r\in (0,r_v(T))$ and $\omega\in \Omega_2$. Since $\lim_{r\to 0} \gamma_k(r)/\gamma_1(r)=0$ for any $k\geq 2$, there exists a constant $r_2'(T)\in (0, r_v(T))$ such that $\gamma_k(r)\sqrt{\log\log(1/\gamma_k(r))}\leq \gamma_1(r)\sqrt{\log\log(1/\gamma_1(r))}$ for any $r\in (0,r_2'(T))$. Hence, we have
    \begin{equation*}
        \begin{aligned}
            &\sup_{0\leq t \leq T} \Norm{\scale{\bar{X}}{k}{}(t, \omega) - \gamma_k^2(r) Z^\uu(0, \omega)} \\
            & \qquad \leq \Big( c_0T \max_{i\in \J}\sum_{j\in \J}\abs{R_{ij}}+ C_v(\omega, T) \Big)\gamma_1(r)\sqrt{\log\log(1/\gamma_1(r))}
        \end{aligned}
    \end{equation*}
    for any $r\in (0,r_2'(T))$ and $\omega\in \Omega_2$.

    Therefore, the Lipschitz continuity of Skorokhod reflection mapping implies that for any $r\in (0,r_2'(T))$ and $\omega\in \Omega_2$,
    \begin{equation*}
        \begin{aligned}
            \sup_{0\leq t \leq T} \Norm{\scale{\bar{Y}}{k}{}(t, \omega)}  &= \sup_{0\leq t \leq T} \Norm{\scale{\bar{Y}}{k}{}(t, \omega) - \Psi(\gamma_k^2(r) Z^\uu(0, \omega); R)}\\
            & \leq \kappa(R) \Big( c_0T \max_{i\in \J}\sum_{j\in \J}\abs{R_{ij}}+ C_v(\omega, T) \Big) \gamma_1(r)\sqrt{\log\log(1/\gamma_1(r))},
        \end{aligned}
    \end{equation*}
    where  $\kappa(R)$ is defined in \eqref{eq: Lipschitz}.
    Hence, we can obtain \eqref{eq: LIL_B} by the relation between $\scale{\bar{Y}}{k}{}$ and $\scale{\bar{B}}{k}{}$ in \eqref{eq: LY}.      
\end{proof}

The proof of Lemma \ref{lemma: SA_BM} utilizes the following lemma.

\begin{lemma}[Theorem 1.2.1 of \citet{Csor1981}] \label{lemma: SA_BM_2}
    Let $a_{\tau}$ be a monotonically nondecreasing function of $\tau$ for which
    \begin{enumerate}
        \item[(i)] $0<a_\tau \leq \tau$,
        \item[(ii)] $\tau / a_\tau$ is monotonically  nondecreasing as $\tau$ increases.
    \end{enumerate}
    Then there exists a set $\Omega_2 \subseteq \Omega$ with $\Prob(\Omega_2)=1$ such that for any $\omega\in \Omega_2$
    $$
    \limsup_{\tau \rightarrow \infty} \sup _{0 \leq t \leq \tau-a_\tau} \sup_{0\leq s \leq a_\tau} \frac{\abs{W\left(t+s, \omega\right)-W(t, \omega)}}{\sqrt{2 a_\tau \left[\log \frac{\tau }{a_\tau }+\log \log \tau \right]}} = 1.
    $$
\end{lemma}

\begin{proof}[Proof of Lemma \ref{lemma: SA_BM}.]
    By  
    \begin{equation*} 
        \begin{aligned}
            &\left\{ (s,t): |s-t| \leq a_\tau, s, t \in [0, \tau]\right\} \\
            &\quad \subseteq \left\{ (s,t): 0 \leq s \leq \tau, 0\leq t-s\leq a_\tau\right\}  \cup \left\{ (s,t): 0 \leq t \leq \tau, 0\leq s-t\leq a_\tau\right\},
        \end{aligned}
    \end{equation*}
    and respectively replacing $\tau$, $t$ and $s$ with $\max\{t, s\}$, $\tau+a_\tau$ and $|t-s|$ in Lemma~\ref{lemma: SA_BM_2}, then there {exists} a constant $\tau_0$  such that for any $\tau\geq \tau_0$
    \begin{equation}
        \sup _{\substack{s, t \in [0, \tau]\\|s-t| \leq a_\tau}}\abs{W\left(s, \omega\right)-W(t, \omega)} \leq 2\sqrt{ a_\tau \left[\log \frac{\tau+a_\tau }{a_\tau }+\log \log (\tau+a_\tau) \right]}. \label{eq: SA_BM_2}
    \end{equation}

    Let $\tau = T/\gamma_k^2(r)$ and $a_\tau = (\gamma_1(r)/\gamma_k^2(r))C\sqrt{\log\log (1/\gamma_1(r))}$. Then, we can check that (i) $0< (\gamma_1(r)/\gamma_k^2(r))C\sqrt{\log\log (1/\gamma_1(r))} \leq T/\gamma_k^2(r)$ for any $r \in (0, \gamma_1^{-1}(2(T/C)^2 \wedge (1/e)))$, and (ii) for any $r \in (0,\gamma_1^{-1}(1/e))$,  $\tau / a_\tau  = T/(C \gamma_1(r)\sqrt{\log\log (1/\gamma_1(r))})$ is monotonically increasing as $\gamma_k(r)$ decreases, because $\gamma_k$ for $k\in \J$ are all strictly increasing.
    Therefore, by replacing $s$ and $t$ with $s/\gamma_k^2(r)$ and $t/\gamma_k^2(r)$, Lemma \ref{lemma: SA_BM_2} and \eqref{eq: SA_BM_2} indicate that for any $\omega \in \Omega_3$ and $r \in (0, r_3(T,C))$ with $r_3(T,C):=\gamma_1^{-1}((T/C) \wedge (C/2T) \wedge T \wedge \exp(-e/2)) \wedge \gamma_k^{-1}((2T)^{-1/3} \wedge \exp(-6) \wedge (T/\tau_0)^{1/2})$,
    \begin{equation*}
        \begin{aligned}
            &\sup _{\substack{s, t \in [0, T]\\|s-t| \leq C \gamma_1(r)\sqrt{\log\log (1/\gamma_1(r))}}}\abs{W\left(s/\gamma_k^2(r), \omega\right)-W(t/\gamma_k^2(r), \omega)} \\
            &\leq  2\sqrt{ C \frac{\gamma_1(r)}{\gamma_k^2(r)}\sqrt{\log\log \frac{1}{\gamma_1(r)}}\left(\log \frac{T+C \gamma_1(r)\sqrt{\log\log \frac{1}{\gamma_1(r)}}}{C \gamma_1(r)\sqrt{\log\log \frac{1}{\gamma_1(r)}}}+\log \log \left( \frac{T+C \gamma_1(r)\sqrt{\log\log \frac{1}{\gamma_1(r)}}}{\gamma_k^2(r)} \right)\right)} \\ 
            &\leq 4\gamma_k^{-1}(r)\sqrt{ C \gamma_1(r)\sqrt{\log\log {1}/{\gamma_1(r)}}\left[\log(1/ \gamma_1(r))+\log \log \left( 1/\gamma_k(r)\right)\right]}.
        \end{aligned}
    \end{equation*}
    
    Multiplying $\gamma_k(r)$ on both sides,  we have for any $r\in (0,r_3(T,C))$ and $\omega\in \Omega_3$
    \begin{equation*}
        \begin{aligned}
            &\sup _{\substack{s, t \in [0, T]\\|s-t| \leq C \gamma_1(r)\sqrt{\log\log (1/\gamma_1(r))}}}\abs{\gamma_k(r)W\left(s/\gamma_k^2(r), \omega\right)-\gamma_k(r)W(t/\gamma_k^2(r), \omega)}\\
            &\leq 4\sqrt{C}\sqrt{ \gamma_1(r)\sqrt{\log\log {1}/{\gamma_1(r)}}\left[\log(1/ \gamma_1(r))+\log \log \left( 1/\gamma_k(r)\right)\right]}.
        \end{aligned}
    \end{equation*}
    This completes the proof.         
\end{proof}

\section{Proof of technical lemma in Section \ref{sec: proof SRBM}} \label{sec: proof lemma k}

\begin{proof}[Proof of Lemma \ref{lemma: Skorokhod positive infty}.]
    {By} the one-dimensional Skorokhod reflection mapping, we have the following representation:
    \begin{equation*}
        cy^\uu(t)=\Psi\left(\left\{ u^\uu(t)-v^\uu(t)+a^\uu ;~t\geq 0 \right\}; 1\right) = \sup_{0\leq s \leq t} \left[-u^\uu(s)+v^\uu(s)-a^\uu \right]^+.
    \end{equation*}
    For any $\varepsilon>0$ and $T>0$, there exists $r_0\in (0,1)$ such that for all $r\in (0,r_0)$, we have
    \begin{equation*}
        \norm{u^\uu - u}_{T} \leq \varepsilon, \quad \norm{\tilde{v}^\uu- \tilde{v}}_{T} \leq \varepsilon \quad \text{ and } \quad a^\uu \geq 2\varepsilon + \sup_{0\leq s \leq T}\abs{u(s)} + \sup_{0\leq s \leq T}\abs{\tilde{v}(s)},
    \end{equation*}
    where $\norm{\cdot}_{T}$ denotes the uniform norm on $[0,T]$ defined in \eqref{eq: uniform norm}.
    Therefore, for all $r\in (0,r_0)$, for any $s\in[0,T]$, we have
    \begin{equation*}
        -u^\uu(s)+v^\uu(s)-a^\uu \leq -u(s) + \varepsilon + \tilde{v}(s) + \varepsilon - 2\varepsilon - \abs{u(s)} - \abs{\tilde{v}(s)} \leq 0.
    \end{equation*}
    Hence, for all $r\in (0,r_0)$, we have for any $t\in [0,T]$
    \begin{equation*}
        cy^\uu(t)= \sup_{0\leq s \leq t} \left[-u^\uu(s)+v^\uu(s)-a^\uu \right]^+ = \sup_{0\leq s \leq t} 0 = 0,
    \end{equation*}
    which implies that $y^\uu \to 0$ u.o.c.~as $r\to 0$.     
\end{proof}

\section{Proof of Convergence under Lowest-rate Initial Condition} \label{sec: conventional proof}

The proof of Theorem \ref{thm:GJN cv} is {the} same as that of Theorem \ref{thm:GJN} by utilizing the asymptotic strong approximation in Proposition~\ref{prop: SA}, so we omit it here.

{
\subsection{Proof of Proposition~\ref{prop:Zstar-exp}}
}

We now prove Proposition~\ref{prop:Zstar-exp}.
Since {Theorem~\ref{thm:GJN cv} states that the coordinate
processes of $Z^*$ are mutually independent}, it is sufficient to show that each coordinate
$Z_j^*$ of the limit process $Z^*$ has an
{exponential limiting distribution as $t\to\infty$}. Recall that
for each $j\in\J$, the coordinate process $Z_j^*$ is the first coordinate
of a $(J-j+1)$-dimensional SRBM {with initial state $\xi^{(j)}$, drift vector $-G^{(j)}_{[j:J],j}$, covariance matrix $(E^{(j)} \Gamma (E^{(j)})')_{[j:J],[j:J]}$, and reflection matrix $G^{(j)}_{[j:J],[j:J]}$.}
Therefore, {to prove 
\eqref{eq:Zstar-limiting-distribution},} it is enough to identify the
{limiting distribution as $t\to\infty$} of the first coordinate
of a generic SRBM whose drift is {the negative of} the first
column of its reflection matrix. This is done in the following
proposition.

\begin{proposition}\label{prop:main-exp}
    Consider a $J$-dimensional SRBM $Z^{{z}}$ starting from the
    {almost surely finite} initial state
    $Z^{{z}}(0)=z\in\R_+^J$
    with reflection matrix $R$, drift vector
    $\theta=-R\e{1}$, and covariance matrix $\Gamma$, where
    $\e{1}=(1,0,\ldots,0)^\T$.
    {Assume that $R$ is a nonsingular $\caM$-matrix and write}
    \[
      {
      Z^z=\Phi(z+X;R),\qquad
      X(t)=-R\e{1}t+LW(t),\qquad LL^\T=\Gamma,
      }
    \]
    {where $W$ is a standard $J$-dimensional Brownian motion.
    Then, as $t\to\infty$, $Z_1^z(t)$ converges in distribution to an
    exponential random variable with mean
    $\Gamma_{11}/(2R_{11})$.}
\end{proposition}

To prove Proposition~\ref{prop:main-exp}, we introduce the following lemma by the fluid analysis.

\begin{lemma}\label{lem:fluid}
{Consider an SRBM $Z$ with initial state $z$, drift vector $\theta=-R\e{1}$, covariance matrix $\Gamma$, and reflection matrix $R$, and let $Y$ be its regulator process. Then}
\begin{equation*}
    {\lim_{t\to\infty} \frac{Y(t)}{t} \aseq \e{1} }.
\end{equation*}
In particular,  {$\lim_{t\to\infty} {Y_1(t)}/{t} \aseq 1$ and $\lim_{t\to\infty}{Y_j(t)}/{t} \aseq 0$} for $j=2,\dots,J$.
\end{lemma}

\begin{proof}[Proof of Lemma~\ref{lem:fluid}]
Consider the fluid scaling process 
\begin{equation*}
    \bar Z^{(n)} := \{Z(nt)/n; t\geq 0\}, \quad \bar Y^{(n)} := \{Y(nt)/n; t\geq 0\}, \quad \bar X^{(n)} := \{({z+}X(nt))/n; t\geq 0\}.
\end{equation*}
 {Note that }
\begin{equation*}
    \bar Z^{(n)} = \bar X^{(n)} + R \bar Y^{(n)} \quad \text{and} \quad \bar{Y}^{(n)} = \Psi(\bar X^{(n)}; R).
\end{equation*}
 {Since} $\bar X^{(n)} = \{z / n + \theta t + L W(nt)/n;~t\geq 0\}\overset{a.s.}{\to} \{\theta t;~t\geq 0\}$ u.o.c.~as $n \to \infty$, by the Lipschitz continuity of the Skorokhod map, 
\begin{equation*}
     {\lim_{n\to \infty} \bar Y^{(n)} \aseq \bar Y\quad \text{ and } \quad \lim_{n\to \infty} \bar Z^{(n)} \aseq \bar Z \quad \text{u.o.c.}}
\end{equation*}
The limit process $(\bar Z,\bar Y)$ solves the deterministic Skorokhod problem:  {$\bar Z(t) = \bar X(t) + R \bar Y(t)$ and $\bar Y(t) = \Psi(\bar X(t); R)$,}
and input  {process} $\bar X(t)=\theta t$.  {By the uniqueness of the solution to the Skorokhod problem (see Theorem 1 of \citet{HarrReim1981}), we have $(\bar Z(t),\bar Y(t))=(0, t \e{1})$. In particular, $\lim_{n\to\infty} Y(n)/n = \lim_{n\to \infty} \bar Y^{(n)}(1)\aseq \bar Y(1)=\e{1}$.}

Since each component of $Y(\cdot)$ is continuous and nondecreasing, we have
\begin{equation*}
 {\frac{Y([t])}{[t]+1}\le \frac{Y(t)}{t}\le \frac{Y([t]+1)}{[t]}.}
\end{equation*}
 {It follows that $\lim_{t\to\infty} {Y(t)}/{t} \aseq \e{1}$. This completes the proof.}
\end{proof}

\begin{proof}[Proof of Proposition~\ref{prop:main-exp}]
    {We first take $z=0$ and abbreviate $Z^0$ and $Y^0$ as
    $Z$ and $Y$, respectively.} By the definition of an SRBM,
    \begin{equation*}
        Z_1(t) = -R_{11}t + \sum_{j=1}^J L_{1j}W_j(t) + R_{11}Y_1(t) + \sum_{j=2}^J R_{1j}Y_j(t),
    \end{equation*}
    where $L$ is the Cholesky decomposition of $\Gamma$ with $L L^\T = \Gamma$.
    Applying Itô's formula to $e^{\eta Z_1(t)}$ for any $\eta<0$ and taking expectations, we get
    \begin{align*}
        &\mathbb{E}\left[ \exp(\eta Z_1(t)) \right] - {1} = \int_0^t \left(\frac{1}{2}\Gamma_{11}\eta^2  - R_{11} \eta \right) \mathbb{E}\left[ \exp(\eta Z_1(s)) \right] ds \\
        & \qquad + R_{11} \eta \mathbb{E}\left[ \int_0^t  \exp(\eta Z_1(s))  dY_1(s)\right]+ \sum_{j=2}^J R_{1j} \eta \mathbb{E}\left[ \int_0^t  \exp(\eta Z_1(s))  dY_j(s)\right].
    \end{align*}
    Dividing both sides by $t$ and using $0<\mathbb E e^{\eta Z_1(t)}\le1$, we obtain
    \begin{align}
        &\frac{1}{t}\int_0^t \left(\frac{1}{2}\Gamma_{11}\eta^2  - R_{11} \eta \right) \mathbb{E}\left[ \exp(\eta Z_1(s)) \right] ds + R_{11} \eta \frac{1}{t}\mathbb{E}\left[ \int_0^t  \exp(\eta Z_1(s))  dY_1(s)\right] \nonumber\\
        & \qquad + \sum_{j=2}^J R_{1j} \eta \frac{1}{t} \mathbb{E}\left[ \int_0^t  \exp(\eta Z_1(s))  dY_j(s)\right] \rightarrow 0 \quad \text{ as } t\to\infty. \label{eq: Ito1} 
    \end{align}

    Recall  {from} Lemma~\ref{lem:fluid} that 
    \begin{equation}
         {\lim_{t\to\infty} {Y_1(t)}/{t} \aseq 1 \quad \text{and} \quad \lim_{t\to\infty} {Y_j(t)}/{t} \aseq 0 \quad \text{for } j=2,\dots,J.} \label{eq: fluid-limit2}
    \end{equation}
    The Lipschitz continuity of the Skorokhod map in \eqref{eq: Lipschitz} shows that
    \[
      \|Y(t)\|\le {\kappa(R)}
      \Big({\|R\e{1}\|}t+\|L\|\sup_{0\le s\le t}\|W(s)\|\Big).
    \]
    {Fix $p>1$.} By Burkholder--Davis--Gundy (BDG) inequality in Theorem 3.28 of \citet{karatzas2014brownian}, $\mathbb E\![\sup_{0\leq s\leq t}\|W(s)\|^p]\le C_p\, t^{p/2}$. Hence, for each $j=1,\dots,J$, $\sup_{t\ge1}\mathbb E[(Y_j(t)/t)^p]<\infty$,
    so $\{Y_j(t)/t:t\ge1\}$ is uniformly integrable.  {This together with \eqref{eq: fluid-limit2} yields}
    \[
       {\lim_{t\to\infty} \frac{1}{t}\,\mathbb E Y_1(t) = 1 \quad \text{and} \quad \lim_{t\to\infty} \frac{1}{t}\,\mathbb E Y_j(t) = 0 \quad \text{for } j=2,\dots,J.}
    \]
    Since $Y_1$ increases only when $Z_1=0$, we have $\int_0^t e^{\eta Z_1(s)}\,dY_1(s)=Y_1(t)$. 
    As $e^{\eta Z_1}\le1$ for $\eta<0$, we also have $0\le \frac1t\mathbb E\!\int_0^t e^{\eta Z_1}\,dY_j \le \mathbb E[Y_j(t)/t]$  {for $j=2,\dots,J$}.  {Therefore,}
     {\[
      \lim_{t\to\infty} \frac{1}{t}\mathbb{E}\!\left[ \int_0^t  \exp(\eta Z_1(s))\, dY_1(s) \right] = 1\quad\text{and}\quad
      \lim_{t\to\infty} \frac{1}{t}\mathbb{E}\!\left[ \int_0^t  \exp(\eta Z_1(s))\, dY_j(s) \right] = 0
    \]
    for $j=2,\dots,J$.}

    By \eqref{eq: Ito1}, we have
    \[
         {\lim_{t\to\infty} }\frac{1}{t}\int_0^t \left(\frac{1}{2}\Gamma_{11}\eta^2  - R_{11} \eta \right) \mathbb{E}\left[ \exp(\eta Z_1(s)) \right]\, ds + R_{11} \eta = 0.
    \]
    {Equivalently,}
    \begin{equation}
        {\lim_{t\to\infty}\frac1t\int_0^t
        \mathbb E e^{\eta Z_1(s)}\,ds}
        =\frac{1}{1-\frac{\Gamma_{11}}{2R_{11}}\eta}.
        \label{eq: zero-Cesaro-LT}
    \end{equation}

    {We next {use the time-average limit} in
    \eqref{eq: zero-Cesaro-LT} {to identify the limit of $\mathbb E e^{\eta Z_1^0(t)}$ as $t\to\infty$}. Since $X(0)=0$,
    $X$ has stationary increments, and a nonsingular $\caM$-matrix can
    be reduced to the form $I-P^\T$ by standard positive diagonal
    rescaling, Theorem~7 of \citet{KellaWhitt1996} implies that $Z^0$
    is stochastically increasing, meaning that
    $\mathbb E f(Z^0(t))\leq \mathbb E f(Z^0(t+s))$ for every bounded
    coordinatewise nondecreasing function $f$ and all $s,t\geq0$.
    Therefore, for every $\eta<0$,
    $h_\eta(t):=\mathbb{E}e^{\eta Z_1^0(t)}$ is nonincreasing. As a bounded
    monotone function and its {time average} have the same limit, it 
    follows from \eqref{eq: zero-Cesaro-LT} that
    }
    \[
      {
      \lim_{t\to \infty}\mathbb{E}e^{\eta Z_1^0(t)}
      =
      \frac{1}{1-\frac{\Gamma_{11}}{2R_{11}}\eta},
      \qquad  \eta<0.
      }
    \]
    The right-hand side is the Laplace transform of an exponential
    distribution with mean $\Gamma_{11}/(2R_{11})$.
    {By the continuity theorem for Laplace transforms,
    $Z_1^0(t)$ therefore converges in distribution, as $t\to\infty$, to 
    an exponential random variable with this mean.}

    {It remains to remove the zero-initial-state restriction.
    The law of the iterated logarithm gives}
    \[
      {
      \frac1t\bigl(R^{-1}X(t)\bigr)_1
      =-1+\frac1t\bigl(R^{-1}LW(t)\bigr)_1
      \longrightarrow-1
      \qquad\text{a.s. as }t\to\infty.
      }
    \]
    {Consequently,
    $\lim_{t\to\infty}(R^{-1}X(t))_1=-\infty$ almost surely.
    Theorem~2.2 of \citet{KellaRamasubramanian2012}, which permits an
    almost surely finite random initial state even when it depends on
    the driving process, therefore yields}
    \[
      { 
      Z_1^z(t)-Z_1^0(t)\longrightarrow0
      \qquad\text{a.s. as }t\to\infty.
      }
    \]
    {The desired conclusion now follows from {Theorem~3.1 in Chapter~1 of \citet{Bill2013}}.}
\end{proof}

{
\subsection{Proof of Theorem~\ref{prop: multi-SRBM conventional}}  
}

To prove Theorem~\ref{prop: multi-SRBM conventional}, we first introduce the following proposition similar to Proposition~\ref{prop: rk initial}, which states the process-level convergence at each individual scaling.
\begin{proposition} \label{prop: r initial}
       Suppose that Assumptions \ref{assmpt: multiscaling} and \ref{assmpt: conventional SRBM} hold. {For each $k\in\J$, let $\tilde Z^{(*,k)}$ be the $(J-k+1)$-dimensional SRBM with initial state $\tilde\xi^{(k)}$, drift vector $-G^{(k)}_{[k:J],k}$, covariance matrix $(E^{(k)} \Gamma (E^{(k)})')_{[k:J],[k:J]}$, and reflection matrix $G^{(k)}_{[k:J], [k:J]}$.} Then as $r\to 0$,
    \begin{equation*}
        \scale{\tilde Z}{k}{[1:k-1]} \Longrightarrow \mathbf{0} \quad \text{and} \quad \scale{\tilde Z}{k}{[k:J]} \Longrightarrow {\tilde Z^{(*,k)}}.
    \end{equation*}
\end{proposition}

\begin{proof}[Proof of Proposition \ref{prop: r initial}.]
    The proof also utilizes the Skorokhod representation theorem similar to Proposition \ref{prop: rk initial} in Section \ref{sec: proof multi-SRBM each}. Specifically, by the weak convergence of the initial states in Assumption~\ref{assmpt: conventional SRBM} and Brownian motions in Corollary~\ref{cor: BM}, it is sufficient to prove that there {exist} random variables $\scale{\hat Z}{k}{{}}(0)\overset{d}{=}\scale{\tilde Z}{k}{{}}(0)$ and $\hat{\xi}\overset{d}{=} \tilde{\xi}$ with $\scale{\hat Z}{1}{{j}}(0)\overset{a.s.}{\to} \hat{\xi}_j$ for $j\in \J$ as $r\to 0$ and stochastic processes $\scale{\hat W}{k}{{}}\overset{d}{=}\scale{W}{k}{{}}$ and $\scaleL{\hat W}{k}{{}} \overset{d}{=} \scaleL{ W}{k}{{}}$ with $\scale{\hat W}{k}{{}}\overset{a.s.}{\to} \scaleL{\hat W}{k}{{}}$ u.o.c.~as $r\to 0$.
    {For each $k\in\J$, let $\hat\xi^{(k)}$ be the $(J-k+1)$-dimensional vector defined as $\hat\xi$ for $k=1$ and $\mathbf{0}$ for $k>1$, and define $\hat Z^{(*,k)}:=\Phi(\scaleL{\hat S}{k}{};G^{(k)}_{[k:J],[k:J]})$, where $\scaleL{\hat S}{k}{}$ is defined in \eqref{eq: limit X}.}
    {Thus, $\hat Z^{(*,k)}$ is a $(J-k+1)$-dimensional SRBM with initial state $\hat\xi^{(k)}$, drift vector $-G^{(k)}_{[k:J],k}$, covariance matrix $(E^{(k)}\Gamma(E^{(k)})^\T)_{[k:J],[k:J]}$, and reflection matrix $G^{(k)}_{[k:J],[k:J]}$. It remains to prove that}
    \begin{equation} \label{eq: suff prop3B}
        \scale{\hat Z}{k}{{[1:k-1]}} \overset{a.s.}{\to} \mathbf{0} \quad  \text{and} \quad \scale{\hat Z}{k}{[k:J]} \overset{a.s.}{\to} {\hat Z^{(*,k)}}
    \end{equation}
    u.o.c.~as $r\to 0$, where 
    \begin{equation*}
        \scale{\hat Z}{k}{} \defi \Phi\left(\scale{\hat X}{k}{{}}; R\right) \quad \text{and} \quad \scale{\hat X}{k}{{}}(t) = \scale{\hat Z}{k}{{}}(0) - R \delta^\uu t / \gamma_k(r) + {L}\scale{\hat W}{k}{{}}(t).
    \end{equation*}

    The proof of the first part of \eqref{eq: suff prop3B} is the same as that of Proposition~\ref{prop: rk initial} by utilizing the mathematical induction. So we omit the details here. 

    To prove the second part of \eqref{eq: suff prop3B}, we note that \eqref{eq: RBM2 r}-\eqref{eq: RBM4 r} and \eqref{eq: RBM1 rk k} uniquely define $\scale{\hat Z}{k}{[k:J]}$ as the Skorokhod reflection mapping $\Phi$ with reflection matrix $G^{(k)}_{[k:J],[k:J]}$ of the $J+1-k$-dimensional driving process $\scale{\hat S}{k}{{}}$, where $\scale{\hat S}{k}{{}}$ is the process inside the square bracket of \eqref{eq: RBM1 rk k}:
    \begin{equation} \label{eq: limit XrB}
        \begin{aligned}
            \scale{\hat S}{k}{{}}(t) &= -E^{(k)}_{[k:J],[1:k-1]}\scale{\hat Z}{k}{[1:k-1]}(t) + E^{(k)}\scale{\hat Z}{k}{{}}(0)- G^{(k)}_{[k:J],[k:J]}\delta_{[k:J]}^\uu t/\gamma_k(r) \\
            &\quad  + E^{(k)}_{[k:J], [1:J]}L \scale{\hat W}{k}{{}}(t).
        \end{aligned}
    \end{equation}
    
    Since the lowest-rate initial condition in Assumption \ref{assmpt: conventional SRBM} indicates $\hat Z^{(r,1)}{(0)} \asto \hat\xi$ and $\hat Z^{(r,k)}{(0)} \asto \mathbf{0}$ for $k\geq 2$ as $r\to 0$, we have $E^{(k)}_{[k:J],[1:J]} \scale{\hat Z}{k}{}(0) {\overset{a.s.}{\longrightarrow}} \hat\xi^{(k)}$ as $r\to 0$, {with $\hat\xi^{(k)}$ as defined in the paragraph preceding \eqref{eq: suff prop3B}}. Hence, $\scale{\hat Z}{k}{[1:k-1]} \asto \mathbf{0}$ and Corollary \ref{cor: BM} imply that as $r\to 0$,
    \begin{equation}   \label{eq: limit X}
        \scale{\hat S}{k}{{}}(t) {\overset{a.s.}{\longrightarrow}} \scaleL{\hat S}{k}{{}}(t) \defi \hat\xi^{(k)} - G^{(k)}_{[k:J],k} t + E^{(k)}_{[k:J], [1:J]}L \scaleL{\hat W}{k}{{}}(t), \quad {\text{u.o.c.}},
    \end{equation}
    which is a Brownian motion starting from the initial state $\hat\xi^{(k)}$ with drift $-G^{(k)}_{[k:J],k}$ and covariance matrix $E^{(k)}_{[k:J], [1:J]}L(E^{(k)}_{[k:J], [1:J]}L)^\T = (E^{(k)}\Gamma (E^{(k)})^\T)_{[k:J],[k:J]}$. Therefore, the proof is completed by the Lipschitz continuity of the Skorokhod reflection mapping.               
\end{proof}

\begin{proof}[Proof of Theorem \ref{prop: multi-SRBM conventional}.]
    The proof is similar to that of Theorem \ref{prop: multi-SRBM} in Section \ref{sec: proof multi-SRBM T}. We provide it here for completeness.

    It follows from {the} Skorokhod representation theorem in the proof of Proposition \ref{prop: r initial} that
    \begin{equation*}
        \begin{aligned}
            &\left\{ \left( \gamma_1(r)\tilde Z_1^\uu(t/\gamma_1^2(r)),  \ldots, \gamma_J(r)\tilde Z_J^\uu(t/\gamma_J^2(r)) \right);~t\geq 0 \right\} = \left( \scale{\tilde Z}{1}{1}, \ldots, \scale{\tilde Z}{J}{J} \right) \\
            &\quad\overset{d}{=} \left( \scale{\hat Z}{1}{1}, \ldots, \scale{\hat Z}{J}{J} \right) =\left( \Phi_1(\scale{\hat S}{1}{}), \ldots, \Phi_1(\scale{\hat S}{J}{}) \right) \asto \left( \Phi_1(\scaleL{\hat S}{1}{}), \ldots, \Phi_1(\scaleL{\hat S}{J}{}) \right)
        \end{aligned}
    \end{equation*}
as $r\to 0$,
    where $\Phi_1$ is the first coordinate process of the Skorokhod reflection mapping, and for each $k\in \J$, both $\scale{\hat S}{k}{}$ and $\scaleL{\hat S}{k}{}$ are defined in the square bracket of \eqref{eq: limit XrB} and in \eqref{eq: limit X}, respectively.

    Crucially, Corollary~\ref{cor: BM} implies that the components of the limit process, $\{\scaleL{\hat S}{k}{}\}_{k \in \mathcal{J}}$, are mutually independent.
    By the Skorokhod representation theorem,
    we have
    \begin{equation*}
        \big(
            \scale{\tilde Z}{1}{1}, \scale{\tilde Z}{2}{2}, \ldots, \scale{\tilde Z}{J}{J}
        \big)  \Longrightarrow \big(
            \tilde Z_1^*, \tilde Z_2^*, \ldots, \tilde Z_J^*
        \big) \quad \text{ as $r\to 0$},
    \end{equation*}
    and all the coordinate processes of $\tilde Z^*$ are mutually independent. This completes the proof.       
\end{proof}

\section{Proof of Convergence in Blockwise multi-scale Heavy Traffic} \label{sec: block heavy traffic proof}

\subsection{SRBMs in blockwise multi-scale heavy traffic}\label{sec: block SRBM}
We similarly consider a family of SRBMs indexed by $r\in (0,1)$, and denoted by $\tilde Z^{(r)}=\{\tilde Z^{(r)}(t);~t\geq 0\}$. {For each $r\in(0,1)$, the $r$th SRBM $\tilde Z^{(r)}$ has initial state $\tilde Z^{(r)}(0)$, drift vector $\theta^\uu$, covariance matrix $\Gamma$, and reflection matrix $R$.} Consistent with Assumption \ref{assmpt: blockscale}, we introduce the {blockwise multi-scaling} regime for SRBMs as follows.

\begin{assumption}[Blockwise multi-scaling regime for SRBMs] \label{assmpt: blockscaling}
    For all $r\in (0,1)$, 
    \begin{equation*}        \theta^\uu= -R\delta^\uu, \quad \text{where } \delta^\uu_{A_k} = \gamma_k(r) b^{(k)} \text{ for $k\in \K$,}
    \end{equation*}
    where $b^{(k)}$ is a positive vector in $\R^{|A_k|}$.
\end{assumption}

In exact analogy to Assumption \ref{assmpt: initial block}, we consider the following initial condition for the SRBMs.
\begin{assumption}[Matching-rate initial condition for SRBMs]
    \label{assmpt: initial SRBM block}
There exists an $\R_+^J$-valued   random vector $\tilde \xi$ {{whose block subvectors $\tilde\xi_{A_1},\ldots,\tilde\xi_{A_K}$ are mutually independent}} such that
    \begin{equation*}
        \left( \gamma_1(r)\tilde Z^\uu_{A_1}(0),~\gamma_2(r) \tilde Z^\uu_{A_2}(0),~\ldots,~\gamma_K(r) \tilde Z^\uu_{A_K}(0) \right)  \Longrightarrow \tilde \xi \quad \text{as } r\to 0,
    \end{equation*}
    and for all $k=2,\ldots,K$,
    \begin{equation*}
        \gamma_{k-1}(r) \tilde Z_{A_k}^\uu (0) \Longrightarrow \infty \quad \text{as } r\to 0.
    \end{equation*}
\end{assumption}

\begin{theorem}\label{thm:SRBM block}
    Suppose Assumptions \ref{assmpt: blockscaling} and \ref{assmpt: initial SRBM block} hold. Then
    \begin{equation*}         \left\{ \left( \gamma_1(r)\tilde Z_{A_1}^\uu(t/\gamma_1^2(r)), \gamma_2(r)\tilde Z_{A_2}^\uu(t/\gamma_2^2(r)), \ldots, \gamma_K(r)\tilde Z_{A_K}^\uu(t/\gamma_K^2(r)) \right);~t\geq 0 \right\} \Longrightarrow \tilde Z^* \quad \text{as } r\to 0,
    \end{equation*}
    where $\tilde Z^*=\{\tilde Z^*(t);~t\geq 0\}$ is a $J$-dimensional diffusion process.
    Furthermore, all corresponding blocks of $\tilde Z^*$ are mutually independent, and each block $\tilde Z^*_{A_k}$ for $k\in \K$ is a $|A_k|$-dimensional SRBM {with initial state $\tilde\xi_{A_k}$, drift vector $-G^{(\indexa{k})}_{A_k,A_k}b^{(k)}$, covariance matrix $(E^{(\indexa{k})} \Gamma (E^{(\indexa{k})})')_{A_k,A_k}$, and reflection matrix $G^{(\indexa{k})}_{A_k,A_k}$.}
    {Here,} $\indexa{k}$ is the lowest index in $A_k$, and matrices $G$ and $E$ are defined in \eqref{eq: Gauss elim}.
\end{theorem}

The following gives the functional limit result for block multi-scaling SRBMs under the lowest-rate initial condition in Assumption \ref{assmpt: conventional SRBM}.

\begin{theorem}\label{thm:SRBM block cv}
    Suppose Assumptions \ref{assmpt: conventional SRBM} and \ref{assmpt: blockscaling}  hold. Then
    \begin{equation*}         \left\{ \left( \gamma_1(r)\tilde Z_{A_1}^\uu(t/\gamma_1^2(r)), \gamma_2(r)\tilde Z_{A_2}^\uu(t/\gamma_2^2(r)), \ldots, \gamma_K(r)\tilde Z_{A_K}^\uu(t/\gamma_K^2(r)) \right);~t\geq 0 \right\} \Longrightarrow \tilde Z^* \ \text{as } r\to 0,
    \end{equation*}
    where $\tilde Z^*=\{\tilde Z^*(t);~t\geq 0\}$ is a $J$-dimensional process.
    Furthermore, all corresponding blocks of $\tilde Z^*$ are mutually independent, and each block $\tilde Z^*_{A_k}$ for $k\in \K$ is the first $|A_k|$ coordinates of a $(J-\alpha_k+1)$-dimensional SRBM {with initial state $\tilde\xi^{(k)}$, drift vector $-G^{(\indexa{k})}_{[\indexa{k},J],A_k}b^{(k)}$, covariance matrix $(E^{(\indexa{k})} \Gamma (E^{(\indexa{k})})')_{[\indexa{k},J],[\indexa{k},J]}$, and reflection matrix $G^{(\indexa{k})}_{[\indexa{k},J],[\indexa{k},J]}$.}
    {Here,} $\indexa{k}$ is the lowest index in $A_k$, {for each $k\in\K$, $\tilde\xi^{(k)}$ is the $(J-\indexa{k}+1)$-dimensional vector defined as $\tilde\xi$ for $k=1$ and $\mathbf{0}$ for $k>1$}, and matrices $G$ and $E$ are defined in \eqref{eq: Gauss elim}.
\end{theorem}

\subsection{Convergence of GJNs in blockwise multi-scale heavy traffic}

The proofs of {Theorems}~\ref{thm:GJN block} and \ref{thm:GJN block cv} are almost the same as {those} of {Theorems} \ref{thm:GJN} and \ref{thm:GJN cv} by using the asymptotic functional strong approximation of GJNs in Proposition \ref{prop: SA}, so we omit {them} here and only prove the limit behavior of SRBMs in {Theorems}~\ref{thm:SRBM block} and \ref{thm:SRBM block cv}. {Proposition~\ref{prop: SA} applies at each block scale $\gamma_k$, $k\in\K$, since it approximates the full $J$-dimensional process at every observation scale.} 
    
\subsection{Convergence of SRBMs in blockwise multi-scale heavy traffic}

To prove Theorem \ref{thm:SRBM block}, we first introduce the following proposition, which states the process-level convergence {results for} SRBMs at each individual scaling.

\begin{proposition}\label{prop: block rk initial}
    Suppose that Assumptions \ref{assmpt: blockscaling} and \ref{assmpt: initial SRBM block} hold. Then  as $r\to 0$, 
    \begin{equation*}
        \scale{\tilde Z}{k}{[1:\indexb{k-1}]} \Longrightarrow \mathbf{0}, \quad \scale{\tilde Y}{k}{[\indexa{k+1}:J]} \Longrightarrow \mathbf{0}, \quad \scale{\tilde Z}{k}{A_k}  \Longrightarrow \tilde Z^*_{A_k} , 
    \end{equation*}
    where $\tilde Z^*$ is the limit process defined in Theorem \ref{thm:SRBM block}.
\end{proposition}

To prove Proposition \ref{prop: block rk initial}, we need the following lemma, which is an extension of Lemma~\ref{lemma: Skorokhod negative infty}.

\begin{lemma} \label{lemma: block Skorokhod negative infty}
    For any fixed integer $d$, consider three sequences of functions $u^\uu,v^\uu\in \C([0,\infty),\R^d)$ and vector $a^\uu\in \R^d$ for $r\in (0,1)$. Suppose the following conditions hold:
    \begin{enumerate}  
        \item[\rm (i)] {$u^\uu(0)\in\R_+^d$ for all $r\in(0,1)$, $u^\uu(0)\to\mathbf{0}$, and} $u^\uu \rightarrow u$ u.o.c.~as $r\to 0$;
        \item[\rm (ii)]	 $v^\uu(0)=\mathbf{0}$ and $v^\uu$ is element-wise  nondecreasing for each $r\in (0,1)$;
        \item[\rm (iii)] {$a^\uu\to\infty$ as $r\to0$.}
    \end{enumerate}
    Define
    $$
    z^\uu=\Phi\left(\left\{ u^\uu(t)-v^\uu(t)-Ra^\uu t \right\};R\right),
    $$
    where the reflection matrix $R\in \R^{d\times d}$ is an $\mathcal{M}$-matrix.
    Then  $z^\uu$ converges to $\mathbf{0}$ u.o.c.~as $r\to 0$.
\end{lemma}

\begin{proof}[Proof of Lemma \ref{lemma: block Skorokhod negative infty}.]
    {
    Fix $T>0$ and $\varepsilon>0$.  Since $u^\uu\to u$ uniformly on $[0,T]$ and $u(0)=\mathbf{0}$, there exists a Lipschitz continuous function $\bar u:[0,T]\to\R^d$ such that
    \begin{equation*}
        \bar u(0)=\mathbf{0},\qquad \norm{u-\bar u}_T<\varepsilon.
    \end{equation*}
    Let
    \begin{equation*}
        L_T:=\max_{1\leq i\leq d}\operatorname{Lip}\big((R^{-1}\bar u)_i;[0,T]\big)<\infty.
    \end{equation*}
    For all sufficiently small $r$,
    \begin{equation*}
        \norm{u^\uu-\bar u}_T<2\varepsilon
        \qquad\text{and}\qquad
        \min_{1\leq i\leq d}a_i^\uu\geq L_T.
    \end{equation*}
    Define, for $t\in[0,T]$,
    \begin{equation*}
        \bar y^\uu(t):=a^\uu t+R^{-1}v^\uu(t)-R^{-1}\bar u(t).
    \end{equation*}
    Since $R$ is a nonsingular $\mathcal{M}$-matrix, $R^{-1}\geq0$. Moreover, $\bar y^\uu(0)=\mathbf{0}$, and, for $0\leq s\leq t\leq T$ and $1\leq i\leq d$,
    \begin{equation*}
        \begin{aligned}
            \bar y^\uu_i(t)-\bar y^\uu_i(s) 
            &=
            a_i^\uu(t-s)
            +\big(R^{-1}(v^\uu(t)-v^\uu(s))\big)_i \\
            &\quad
            -\big(R^{-1}(\bar u(t)-\bar u(s))\big)_i\\
            &\geq (a_i^\uu-L_T)(t-s)\geq0.
        \end{aligned}
    \end{equation*}
    Thus, $\bar y^\uu$ is element-wise nondecreasing. In addition,
    \begin{equation*}
        \bar u(t)-v^\uu(t)-Ra^\uu t+R\bar y^\uu(t)=\mathbf{0},
        \qquad 0\leq t\leq T.
    \end{equation*}
    Therefore, $(\mathbf{0},\bar y^\uu)$ solves the Skorokhod problem on $[0,T]$ for the input $\bar u-v^\uu-Ra^\uu t$, and hence
    \begin{equation*}
        \Phi\left(\left\{\bar u(t)-v^\uu(t)-Ra^\uu t\right\};R\right)=\mathbf{0}
        \qquad\text{on }[0,T].
    \end{equation*}
    By the Lipschitz continuity of the Skorokhod reflection mapping in \eqref{eq: Lipschitz},
    \begin{equation*}
        \begin{aligned}
            \norm{z^\uu}_T
            &\leq \kappa(R)\norm{u^\uu-\bar u}_T
            \leq 2\kappa(R)\varepsilon
        \end{aligned}
    \end{equation*}
    for all sufficiently small $r$. Since $\varepsilon>0$ is arbitrary, $\norm{z^\uu}_T\to0$. As $T>0$ was arbitrary, $z^\uu\to\mathbf{0}$ u.o.c.
    }
\end{proof}

\begin{proof}[Proof of Proposition \ref{prop: block rk initial}.]
    The proof is similar to that of Proposition \ref{prop: rk initial} by changing each component to each block.

    We also utilize the Skorokhod representation theorem similar to Proposition \ref{prop: rk initial} in Section \ref{sec: proof multi-SRBM each}. Specifically, by the weak convergence of the initial states in Assumption {\ref{assmpt: initial SRBM block}} and Brownian motions in Corollary~\ref{cor: BM}, it is sufficient to prove that there {exist} random variables $\scale{\hat Z}{k}{{}}(0)\overset{d}{=}\scale{\tilde Z}{k}{{}}(0)$ and $\hat{\xi}\overset{d}{=} \tilde{\xi}$ with $\scale{\hat Z}{j}{{A_j}}(0)\overset{a.s.}{\to} \hat{\xi}_{A_j}$ for $j\in \K$ as $r\to 0$ and stochastic processes $\scale{\hat W}{k}{{}}\overset{d}{=}\scale{W}{k}{{}}$ and $\scaleL{\hat W}{k}{{}} \overset{d}{=} \scaleL{W}{k}{{}}$ with $\scale{\hat W}{k}{{}}\overset{a.s.}{\to} \scaleL{\hat W}{k}{{}}$ u.o.c.~as $r\to 0$ such that
    \begin{equation} \label{eq: suff prop3C}
        \scale{\hat Z}{k}{[1:\indexb{k-1}]} \asto \mathbf{0}, \quad \scale{\hat Y}{k}{[\indexa{k+1}:J]} \asto \mathbf{0}, \quad \scale{\hat Z}{k}{A_k}  \asto \hat Z^*_{A_k} , \quad \text{u.o.c.~as } r\to 0, 
    \end{equation}
    where 
\begin{equation*}
    \scale{\hat Z}{k}{} \defi \Phi\left(\scale{\hat X}{k}{{}}; R\right), \quad \scale{\hat Y}{k}{} \defi \Psi\left(\scale{\hat X}{k}{{}}; R\right), \quad \hat Z^*_{A_k} \overset{d}{=} \tilde Z^*_{A_k},
\end{equation*}
and 
\begin{equation*}
    \scale{\hat X}{k}{{}} = \scale{\hat Z}{k}{{}}(0) - R \delta^\uu t / \gamma_k(r) + {L}\scale{\hat W}{k}{{}}(t).
\end{equation*}

We are now ready to prove the first statement of \eqref{eq: suff prop3C}. The first part, $\scale{\hat Z}{k}{[1:\indexb{k-1}]} \asto \mathbf{0}$ u.o.c.~as $r\to 0$, is vacuously true for $k=1$. For $k\geq 2$, we use the mathematical induction to prove $\scale{\hat Z}{k}{A_\ell} \asto \mathbf{0}$ u.o.c.~as $r\to 0$ from $\ell=1$ to $k-1$.

    For the base case of block $\ell=1$, the components of the first block of \eqref{eq: RBM1 r} {are}
    \begin{equation*}
        \scale{\hat Z}{k}{A_1}(t) = \scale{\hat Z}{k}{A_1}(0) - R_{A_1,[1:J]}\delta^\uu t / \gamma_k(r) + L_{A_1,[1:J]} \scale{\hat W}{k}{{}}(t) + R_{A_1,[1:J]} \scale{\hat Y}{k}{{}}(t),
    \end{equation*}
    which is equivalent to, for $t\geq 0$,
    \begin{equation}\label{eq: block RBM1 r 1}
        \begin{aligned} 
            \scale{\hat Z}{k}{A_1}(t)&= \left[ \scale{\hat Z}{k}{A_1}(0) - \left( R_{A_1,A_1} b^{(1)} \frac{\gamma_1(r)}{\gamma_k(r)} + \sum_{j=2}^K R_{A_1,A_j}b^{(j)}\frac{\gamma_j(r)}{\gamma_k(r)} \right)t +  L_{A_1,[1:J]} \scale{\hat W}{k}{{}}(t) \right.  \\
            &\quad \left. + \sum_{j=2}^K R_{A_1,A_j}\scale{\hat Y}{k}{A_j}(t) \right] + R_{A_1,A_1}\scale{\hat Y}{k}{A_1}(t). 
        \end{aligned}
    \end{equation}
    Therefore, \eqref{eq: RBM2 r}-\eqref{eq: RBM4 r} and \eqref{eq: block RBM1 r 1} imply $\scale{\hat Z}{k}{A_1} = \Phi (\scale{\hat U}{k}{1}; R_{A_1,A_1})$, where $\scale{\hat U}{k}{1}(t)$ is a $|A_1|$-dimensional process, defined in the square bracket of \eqref{eq: block RBM1 r 1}. 
    
    The matching-rate initial condition in Assumption {\ref{assmpt: initial SRBM block}} indicates $\scale{\hat Z}{k}{A_1}(0) \asto \mathbf{0}$ as $r\to 0$ for $k\geq 2$. Since $R$ is an $\mathcal{M}$-matrix, $R_{A_1,A_1}$ is also an $\caM$-matrix and $R_{A_1,A_j} \leq 0$ for $j>1$, {the process $-\sum_{j=2}^K R_{A_1,A_j}\scale{\hat Y}{k}{A_j}$ is element-wise nondecreasing, and} we have $b^{(1)} \gamma_1(r)/\gamma_k(r) + \sum_{j=2}^K R_{A_1,A_1}^{-1}R_{A_1,A_j}b^{(j)}\gamma_j(r)/\gamma_k(r) \to \infty$ {componentwise} as $r\to 0$. By Corollary~\ref{cor: BM}, {under the Skorokhod representation,} the driving Brownian motion converges {u.o.c.~almost surely}. Hence, we can apply Lemma~{\ref{lemma: block Skorokhod negative infty}} to conclude that $\scale{\hat Z}{k}{A_1} \asto 0$ u.o.c.~as $r\to 0$.

    For the induction step $2\leq \ell< k$, we assume that $\scale{\hat Z}{k}{i} \asto {0}$ u.o.c.~as $r\to 0$ for $i=1,\ldots,\indexb{\ell-1}$. Left-multiplying $E^{(\indexa{\ell})}$ on both sides of \eqref{eq: RBM1 r}, for $t\geq 0$, we have
    \begin{equation*}
        E^{(\indexa{\ell})}\scale{\hat Z}{k}{{}}(t) = E^{(\indexa{\ell})}\scale{\hat Z}{k}{{}}(0) - G^{(\indexa{\ell})}\delta^\uu t / \gamma_k(r) + E^{(\indexa{\ell})}L \scale{\hat W}{k}{{}}(t) + G^{(\indexa{\ell})}\scale{\hat Y}{k}{{}}(t).
    \end{equation*}
    Since $E^{(\indexa{\ell})}_{A_{\ell}, A_{\ell}} = I_{|A_\ell|}$, $E^{(\indexa{\ell})}_{A_{\ell}, [\indexa{\ell+1}:J]} = 0$ and $G^{(\indexa{\ell})}_{A_{\ell}, [1:\indexb{\ell-1}]} = 0$  in \eqref{eq: Gauss elim}, the components $A_{\ell}$ of the above equation become
    \begin{equation*}
        \begin{aligned}
            E^{(\indexa{\ell})}_{A_{\ell}, [1:\indexb{\ell-1}]}\scale{\hat Z}{k}{{[1:\indexb{\ell-1}]}}(t) + \scale{\hat Z}{k}{{A_{\ell}}}(t) &= E^{(\indexa{\ell})}_{A_{\ell}, [1:\indexb{\ell}]}\scale{\hat Z}{k}{{[1:\indexb{\ell}]}}(0) - G^{(\indexa{\ell})}_{A_{\ell}, [\indexa{\ell}:J]}\delta^\uu_{[\indexa{\ell}:J]} t / \gamma_k(r) \\
            &\quad + E^{(\indexa{\ell})}_{A_{\ell}, [1:J]}L \scale{\hat W}{k}{{}}(t) + G^{(\indexa{\ell})}_{A_{\ell}, [\indexa{\ell}:J]}\scale{\hat Y}{k}{{[\indexa{\ell}:J]}}(t),
        \end{aligned}
    \end{equation*}
    which is equivalent to
    \begin{equation} \label{eq: block RBM1 r ell}
        \begin{aligned}
            \scale{\hat Z}{k}{{A_{\ell}}}(t) & = \Big[ - E^{(\indexa{\ell})}_{A_{\ell}, [1:\indexb{\ell-1}]}\scale{\hat Z}{k}{{[1:\indexb{\ell-1}]}}(t) + E^{(\indexa{\ell})}_{A_{\ell}, [1:\indexb{\ell}]}\scale{\hat Z}{k}{{[1:\indexb{\ell}]}}(0)  + E^{(\indexa{\ell})}_{A_{\ell}, [1:J]}L \scale{\hat W}{k}{{}}(t) \\
            &\qquad  - \Big( G^{(\indexa{\ell})}_{A_{\ell}, A_{\ell}} b^{(\ell)} \frac{\gamma_\ell(r)}{\gamma_k(r)} + \sum_{j=\ell+1}^K G^{(\indexa{\ell})}_{A_{\ell}, A_j}b^{(j)}\frac{\gamma_j(r)}{\gamma_k(r)} \Big)t  \\
            &\qquad + \sum_{j=\ell+1}^K G^{(\indexa{\ell})}_{A_{\ell}, A_j} \scale{\hat Y}{k}{{A_j}}(t) \Big] + G^{(\indexa{\ell})}_{A_{\ell}, A_{\ell}}\scale{\hat Y}{k}{{A_{\ell}}}(t), 
        \end{aligned}
    \end{equation}
    Therefore, \eqref{eq: RBM2 r}-\eqref{eq: RBM4 r} and \eqref{eq: block RBM1 r ell} imply $\scale{\hat Z}{k}{A_{\ell}} = \Phi (\scale{\hat U}{k}{\ell}; G_{A_{\ell},A_{\ell}}^{(\indexa{\ell})})$, where $\scale{\hat U}{k}{\ell}$ is a $|A_\ell|$-dimensional process, defined in the square bracket of \eqref{eq: block RBM1 r ell}. 
    
    By the hypothesis induction, we have $\scale{\hat Z}{k}{i} \asto 0$ u.o.c.~as $r\to 0$ for $i=1,\ldots,\indexb{\ell-1}$. The matching-rate initial condition in Assumption {\ref{assmpt: initial SRBM block}} indicates $\scale{\hat Z}{k}{{i}}(0)\asto \mathbf{0}$ as $r\to 0$ for $i=1,\ldots,\indexb{k-1}$ and $k\geq 2$. {Together with the Skorokhod representation, these convergences verify condition {\rm (i)} of Lemma~\ref{lemma: block Skorokhod negative infty} for the corresponding terms in $\scale{\hat U}{k}{\ell}$.} Since $G^{(\indexa{\ell})}_{A_{\ell}, A_{\ell}}$ is an $\mathcal{M}$-matrix, $(G^{(\indexa{\ell})}_{A_{\ell}, A_{\ell}})^{-1}\geq 0$ and $G^{(\indexa{\ell})}_{A_{\ell}, [\indexa{\ell+1}:J]}\leq 0$, {the process $-\sum_{j=\ell+1}^K G^{(\indexa{\ell})}_{A_\ell,A_j}\scale{\hat Y}{k}{A_j}$ is element-wise nondecreasing, and} we have ${b^{(\ell)}}\gamma_\ell(r)/\gamma_k(r) + \sum_{j=\ell+1}^K (G^{(\indexa{\ell})}_{A_{\ell}, A_{\ell}})^{-1}G^{(\indexa{\ell})}_{A_{\ell}, A_j}{b^{(j)}}\gamma_j(r)/\gamma_k(r) \to \infty$ {componentwise} as $r\to 0$. Hence, we can apply Lemma \ref{lemma: block Skorokhod negative infty} and conclude that $\scale{\hat Z}{k}{A_{\ell}} \asto 0$ u.o.c.~as $r\to 0$.

    In summary, by the mathematical induction, we have proved that $\scale{\hat Z}{k}{[1:\indexb{k-1}]} \asto 0$ u.o.c.~as $r\to 0$. 

    \vspace{1em}
    
    To prove the remaining two parts, we left-multiply $E^{(\indexa{k})}$ on both sides of \eqref{eq: RBM1 r}, for $t\geq 0$, we have
    \begin{equation*}
        E^{(\indexa{k})}\scale{\hat Z}{k}{{}}(t) = E^{(\indexa{k})}\scale{\hat Z}{k}{{}}(0) - G^{(\indexa{k})}\delta^\uu t / \gamma_k(r) + E^{(\indexa{k})}L \scale{\hat W}{k}{{}}(t) + G^{(\indexa{k})}\scale{\hat Y}{k}{{}}(t).
    \end{equation*}
    Since $E^{(\indexa{k})}_{[\indexa{k}:J], [\indexa{k}:J]} = I_{J-\indexa{k}+1}$ and $G^{(\indexa{k})}_{[\indexa{k}:J], [1:\indexb{k-1}]}=0$ in \eqref{eq: Gauss elim}, the components $[\indexa{k}:J]$ of the above equation become
    \begin{equation}\label{eq: RBM1 rk k block}
        \begin{aligned} 
            \scale{\hat Z}{k}{[\indexa{k}:J]}(t) &= \left[ -E^{(\indexa{k})}_{[\indexa{k}:J],[1:\indexb{k-1}]}\scale{\hat Z}{k}{[1:\indexb{k-1}]}(t) + E^{(\indexa{k})}\scale{\hat Z}{k}{{}}(0)- G^{(\indexa{k})}_{[\indexa{k}:J],[\indexa{k}:J]}\delta_{[\indexa{k}:J]}^\uu t / \gamma_k(r) \right. \\
            &\quad \left. + E^{(\indexa{k})}_{[\indexa{k}:J], [1:J]}L \scale{\hat W}{k}{{}}(t) \right]  + G^{(\indexa{k})}_{[\indexa{k}:J],[\indexa{k}:J]}\scale{\hat Y}{k}{{[\indexa{k}:J]}}(t). 
        \end{aligned}
    \end{equation}
    Therefore, \eqref{eq: RBM2 r}-\eqref{eq: RBM4 r} and \eqref{eq: RBM1 rk k block} imply $\scale{\hat Y}{k}{[\indexa{k}:J]} = \Psi (\scale{\hat S}{k}{{}}; G^{(\indexa{k})}_{[\indexa{k}:J],[\indexa{k}:J]})$, where $\scale{\hat S}{k}{{}}$ is a $({J}-\indexa{k}+1)$-dimensional process, defined in the square bracket of \eqref{eq: RBM1 rk k block}. 

    To prove the second statement of \eqref{eq: suff prop3C}. We will first show that the family  $\{\hat Y^{(r,k)}_{[\indexa{k}:J]}(t);~t\geq 0\}_{r\in(0,1)}$ is pathwise bounded by a continuous nondecreasing process. 

    To show this, we use a comparison argument. Consider a simple driving process $\scale{\check{S}}{k}{{}}:=E^{(\indexa{k})}\scale{\hat Z}{k}{{}}(0)$. Its associated regulator process is $\scale{\check{Y}}{k}{{}}=\Psi(\scale{\check{S}}{k}{{}}; {G^{(\indexa{k})}}_{[\indexa{k}:J],[\indexa{k}:J]})=\mathbf{0}$. By the Lipschitz continuity of the Skorokhod reflection mapping in \eqref{eq: Lipschitz}, we have, for any fixed $T>0$,
    \begin{equation*}
        \begin{aligned}
            &\sup_{t\in [0,T]}\sup_{\indexa{k}\leq \ell \leq J}{\scale{\hat Y}{k}{\ell}(t) }  = \sup_{t\in [0,T]}\Norm{\scale{\hat Y}{k}{[\indexa{k}:J]}(t) - \scale{\check{Y}}{k}{{}}(t) }  \\
            &\quad \leq \kappa(G^{(\indexa{k})}_{[\indexa{k}:J],[\indexa{k}:J]})\sup_{t\in [0,T]}\Norm{ \scale{\hat S}{k}{{}}(t) - \scale{\check{S}}{k}{{}}(t)}  \\
            &\quad \leq \kappa(G^{(\indexa{k})}_{[\indexa{k}:J],[\indexa{k}:J]})\sup_{t\in [0,T]}\left( \Norm{ \left\{ E^{(\indexa{k})}_{[\indexa{k}:J],[1:\indexb{k-1}]}\scale{\hat Z}{k}{[1:\indexb{k-1}]}(t) \right\} } + \Norm{ \left\{  E^{(\indexa{k})}_{[\indexa{k}:J], [1:J]}L \scale{\hat W}{k}{{}} (t) \right\} }  \right) \\
            &\qquad +  \kappa(G^{(\indexa{k})}_{[\indexa{k}:J],[\indexa{k}:J]})\Norm{ G^{(\indexa{k})}_{[\indexa{k}:J],[\indexa{k}:J]}\delta_{[\indexa{k}:J]}^\uu / \gamma_k(r) }t , \quad \text{a.s.}
        \end{aligned}
    \end{equation*}
{
where the last right-hand side converges u.o.c.~almost surely.  Hence, the
preceding bound provides the majorant required in condition~(ii) of
Lemma~\ref{lemma: Skorokhod positive infty}.
}
    To apply Lemma \ref{lemma: Skorokhod positive infty} to prove $\{\hat Y^{(r,k)}_{\ell}(t);~t\geq 0\}\asto 0$ u.o.c.~as $r\to 0$ for $\ell \geq \indexa{k+1}$, we consider the $\ell$th component of \eqref{eq: RBM1 rk k block}.
    Since $E^{(\indexa{k})}_{\ell, \ell} = 1$ for $\ell\geq \indexa{k}$, $E^{(\indexa{k})}_{\ell, j} = 0$ for $\ell,j\geq \indexa{k}$ and $\ell\neq j$, and $G^{(\indexa{k})}_{\ell, j} = 0$ for $j<\indexa{k}\leq\ell$ in \eqref{eq: Gauss elim}, the $\ell$th component of \eqref{eq: RBM1 rk k block} for $\ell \geq \indexa{k+1}$ becomes
    \begin{equation} \label{eq: RBM1 rk ell block}
        \begin{aligned} 
            &\scale{\hat Z}{k}{\ell}(t) \\
            &\quad = \Big[ -E^{(\indexa{k})}_{\ell,[1:\indexb{k-1}]}\scale{\hat Z}{k}{[1:\indexb{k-1}]}(t) + E^{(\indexa{k})}_{\ell, [1:\indexb{k-1}]}\scale{\hat Z}{k}{{[1:\indexb{k-1}]}}(0)+ \scale{\hat Z}{k}{\ell}(0) - G^{(\indexa{k})}_{\ell,[\indexa{k}:J]}\delta_{[\indexa{k}:J]}^\uu t / \gamma_k(r)   \\ 
            &\qquad + E^{(\indexa{k})}_{\ell, [1:J]}L \scale{\hat W}{k}{{}}(t) + \sum_{j=\indexa{k},j\neq \ell}^J G^{(\indexa{k})}_{\ell,j}\scale{\hat Y}{k}{{j}}(t) \Big]  + G^{(\indexa{k})}_{\ell,\ell}\scale{\hat Y}{k}{{\ell}}(t). 
        \end{aligned}
    \end{equation}
    Therefore, \eqref{eq: RBM2 r}-\eqref{eq: RBM4 r} and \eqref{eq: RBM1 rk ell block} imply that $\scale{\hat Z}{k}{\ell} = \Phi (\scale{\hat U}{k}{\ell}; G^{(\indexa{k})}_{\ell,\ell})$, where $\scale{\hat U}{k}{\ell}$ is a one-dimensional process, which is defined in the square bracket of \eqref{eq: RBM1 rk ell block}. 
    
    By the matching-rate initial condition in Assumption {\ref{assmpt: initial SRBM block}}, we have $\scale{\hat Z}{k}{{[1:\indexb{k-1}]}}(0)\asto \mathbf{0}$ and $\hat Z^\uu_\ell(0)\asto \infty$ as $r\to 0$ for $\ell \geq \indexa{k+1}$. Hence, we can apply Lemma \ref{lemma: Skorokhod positive infty} and conclude that $\scale{\hat Y}{k}{\ell} \asto 0$ u.o.c.~as $r\to 0$ for $\ell \geq \indexa{k+1}$. In summary, $\scale{\hat Y}{k}{[\indexa{k+1}:J]} \asto \mathbf{0}$ u.o.c.~as $r\to 0$.

    The components in $A_k$ of \eqref{eq: RBM1 rk k block} are
    \begin{equation} \label{eq: RBM1 rk ell k block}
        \begin{aligned} 
            \scale{\hat Z}{k}{A_k}(t) &= \left[ -E^{(\indexa{k})}_{A_k,[1:\indexb{k-1}]}\scale{\hat Z}{k}{[1:\indexb{k-1}]}(t) + E^{(\indexa{k})}_{A_k, [1:\indexb{k-1}]}\scale{\hat Z}{k}{{[1:\indexb{k-1}]}}(0) + \scale{\hat Z}{k}{A_k}(0)  \right.  \\ 
            &\quad - G^{(\indexa{k})}_{A_k,[\indexa{k}:J]}\delta_{[\indexa{k}:J]}^\uu t / \gamma_k(r)+ E^{(\indexa{k})}_{A_k, [1:J]}L \scale{\hat W}{k}{{}}(t)  +\left.G^{(\indexa{k})}_{A_k,[\indexa{k+1}:J]}\scale{\hat Y}{k}{{[\indexa{k+1}:J]}}(t) \right]  \\
            &\quad  + G^{(\indexa{k})}_{A_k,A_k}\scale{\hat Y}{k}{{A_k}}(t). 
        \end{aligned}
    \end{equation}
    Therefore, \eqref{eq: RBM2 r}-\eqref{eq: RBM4 r} and \eqref{eq: RBM1 rk ell k block} imply ${\scale{\hat Z}{k}{A_k} = \Phi (\scale{\hat U}{k}{A_k}; G^{(\indexa{k})}_{A_k,A_k})}$, where ${\scale{\hat U}{k}{A_k}}$ is a $|A_k|$-dimensional process, which is defined in the square bracket of \eqref{eq: RBM1 rk ell k block}. 
    
    By Assumption~{\ref{assmpt: initial SRBM block}}, we have $\scale{\hat Z}{k}{{[1:\indexb{k-1}]}}(0)\asto \mathbf{0}$ and $\scale{\hat Z}{k}{A_k}(0)\asto \hat{\xi}_{A_k}$ as $r\to 0$.  Since we have proved that $\scale{\hat Z}{k}{[1:\indexb{k-1}]} \asto \mathbf{0}$ and  $\scale{\hat Y}{k}{[\indexa{k+1}:J]} \asto \mathbf{0}$ as $r\to 0$, we have
    \begin{equation*}
        {\scale{\hat U}{k}{A_k}}(t) \asto {\scaleL{\hat U}{k}{A_k}} \defi \left\{ \hat{\xi}_{A_k} - G^{(\indexa{k})}_{A_k,A_k} b^{(k)} t + E^{(\indexa{k})}_{A_k,[1:J]}L \scaleL{\hat W}{k}{}(t);~t\geq0 \right\}, \quad \text{u.o.c.~as $r\to 0$,}
    \end{equation*}
    which is a Brownian motion with the initial state $\hat{\xi}_{A_k}$, the drift vector $ - G^{(\indexa{k})}_{A_k,A_k} b^{(k)}$ and the covariance matrix $E^{(\indexa{k})}_{A_k,[1:J]}L (E^{(\indexa{k})}_{A_k,[1:J]}L)^\T = (E^{(\indexa{k})}\Gamma(E^{(\indexa{k})})^\T)_{A_k,A_k}$. Thus, the proof is completed by the Lipschitz continuity of the Skorokhod reflection mapping.      
\end{proof}

The proof of Theorem \ref{thm:SRBM block} is the same as the proof of Theorem \ref{prop: multi-SRBM} in Section \ref{sec: proof multi-SRBM A} by applying Corollary \ref{cor: BM}. The proof of Theorem \ref{thm:SRBM block cv} is the same as the proof of Theorem~\ref{prop: multi-SRBM conventional} in Appendix \ref{sec: conventional proof}. So we omit their proofs here.

\end{document}